\documentclass[12pt,reqno]{amsart}

\usepackage[margin=1.3in]{geometry}
\usepackage[utf8]{inputenc}
\usepackage[T1]{fontenc}
\usepackage[greek,english]{babel}
\usepackage{teubner}
\usepackage{comment}
\usepackage{graphicx}
\usepackage{csquotes}
\usepackage{microtype}
\usepackage[skip=.5\baselineskip]{parskip}
\usepackage{anyfontsize}

\usepackage[dvipsnames]{xcolor}
\definecolor{highlight}{cmyk}{90,99,0,0}

\usepackage[textsize=footnotesize]{todonotes}
\setuptodonotes{color=blue!30}
\usepackage{amsmath,amssymb,amsthm,mathtools}
\usepackage{mathrsfs}
\usepackage{euscript}
\usepackage[cal=cm]{mathalfa}
\usepackage{xifthen}
\usepackage{thmtools}
\numberwithin{equation}{section}

\usepackage[lining]{libertine}
\usepackage{courier}
\usepackage[libertine,libaltvw,liby]{newtxmath}
\makeatletter
    \renewcommand*\libertine@figurestyle{LF}
\makeatother

\usepackage{enumitem}
\setlist[enumerate,1]{label=(\roman*),itemsep=0.9ex}
\setlist[itemize]{itemsep=0.9ex}

\usepackage[colorlinks=true,allcolors=highlight,hyperfootnotes=false,linktocpage=true]{hyperref}
\usepackage{url}
\usepackage{cleveref}

\usepackage{tikz,tikz-cd}
\usepackage{pgfplots}
\pgfplotsset{compat=1.18}
\usepackage{subcaption}

\hypersetup{bookmarksdepth = 3}

\makeatletter
\def\@tocline#1#2#3#4#5#6#7{\relax
  \ifnum #1>\c@tocdepth
  \else
    \par \addpenalty\@secpenalty\addvspace{#2}%
    \begingroup \hyphenpenalty\@M
    \@ifempty{#4}{%
      \@tempdima\csname r@tocindent\number#1\endcsname\relax
    }{%
      \@tempdima#4\relax
    }%
    \parindent\z@ \leftskip#3\relax \advance\leftskip\@tempdima\relax
    \rightskip\@pnumwidth plus4em \parfillskip-\@pnumwidth
    #5\leavevmode\hskip-\@tempdima
    \ifcase #1
    \or\or \hskip 1em \or \hskip 2em \else \hskip 3em \fi
    #6\nobreak\relax
    \dotfill\hbox to\@pnumwidth{\@tocpagenum{#7}}\par
    \nobreak
    \endgroup
  \fi}
\makeatother

\DeclareMathOperator{\Tot}{Tot}
\DeclareMathOperator{\Exp}{Exp}

\newcommand{\mr}{\mathrm}

\newcommand{\msf}{\mathsf}

\newcommand{\ri}{\mathrm{i}}
\newcommand{\re}{\mathrm{e}}

\newcommand{\mgp}[1]{\mathsf{#1}}
\newcommand{\invariantsfont}[1]{\mathsf{#1}}
\newcommand{\modulifont}[1]{#1}
\newcommand{\Mbar}{\overline{\modulifont{M}}}
\newcommand{\vir}{\mathrm{virt}}

\newcommand{\Aaff}[1]{\mathbb{C}^{#1}}
\renewcommand{\Gm}[1][]{%
  \ifthenelse{\isempty{#1}}
  {\mathbb{C}^\times}
  {(\mathbb{C}^\times)^{#1}}
}

\newcommand{\argvec}[3][]{%
  \ifthenelse{\isempty{#1}}
  {\boldsymbol{#2}_{#3}}
  {\boldsymbol{#2}_{#3}^{\boldsymbol{#1}}}
}
\newcommand{\epsvec}[2][]{\argvec[#1]{\epsilon}{#2}}
\newcommand{\zvec}[2][]{\argvec[#1]{z}{#2}}

\newcommand{\CYf}{Z}
\newcommand{\GW}{\invariantsfont{GW}}
\newcommand{\GWargs}[3]{\mathsf{F}_{#2}^{\phantom{Y}}(#1,#3)}
\newcommand{\GWdiscargs}[3]{\mathsf{Z}_{#2}^{\GW}(#1,#3)}
\newcommand{\PT}{\invariantsfont{PT}}
\newcommand{\PTargs}[4]{\mathsf{Z}_{#3}^{\PT}(#1 \subset #2,#4)}
\newcommand{\DT}{\invariantsfont{DT}}
\newcommand{\DTargs}[4]{\mathsf{Z}_{#3}^{\DT}(#1 \subset #2,#4)}
\newcommand{\Mtwo}{\invariantsfont{M2}}

\newcommand{\ch}[1][]{\mathrm{ch}_{#1}}
\newcommand{\td}[1][]{\mathrm{td}_{#1}}

\newcommand{\Chow}{\mathsf{CH}}
\newcommand{\hhh}{\mathsf{H}}

\newcommand{\HodgeLambda}[2][]{\Lambda_{#1}(#2)}
\newcommand{\HodgeH}{\mathsf{H}}
\newcommand{\HodgeHst}{\mathsf{H}^{\mathrm{st}}}

\newcommand{\vertex}{\mathsf{V}}
\newcommand{\vertexlimit}{\mathsf{U}}
\newcommand{\gamfct}[2]{\gamma(#1;#2)}
\newcommand{\cs}{\varsigma}
\newcommand{\mathqoppa}{\mathord{\text{\ensuregreek{\qoppa}}}}

\makeatletter
\def\makeCal#1{\expandafter\newcommand\csname c#1\endcsname{\mathcal{#1}}}
\def\makeBB#1{\expandafter\newcommand\csname b#1\endcsname{\mathbb{#1}}}
\def\makeFrak#1{\expandafter\newcommand\csname f#1\endcsname{\mathfrak{#1}}}
\def\makeScr#1{\expandafter\newcommand\csname s#1\endcsname{\mathscr{#1}}}
\def\makeSF#1{\expandafter\newcommand\csname sf#1\endcsname{\mathsf{#1}}}

\count@=0
\loop
  \advance\count@ 1
  \edef\y{\@Alph\count@}
  \expandafter\makeCal\y
  \expandafter\makeBB\y
  \expandafter\makeFrak\y
  \expandafter\makeScr\y
  \expandafter\makeSF\y
\ifnum\count@<26
\repeat
\makeatother

\theoremstyle{plain}
\newtheorem{thm}{Theorem}[section]
\newtheorem{prop}[thm]{Proposition}
\newtheorem{lem}[thm]{Lemma}
\newtheorem{cor}[thm]{Corollary}
\newtheorem{conj}[thm]{Conjecture}

\theoremstyle{definition}

\newtheorem{rmk}[thm]{Remark}
\newtheorem{example}[thm]{Example}
\newtheorem{defprop}[thm]{Definition/Proposition}

\theoremstyle{plain}
\newtheorem{innercustomthm}{Theorem}
\newenvironment{customthm}[1]
{\renewcommand\theinnercustomthm{#1}\innercustomthm}
{\endinnercustomthm}

\AddToHook{cmd/appendix/before}{\crefalias{section}{appendix}}

\crefname{equation}{Equation}{Equations}
\crefname{eqnarray}{Equation}{Equations}

\crefname{thm}{Theorem}{Theorems}
\crefname{prop}{Proposition}{Propositions}
\crefname{lem}{Lemma}{Lemmas}
\crefname{cor}{Corollary}{Corollaries}
\crefname{conj}{Conjecture}{Conjectures}
\Crefname{conj}{Conjecture}{Conjectures}
\crefname{defn}{Definition}{Definitions}
\crefname{rmk}{Remark}{Remarks}
\crefname{example}{Example}{Examples}
\crefname{defprop}{Definition/Proposition}{Definitions/Propositions}

\crefname{innercustomthm}{Theorem}{Theorems}
\crefname{innercustomcor}{Corollary}{Corollaries}
\crefname{innercustomconj}{Conjecture}{Conjectures}

\crefname{section}{Section}{Sections}
\crefname{appendix}{Appendix}{Appendices}
\crefname{figure}{Figure}{Figures}
\crefname{table}{Table}{Tables}

\usepackage[
  backend=biber,
  style=alphabetic,
  natbib=true,
  url=false,
  doi=true,
  eprint=true,
  isbn=false,
  maxcitenames=5,
  maxbibnames=4,
  maxalphanames=4,
  sortcites=true,
  giveninits=true
]{biblatex}

\AtEveryBibitem{\clearlist{language}}
\renewbibmacro{in:}{%
  \ifboolexpr{
    test {\ifentrytype{article}}
    or
    test {\ifentrytype{inproceedings}}
  }
  {}
  {\printtext{\bibstring{in}\intitlepunct}}%
}

\DeclareFieldFormat[article,inbook,incollection,inproceedings,patent,thesis,unpublished,techreport,misc,book]{title}{\mkbibquote{#1}}

\title{Membranes and Maps}

\author[A. Giacchetto]{Alessandro Giacchetto}
\address{Department Mathematik, ETH Z\"urich, CH-8006 Z\"urich, Switzerland}
\email[AG]{alessandro.giacchetto@math.ethz.ch}

\author[R. Pandharipande]{Rahul Pandharipande}
\email[RP]{rahul.pandharipande@math.ethz.ch}

\author[Y. Schuler]{Yannik Schuler}
\email[YS]{yannik.schuler@math.ethz.ch}

\date{September 2026}

\begin{document}

\begin{abstract}
    In positive degree, equivariant Gromov--Witten invariants of Calabi--Yau fivefolds are expected to admit an interpretation in terms of M2-branes, while we conjecture that constant maps are governed by the corresponding supergravity index. We make the first expectation precise by proposing a modular interpretation of M2-branes supported on several smooth curves meeting at an $n$-fold point. Together, the two pictures yield conjectural formulas for the Gromov--Witten invariants, which we translate into closed formulas for pointed and unpointed quintuple Hodge integrals. We show that these conjectures imply a K-theoretic Gromov--Witten/Pairs correspondence for local curves in degree one and link the generating series of constant maps to Donaldson--Thomas theory of points. We also prove the conjectures in two limits of the equivariant parameters. Along the way, we obtain new closed formulas for certain triple Hodge integrals.
\end{abstract}

\maketitle
\tableofcontents

\newpage
\section{Introduction}
\subsection{Overview}
A central speculation put forward in recent work of Brini and the third author \cite{BS24:refGW} is that the Gromov--Witten theory of a Calabi--Yau fivefold admits an interpretation in terms of M-theory. More precisely, for a Calabi--Yau fivefold $\CYf$ equipped with an action of a torus $\mgp{T}$ that leaves the Calabi--Yau form invariant, consider the Gromov--Witten series
\begin{equation} \label{gwseries}
    \GWargs{\CYf}{\mgp{T}}{\beta}
    =
    \sum_{g \geq 0}
        \int_{[\Mbar_g(\CYf,\beta)]^{\vir}_{\mgp{T}}} 1 \,
        \in
        \widehat{\mathsf{R}} \,,
\end{equation}
well-defined by $\mgp{T}$-equivariant localization whenever the $\mgp{T}$-fixed loci
\begin{equation}
    \Mbar_g(\CYf,\beta)^{\mgp{T}}
    \subset
    \Mbar_g(\CYf,\beta)
\end{equation}
of the moduli spaces of stable maps are proper.
The series \eqref{gwseries} lies in a suitable completion \smash{$\widehat{\mathsf{R}}$} of the field of fractions of $H^\bullet_{\mgp{T}}(\mathsf{pt})$.
For non-zero effective curve classes $\beta\in H_2(\CYf,\bZ)$, the expectation is that the series \eqref{gwseries} is governed by the moduli space $\Mtwo(\CYf,\beta)$ of M2-branes supported on curves in class $\beta$. When the curve class is primitive\footnote{
    Without the primitive assumption, the right-hand side of \eqref{GWM2} must be corrected by multi-covering contributions; see \cite[\S0.1 \& 0.6.2]{Sch26:TVCY5}.
} the speculation proposed in \cite{BS24:refGW} is
\begin{equation} \label{GWM2}
    \GWargs{\CYf}{\mgp{T}}{\beta}
    \overset{?}{=}
    \hat{\mathsf{A}}_{\mgp{T}}\big( \Mtwo(\CYf,\beta) \big) \,,
\end{equation}
where the right-hand side denotes the $\mgp{T}$-equivariant Hirzebruch \smash{$\hat{\mathsf{A}}$}-genus. Heuristically, \eqref{GWM2} reflects the expectation that the BPS sector of class $\beta$ is generated by M2-branes wrapping curves in class $\beta$, and that the supersymmetric index of the former is computed by a Dirac-type index on the corresponding moduli space.

The degree-zero sector should be of a different nature. The constraint $\beta=0$ corresponds to constant maps on the Gromov--Witten side. On the other hand, the moduli space of M2-branes with charge zero is empty, and we expect the relevant contribution to instead come from the fields of M-theory (more precisely, the fields of its associated low-energy eleven-dimensional supergravity theory).

At present, these expectations can serve at most as guiding principles. Indeed, very little is known mathematically about the moduli spaces $\Mtwo(\CYf,\beta)$ beyond the work of Nekrasov and Okounkov \cite{NO14:membranes}, which gives a modular interpretation for M2-branes whose support is either a reduced curve with at worst nodal singularities embedded in $\CYf$, or twice a smooth curve.

We study here the proposal \eqref{GWM2} for certain local models $\CYf$ for which $\hat{\mathsf{A}}_{\mgp{T}}\big(\Mtwo(\CYf,\beta)\big)$ is well defined and computable. We then extract from \eqref{GWM2} explicit conjectural identities for quintuple Hodge integrals, and prove these identities in two special regimes.

\subsection{Stable maps and membranes}
To make the above picture concrete in the effective curve class case, we study local models $\CYf_n$ in which the relevant curve class is represented by a union of rational curves meeting at a single $n$-fold point (for $1\leq n \leq 5$). The proposal \eqref{GWM2} then leads to a precise conjectural identification between Gromov--Witten invariants and the $\mgp{T}$-equivariant \smash{$\hat{\mathsf{A}}$}-genus of a moduli space of stable maps.

Let $\mgp{T}$ be a four-dimensional subtorus of $\Gm[5]$ acting on $\Aaff{5}$ and preserving the holomorphic volume form. We consider partial toric compactifications
\begin{equation}
    \Aaff{5} \subset \CYf_n
\end{equation}
satisfying the following properties:
\begin{enumerate}
    \item[(i)] $\CYf_n$ is $\mgp{T}$-equivariantly Calabi--Yau,

    \item[(ii)] the only $\mgp{T}$-invariant complete curves of $\CYf_n$ are the closures $C_1,\ldots,C_n\subset \CYf_n$ of the first $n$ coordinate axes of $\Aaff{5}$,
    
    \item[(iii)] the classes $[C_1],\ldots,[C_n]\in H_2(\CYf_n,\bZ)$ are linearly independent.
\end{enumerate}
Explicit constructions of $\CYf_n$ will be given in \cref{sec: membrane moduli}. The following conjecture proposes a modular interpretation of M2-branes supported on these rational curves meeting at an $n$-fold point.

\begin{conj}
    \label{conj:GW:vertex}
    For $n\in\{1,\ldots,5\}$ and curve class $\beta = [C_1]+\cdots+[C_n]$, the equivariant Gromov--Witten theory of $\CYf_n$ is given by
    \begin{equation}
    \label{eq: GW compact vertex formula intro}
       \GWargs{\CYf_n}{\mgp{T}}{\beta} 
       =
       \hat{\mathsf{A}}_{\mgp{T}}\bigl( \Mbar_{0}(\CYf_n,\beta) \bigr) \,,
    \end{equation}
    where $\Mbar_{0}(\CYf_n,\beta)$ denotes the moduli space of genus-zero stable maps to $\CYf_n$ in curve class $\beta$.
\end{conj}

Since the domains of stable maps have at worst ordinary double points, the most intricate geometric predictions of \cref{conj:GW:vertex} occur in the cases $n=3,4,5$. Since \smash{$\hat{\mathsf{A}}_{\mgp{T}}\bigl( \Mbar_{0}(\CYf_n,\beta) \bigr)$} concerns a single elementary moduli space, the right-hand side of \eqref{eq: GW compact vertex formula intro} is geometrically much simpler than the left-hand side.

\Cref{conj:GW:vertex} can be formulated in the following more general setting.
Let $\CYf$ be a $\mgp{T}$-equivariant Calabi--Yau fivefold, and let $\beta \in H_2(\CYf,\bZ)$ be an effective curve class.
Suppose {\em all $\mgp{T}$-invariant complete curves $C \subset \CYf$ representing $\beta$ are reduced and have only disconnecting singularities of $n$-fold type}\footnote{
    A singularity $p \in C$ of $n$-fold type is disconnecting if the normalization of $C$ at $p$ has exactly $n$ connected components. A non-singular point $p \in C$ is regarded as a disconnecting singularity of $1$-fold type.
} (for $1\leq n \leq 5$), then
\begin{equation}
\label{eq: GW compact vertex formula intro 2}
    \GWargs{\CYf}{\mgp{T}}{\beta}
    =
    \hat{\mathsf{A}}_{\mgp{T}}\bigl( \Mbar_{h}(\CYf,\beta) \bigr) \,,
\end{equation}
where $h$ is the minimal genus of a domain representing $\beta$. In order for the right-hand side of \eqref{eq: GW compact vertex formula intro 2} to be well defined, we must also assume {\em the moduli space of stable maps to $\CYf$ representing $\beta$ is a virtually smooth scheme with proper $\mgp{T}$-fixed locus.}

Given our presentation of \Cref{conj:GW:vertex}, one might speculate whether the moduli space of M2-branes admits a stable map interpretation beyond primitive curve classes. Direct numerical calculations, however, indicate that this is not the case. Instead, our numerical experiments show, for instance, agreement with Nekrasov--Okounkov's proposal for the moduli of M2-branes supported on twice a smooth curve \cite[\S5.2]{NO14:membranes}. Their proposed modular interpretation is a map/sheaf hybrid and, as such, is generally different from stable maps. This suggests that one should look for alternative modular interpretations of the right-hand side of \eqref{eq: GW compact vertex formula intro} in order to obtain a mathematical description of M2-branes. One possibility is to allow target degenerations, rather than only domain degenerations as in \cite{KKO14:unramified,Nes24:unramified}. We leave this question for future research.

\subsection{Constant maps and supergravity}
\label{sec: const maps}
We expect the degree-zero sector to be governed by the bulk supergravity fields rather than by moduli spaces of wrapped M2-branes. The constant-map contribution to the equivariant Gromov--Witten theory of $\CYf$ should then be controlled by the index of the eleven-dimensional supergravity multiplet discussed in \cite{Nek05:ChasingM}; see also \cite{Nek09:InstM,Ok15:lectures}. We propose the following conjectural description of the degree-zero Gromov--Witten series.

\begin{conj}
    \label{conj:index}
    The constant-map contribution to the equivariant Gromov--Witten theory of $\CYf$ is given by
    \begin{equation}
        \label{eq: const maps sugra}
        \exp\Big( 2 \, \GWargs{\CYf}{\mgp{T}}{0} \Big)
        \sim
        \Exp\Big(
            \chi_{\mgp{T}}\big(
                \CYf, T_{\CYf}^\vee - T_{\CYf}
            \big)
        \Big) \,.
    \end{equation}
    Here $\Exp$ denotes the plethystic exponential, and the symbol $\sim$ means that the logarithms of both sides have the same $\mathrm{o}(1)$ tail in their asymptotic expansions.
\end{conj}

In \cref{prop: sugra index asymptotics} we also determine the negative-power part of the asymptotic expansion of the supergravity index appearing on the right-hand side of \cref{eq: const maps sugra} which is of independent physical interest.

The factor of $2$ on the left-hand side of \cref{eq: const maps sugra} is related to the appearance of both the tangent and cotangent sheaves on the right-hand side. The same factor already appears in the study of constant-map contributions to the Gromov--Witten theory of threefolds \cite[Eq.~(2)]{MNOP1}. As stressed in \cite{NO14:membranes}, a satisfactory physical explanation of this doubling phenomenon is still lacking.

\subsection{Quintuple Hodge integrals}
\Cref{conj:GW:vertex,conj:index} both admit reformulations in terms of {\em quintuple Hodge integrals}. We package Hodge integrals in the generating series
\begin{equation}
    \HodgeH(z_1,\ldots,z_n ; \epsilon_1,\ldots, \epsilon_5;u)
    =
    \sum_{g \geq 0}
        u^{2g-2}
        \int_{\Mbar_{g,n}} \frac{\prod_{i=1}^5 \HodgeLambda[g]{\epsilon_i}}{\prod_{j=1}^n (z_j^{-1} - \psi_j)} \,,
\end{equation}
where $\Lambda_g$ is the Hodge class
\begin{equation}
    \HodgeLambda[g]{\epsilon}
    =
    \sum_{k = 0}^g (-1)^k \lambda_k \, \epsilon^{g-1-k} \,.
\end{equation}
The variables $\epsilon_1,\ldots,\epsilon_5$ are identified with the tangent weights of $\mgp{T}$ and are subject to the Calabi--Yau condition $\sum_{i=1}^5 \epsilon_i = 0$.
In the present setting, the genus counting variable $u$ is redundant, since its exponent can be recovered from homogeneity in the variables $\epsilon_i$ and $z_j$. Nevertheless, including $u$ will be convenient.

The Gromov--Witten invariants of \cref{conj:GW:vertex} can be expressed in terms of Hodge integrals via the virtual localization formula \cite{GP97:virtloc}: the left-hand side of \eqref{eq: GW compact vertex formula intro} decomposes into products of pointed quintuple Hodge integrals. After computing the right-hand side of \eqref{eq: GW compact vertex formula intro}, we find that \cref{conj:GW:vertex} is equivalent to the following explicit formulas for pointed quintuple Hodge integrals.

\begin{thm} 
    \label{thm: 5 Hodge n point formula}
    Suppose that $\sum_{i=1}^5 \epsilon_i = 0$. \Cref{conj:GW:vertex} holds if and only if, for all $n\in\{1,\ldots,5\}$, we have
    \begin{equation}
    \label{eq: 5 Hodge n point formula}
        \HodgeH(\epsilon_1^{-1},\ldots, \epsilon_n^{-1} ; \epsilon_1,\ldots, \epsilon_5 ; u)
        =
        \vertex(\epsilon_1,\ldots,\epsilon_n;u)
        \prod_{j=1}^n \frac{1}{u^2}
        \prod_{i\neq j} \frac{\gamfct{-\epsilon_j^{-1}\epsilon_i}{u\epsilon_j}}{\epsilon_i} \,,
    \end{equation}
    where
    \begin{equation}
    \label{eq: vertex function}
        \vertex(\epsilon_1,\ldots,\epsilon_n;u)
        \coloneqq
        \begin{cases}
            1
            & n=1 \,, \\[0.75em]
            \dfrac{\cs(u \epsilon_3)\cs(u \epsilon_4)\cs(u \epsilon_5)}{\cs(u (\epsilon_1+\epsilon_2))}
            & n=2 \,, \\[0.95em]
            \cs(u \epsilon_4)^2 \cs(u \epsilon_5)^2
            & n=3 \,, \\[0.75em]
            \Big(\prod_{i=1}^5 \cs(u\epsilon_i) \Big) \, \cs(u\epsilon_5)^2 \sum_{i=1}^4 \kappa(u\epsilon_i)
            & n=4 \,, \\[0.75em]
            \Big(\prod_{i=1}^5 \cs(u \epsilon_i) \Big)^2 \Big(\big(\sum_{i=1}^5 \kappa(u \epsilon_i)\big)^2 -\frac{1}{4} \Big)
            & n=5 \,,
        \end{cases}
    \end{equation}
    and $\cs(x) = 2 \sinh \tfrac{x}{2}$, $\kappa(x) = \frac{d}{dx} \log{\cs(x)} = \tfrac{1}{2} \coth \tfrac{x}{2}$.
\end{thm}

In \cref{eq: 5 Hodge n point formula}, $\gamfct{x}{u} \in \bQ(x)\llbracket u \rrbracket$ is the unique series in $u$ satisfying $\gamfct{0}{u}=1$ and the shift equation
\begin{equation}
    \gamfct{x+1}{u}
    =
    \frac{\cs(ux)}{ux} \, \gamfct{x}{u} \,.
\end{equation}
More explicitly, we have the formula
\begin{equation}
    \gamfct{x}{u}
    =
    \exp\bigg(
        \sum_{k>0}
            \frac{B_{2k}}{2k \, (2k+1)!} \,
            B_{2k+1}(x) \,
            u^{2k}
    \bigg) \,.
\end{equation}
It is also worth noting that, although the function $\vertex(\epsilon_1,\ldots,\epsilon_n;u)$ is written in a case-by-case form, it admits a simple geometric interpretation. Namely, it is the ratio between the $\hat{\mathsf A}$-genus of the moduli space of stable maps to $\CYf_n$ and the product of the corresponding $\hat{\mathsf A}$-genera for the local curves $\Tot N_{C_j/\CYf_n}$:
\begin{equation}
    \vertex(\epsilon_1,\ldots,\epsilon_n;u)
    =
    \frac{
        \hat{\mathsf{A}}_{\mgp{T}}\bigl(\Mbar_0(\CYf_n,[C_1]+\cdots+[C_n])\bigr)
    }{
        \prod_{j=1}^n
        \hat{\mathsf{A}}_{\mgp{T}}\bigl(\Mbar_0(\Tot N_{C_j/\CYf_n},[C_j])\bigr)
    } \,.
\end{equation}
On the other hand, \cref{conj:index} leads to a conjectural formula for unpointed quintuple Hodge integrals.

\begin{thm} 
    \label{thm: five hodge no psi}
    Suppose that $\sum_{i=1}^5 \epsilon_i =0$. \Cref{conj:index} is equivalent to the equality
    \begin{equation}
        \label{eq: 5 lambda 0 psi formula}
        \HodgeH(~;\epsilon_1,\ldots, \epsilon_5;u)
        =
        \frac{1}{2}
        \sum_{g>1}
            \frac{B_{2g-2}}{2g-2}
            \left[
                \frac{\sum_{i=1}^5 2\sinh(u\epsilon_i)}
                {\prod_{i=1}^5 2\sinh(\tfrac{u\epsilon_i}{2}) }
            \right]_{2g-2} u^{2g-2} \,,
    \end{equation}
    where $B_k$ are Bernoulli numbers and $[\,\cdot\,]_{2g-2}$ denotes the coefficient of $u^{2g-2}$.
\end{thm}


%

We prove \cref{thm: 5 Hodge n point formula,thm: five hodge no psi} in \cref{sec: membrane moduli,sec: const maps main part} respectively. We verify our conjectures on Hodge integrals in two special regimes.

\begin{thm}\label{thm: limits}
    ~
    \begin{enumerate}[label=\textup{(\roman*)}, ref=\textup{\roman*}]
        \item\label{item: antidiagonal}
        Formulas \eqref{eq: 5 Hodge n point formula} and \eqref{eq: 5 lambda 0 psi formula} hold if $\epsilon_i=-\epsilon_j$ for some $i\neq j$.

        \item\label{item: lamg lamg-1}
        After after multiplication of both sides of the equation by $\epsilon_1\cdots\epsilon_5$, formulas \eqref{eq: 5 Hodge n point formula}, for $n\in\{1,2,3\}$, and \eqref{eq: 5 lambda 0 psi formula} hold as equalities in $\bQ[\epsilon_1,\ldots,\epsilon_5] / (\sum_i \epsilon_i,\epsilon_4^2,\epsilon_4\epsilon_5,\epsilon_5^2)\llbracket u \rrbracket$.
    \end{enumerate}    
\end{thm}

We prove this result in \cref{sec:special regimes}. \Cref{thm: limits}.\labelcref{item: antidiagonal} follows from Mumford's relation \cite{Mum83:Towards},
\begin{equation}
    \HodgeLambda[g]{\epsilon} \HodgeLambda[g]{-\epsilon}
    =
    (-\epsilon^2)^{g-1} \,,
\end{equation}
as quintuple Hodge integrals reduce to triple Hodge integrals. Some of the required triple Hodge formulas have already appeared in the literature, and we derive the remaining ones in \cref{sec: triple Hodge formulas 1,sec: triple Hodge formulas 2}.

\Cref{thm: limits}.\labelcref{item: lamg lamg-1} is also proved by reducing quintuple Hodge integrals to triple ones. Here we use the fact that, in this regime, each integrand contains a factor
\begin{equation}
    \lambda_g\lambda_{g-1}
    =
    (-1)^{g-1} (2g-1)! \, \ch[2g-1] (\bE) \,,
\end{equation}
which, by Mumford's formula, can be traded for $\kappa$- and $\psi$-insertions.

\subsection{Local curves \texorpdfstring{\&}{and} stable pairs}
We conclude with an application to degree-one Gromov--Witten invariants of local curves. Let
\begin{equation}
    \begin{tikzcd}
        \CYf \coloneqq &[-3em] \Tot_C \, \cL_2 \oplus \cL_3 \oplus \cL_4 \oplus \cL_5 \ar[d] \\
        & {C}
    \end{tikzcd}%
\end{equation}
be a local Calabi--Yau curve equipped with a torus $\mgp{T}$ acting on the fibers. It turns out that, in degree one, the equivariant Gromov--Witten invariants of $\CYf$ are determined uniquely by the single series
\begin{equation}
    \label{eq: GW series O(-2) x C3}
    \GWargs{
        \Tot_{\bP^1}\, \cO(-2)\oplus \cO \oplus \cO \oplus \cO
    }{\mgp{T}}{[\bP^1]} \,.
\end{equation}
Indeed, exactly as in the three-dimensional case \cite{BP08:LocGW}, the degeneration formula \cite{Li02:degFormula} expresses the Gromov--Witten series of a local curve in terms of certain basic series. In degree one, all these series turn out to be determined by \eqref{eq: GW series O(-2) x C3}. Since \cref{conj:GW:vertex} yields a conjectural formula for this series, we obtain the following conditional result.

\begin{thm}
    \label{thm: intro local curve GW g formula}
    Suppose that $\mgp{T}$ acts on the fibers of the local curve $\CYf$ with weights $\epsilon_2,\ldots,\epsilon_5$ satisfying the Calabi--Yau condition $\sum_i \epsilon_i=0$. If \cref{conj:GW:vertex} holds for $n=1$, then
    \begin{equation}
        \label{eq: local curve GW g formula intro}
        \GWargs{\CYf}{\mgp{T}}{[C]}
        =
        \prod_{i=2}^5 \big(
            2 \sinh \tfrac{\epsilon_i}{2}
        \big)^{g_C - 1 - \deg \cL_i}
        =
        \hat{\mathsf{A}}_{\mgp{T}}\big( \Mbar_{g_C}(\CYf,[C]) \big) \,.
    \end{equation}
\end{thm}

This formula should be compared with Donaldson--Thomas theory, more precisely with the $K$-theoretic stable pair invariants of the threefold $X=\Tot_C \,\cL_2\oplus\cL_3$. In \cite[Conj.~2.1]{NO14:membranes}, Nekrasov and Okounkov predict more generally that, for all threefolds, the generating series of these invariants is equal to the membrane index, in close analogy with our prediction for stable maps.

To make this precise, suppose that $X$ is a smooth quasi-projective threefold equipped with the action of a torus $\mgp{T}'$, and that $\cL_4$, $\cL_5$ are equivariant line bundles on $X$. We assume that these bundles satisfy the Calabi--Yau condition $\cL_4\otimes\cL_5\cong K_X$. Let $P_n(X,\beta)$ denote the moduli space of stable pairs on $X$ with Euler characteristic $n$ and support $\beta$. Nekrasov and Okounkov then suggest studying the generating series
\begin{equation}
     \PTargs{X}{\CYf}{\mgp{T}'}{\beta}
     \coloneqq
     \sum_{n\in \bZ}
        (-1)^{n+(\cL_4,\beta)}
        q^{n + \frac{(\cL_4+\cL_5,\beta)}{2}} \,
        \chi_{\mgp{T}'} \bigl(
            P_n(X,\beta) , \widetilde{\cO}_{\vir}
        \bigr) \,,
\end{equation}
of the Euler characteristics of a suitable twist of the virtual structure sheaf of these moduli spaces:
\begin{equation}
    \widetilde{\cO}_{\vir}
    =
    \cO_{\vir} \otimes \Big(
        K_{\vir} \otimes \det\big(
            \mathbf{R}\pi_{*} \,\bF \otimes \rho^*(\cL_4-\cL_5)
        \big)
    \Big)^{\frac{1}{2}} \,.
\end{equation}
Here, $\bF$ denotes the universal sheaf on $P_n(X,\beta)\times X$, while $\pi$ and $\rho$ are the projections to the respective factors.

The twist is chosen precisely so that the generating series, which a priori encodes only invariants of a threefold, behaves like the generating series of a curve-counting theory on the fivefold $\CYf = \Tot_X \, \cL_4\oplus\cL_5$. From this perspective, one should regard the formal box-counting variable $q$ as the coordinate of a one-dimensional torus $\Gm_q$ acting fiber-wise with opposite weights $\pm1$ on $\cL_4$ and $\cL_5$. Indeed, the central speculation of \cite{NO14:membranes} is that the $K$-theoretic stable pair series equals the index of the moduli space of M2-branes on $\CYf$. Hence, we should expect it to agree with the Gromov--Witten series of $\CYf$ as well.

To state this correspondence precisely, let us denote by
\begin{equation}
    \GWdiscargs{\CYf}{\mgp{T}}{\beta}
\end{equation}
the generating series of disconnected $\mgp{T}=\mgp{T}'\times\Gm_q$-equivariant Gromov--Witten invariants of $\CYf$ in curve class $\beta$, excluding stable maps with contracted connected components of the domain. In other words,
\begin{equation}
    1+ \sum_{\beta \neq 0} Q^\beta \, \GWdiscargs{\CYf}{\mgp{T}}{\beta}
    =
    \exp \left(
        \sum_{\beta \neq 0} Q^\beta \, \GWargs{\CYf}{\mgp{T}}{\beta}
    \right) \,,
\end{equation}
where both sums run over all non-zero effective curve classes in $\CYf$.

\begin{conj}
    \label{conj: maps pairs}
    The stable pair series lifts to a rational function in $q$ satisfying
    \begin{equation}
        \ch[\mgp{T}] \, \PTargs{X}{\CYf}{\mgp{T}'}{\beta}
        =
        \GWdiscargs{\CYf}{\mgp{T}}{\beta} 
    \end{equation}
    under the change of variables induced by the Chern character $\ch[\mgp{T}] \colon \mr{Rep}(\mgp{T}) \to \widehat{\mathsf{R}}$.
\end{conj}

This conjecture was first stated in \cite[Conj.~7.20]{BS24:refGW} in the special case $\cL_4 = \cL_5 = \cO_X$. If, in addition, $\mgp{T}'$ fixes the Calabi--Yau form of $X$, the conjecture reduces to the original Gromov--Witten/Pairs correspondence for Calabi--Yau threefolds \cite[Conj.~3.3]{PT:StablePairs}, as explained in \cite[\S7.2.3]{BS24:refGW}. The original conjecture has been proven for toric threefolds in \cite{MOOP11:GWDTtoric} and for general Calabi-Yau threefolds by Pardon \cite{Par:UnivCountCurve}.

Let us return to our discussion of local curves and assume that
$X = \Tot_C \,\cL_2\oplus\cL_3$. In \cite[Cor.~5.1.19]{Ok15:lectures}, Okounkov proves the following formula for the degree-one stable pair series:
\begin{equation}
     \PTargs{X}{Z}{\mgp{T}'}{[C]}
     =
     \chi_{\mgp{T}}\Big(
        \Mbar_{g_C}(\CYf,[C]),
        \cO_{\vir}\otimes K_{\vir}^{\frac{1}{2}}
     \Big) \,.
\end{equation}
See also \cite[Cor.~6.9]{Mon:refinedDT:local} for a different proof of a special case of this formula. Now by Hirzebruch--Riemann--Roch, the image of the right-hand side under the Chern character is precisely the \smash{$\hat{\mathsf A}$}-genus of \smash{$\Mbar_{g_C}(\CYf,[C])$}. Since \Cref{thm: intro local curve GW g formula} identifies the latter with the all-genus Gromov--Witten series of $\CYf$, we arrive at the following conclusion.

\begin{cor}
    Suppose that \cref{conj:GW:vertex} holds for $n=1$. Then \Cref{conj: maps pairs} holds for local curves in degree one.
\end{cor}

\subsection{Constant maps \texorpdfstring{\&}{and} Donaldson--Thomas theory of points}
As in the previous section, our \Cref{conj:index} for the generating series of constant maps allows a comparison with the $K$-theoretic Donaldson--Thomas theory of points on threefolds.

Suppose that $X$ is a smooth quasi-projective threefold equipped with the action of a torus $\mgp{T}'$. Choose two equivariant line bundles $\cL_4$, $\cL_5$ satisfying the Calabi--Yau condition $\cL_4 \otimes \cL_5 \cong K_X$. Following again the proposal of Nekrasov and Okounkov \cite[Def.~3.1]{NO14:membranes}, we form the generating series of equivariant Euler characteristics of the twisted virtual structure sheaf of the Hilbert scheme of points:
\begin{equation}
    \DTargs{X}{Z}{\mgp{T}'}{0}
    \coloneqq
    \sum_{n \geq 0}
        (-q)^n \, \chi_{\mgp{T}'} \Big(
            \mr{Hilb}_n(X),
            \widetilde{\cO}_{\vir}
        \Big) \,.
\end{equation}
In \cite[Thm.~3.3.6]{Ok15:lectures}, Okounkov proves the following formula conjectured by Nekrasov \cite{Nek05:ChasingM} for this generating series:
\begin{equation}
    \DTargs{X}{Z}{\mgp{T}'}{0}
    =
    \Exp \left(
        \chi_{\mgp{T}'} \biggl(
            X,
            \frac{
                q \cL_4 \bigl( T_X + K_X - T_X^\vee - K_X^{-1} \bigr)
            }{
                \bigl(1-q\cL_4 \bigr)
                \bigl(1-q\cL_5^{-1}\bigr)
            }
        \biggr)
    \right) \,.
\end{equation}
This formula suggests that we should regard $q$ as the coordinate of a torus $\Gm_q$ acting fiber-wise with opposite weights on $\cL_4$ and $\cL_5$. Indeed, after adding a correction term, the argument of the plethystic exponential on the right-hand side admits a simple expression as an Euler characteristic on the fivefold $\CYf=\Tot_X\cL_4\oplus\cL_5$, equivariant with respect to the torus $\mgp{T} = \mgp{T}' \times \Gm_q$ (cf.~\cite[\S2.4.7]{NO14:membranes}):
\begin{equation}
   \chi_{\mgp{T}'}\big(X, K_X - \cO_X \big)
   +
   \chi_{\mgp{T}'} \left(
        X,
        \frac{
            q\cL_4 \bigl(T_X +K_X - T_X^\vee - K_X^{-1} \bigr)
        }{
            \bigl(1-q\cL_4\bigr) \bigl(1-q\cL_5^{-1}\bigr)
        }
    \right)
    =
    \chi_{\mgp{T}} \bigl( \CYf, T_{\CYf}^\vee - T_{\CYf} \bigr) \,.
\end{equation}
We identify the right-hand side with the argument of the plethystic exponential appearing in the index of eleven-dimensional supergravity. Hence, equivalently, \cref{conj:index} can be formulated as a relation between the $K$-theoretic Donaldson--Thomas series of points and the constant-map series.

\begin{conj}
    If $\CYf = \Tot_{X} \cL_4 \oplus \cL_5$, then
    \begin{equation}
        \Exp \Big( \chi_{\mgp{T}'}\big(X, K_X - \cO_X \big) \Big)
        \, \cdot \,
        \DTargs{X}{Z}{\mgp{T}'}{0}
        \sim
        \exp\Big( 2 \, \GWargs{\CYf}{\mgp{T}}{0} \Big) \,.
    \end{equation}
\end{conj}

By \cref{thm: limits}, the conjecture holds for all threefolds with Chern roots $\alpha_1,\alpha_2,\alpha_3$ satisfying either $\alpha_1=-\alpha_2$ or $\alpha_1^2=\alpha_2^2=\alpha_1\alpha_2=0$. Moreover, part (\labelcref{item: antidiagonal}) of the same \namecref{thm: limits} proves the conjecture when $\cL_4 \cong \cL_5^{-1}$ as $\mgp{T}'$-equivariant line bundles. 

\subsection*{AI disclosure}
The mathematical content of the paper is due to the authors. Generative AI was solely used to improve the language and grammar of the paper.

\subsection*{Acknowledgments}
Discussions with A.~Brini, N.~Fasola, D.~Maulik, S.~Monavari, A.~Okounkov, D.~Ranganathan, M.~Schimpf and S.~Shadrin have played an important role in our work.

A.G.~was supported by an ETH Fellowship (22-2 FEL-003) and a Hermann-Weyl-Instructor\-ship from the Forschungsinstitut für Mathematik at ETH Zürich. 
R.P.~was supported by SNF-200020-219369 and SwissMAP.
Y.S.~was supported by SNF-200020-219369 and a Walter Benjamin Fellowship of the Deutsche Forschungsgemeinschaft (DFG) --- Projektnummer 576663726.

\subsection*{Notation}
We denote the generating series of Hodge integrals over the stable range by
\begin{equation}
    \label{eq: def Hodge H st}
    \HodgeHst(z_1,\ldots, z_n ; \epsilon_1,\ldots, \epsilon_d ; u)
    \coloneqq
    \sum_{\substack{g \geq 0 \\ 2g-2+n>0}} u^{2g-2} \,
    \int_{\Mbar_{g,n}} \frac{\prod_{i=1}^d \HodgeLambda[g]{\epsilon_i}}{\prod_{j=1}^n (z_j^{-1} - \psi_j)} \,.
\end{equation}
To ease notation, we will use the short-hand $\epsvec{I} \coloneqq (\epsilon_i)_{i\in I}$ to denote arguments labeled by some index set $I$. For instance, with the notation $[n]\coloneqq (1,\ldots,n)$ we will write $$\HodgeHst(\zvec{[n]} ; \epsvec{[d]} ; u)$$ for the generating series \eqref{eq: def Hodge H st}. For the generating series including unstable terms we write
\begin{equation}
\label{eq: Hodge unstable}
    \HodgeH(\zvec{[n]} ; \epsvec{[d]} ; u)
    =
    \frac{\delta_{n,1}}{u^2}
    \frac{1}{z_1 \prod_{i=1}^d \epsilon_i}
    +
    \frac{\delta_{n,2}}{u^2}
    \frac{z_1 z_2}{(z_1 + z_2) \prod_{i=1}^d \epsilon_i}
    +
    \HodgeHst(\zvec{[n]} ; \epsvec{[d]} ; u) \,.
\end{equation}

\section{Membrane moduli}
\label{sec: membrane moduli}

The main result of this section is the proof of \cref{thm: 5 Hodge n point formula}, which establishes an equivalence between two statements. On the one hand, \cref{conj:GW:vertex} identifies the all-genus Gromov–Witten series of a partial compactification of $n$ coordinate lines in $\mathbb{C}^5$, in primitive curve class, with the \smash{$\hat{\mathsf{A}}$}-genus of the moduli space of genus-zero stable maps. On the other hand, \eqref{eq: 5 Hodge n point formula} gives an explicit formula for quintuple Hodge integrals. \Cref{thm: 5 Hodge n point formula} asserts that these two statements are equivalent. The proof proceeds by evaluating the \smash{$\hat{\mathsf{A}}$}-genus by localization and comparing the resulting expression with the localization formula for the Gromov--Witten series.

We begin with the case $n=1$, which reduces to a local curve calculation and already contains the main features of the argument. As a by-product, we obtain a formula for degree-one Gromov--Witten invariants of local curves of arbitrary genus, summarized in the introduction as \cref{thm: intro local curve GW g formula}. We then turn to the case $n>1$ and complete the proof by the same strategy.

\subsection{One-pointed Hodge integrals and local curves}
We start by recalling \cref{conj:GW:vertex} for $n=1$. In this case, the \namecref{conj:GW:vertex} concerns the Gromov--Witten theory of a partial compactification $\CYf$ of affine five-space in which the first coordinate line is compactified to a rational curve. Such a compactification is described by the splitting of the normal bundle of that rational curve. Equivalently, $\CYf$ is the total space of four line bundles over the projective line.
\begin{equation}
    \label{eq: loc P1}
    \begin{tikzcd}
        \CYf = &[-3em] \Tot_{\bP^1}\, \cL_2 \oplus \cL_3 \oplus \cL_4 \oplus \cL_5 \ar[d] \\
        & {\bP^1}
    \end{tikzcd}%
\end{equation}
The condition that $\CYf$ be Calabi--Yau translates into the requirement $\sum_i \deg \cL_i =-2$. We extend a torus action of $\mgp{T}\cong\Gm[4]$, with tangent weights $\epsilon_1,\ldots,\epsilon_5$ on the affine patch $\Aaff{5}$ at the point $0\in\bP^1$, to $\CYf$. These weights are subject to the Calabi--Yau condition $\sum_{i=1}^5 \epsilon_i =0$. In this setting, \cref{conj:GW:vertex} predicts the following.

\begin{conj}
    \label{conj: local curve GW vs M2}
    (\Cref{conj:GW:vertex}, $n=1$) For every Calabi--Yau local curve over $\bP^1$ equipped with a Calabi--Yau $\mgp{T}$-action, we have
    \begin{equation}
        \label{eq: local curve GW vs M2}
        \GWargs{\CYf}{\mgp{T}}{[\bP^1]}
        =
        \hat{\mathsf{A}}_{\mgp{T}}\Big( \Mbar_{0}\big(\CYf,[\bP^1]\big) \Big) \,.
    \end{equation}
\end{conj}

\subsubsection{The \texorpdfstring{$\hat{\mathsf{A}}$}{Â}-genus}
The right-hand side of \eqref{eq: local curve GW vs M2} may deserve some explanation. The Hirzebruch \smash{$\hat{\mathsf{A}}$}-genus is the genus associated with the formal power series
\begin{equation}
    \hat{\mathsf{A}}(x)
    \coloneqq
    \frac{x/2}{\sinh(x/2)}
    =
    1 - \frac{1}{24}x^2 + \frac{7}{5760} x^4+ \cdots \,.
\end{equation}
For a scheme $M$ equipped with a $\mgp{T}$-action, this genus defines a multiplicative characteristic class $\hat{\mathsf{A}} \colon K^0_{\mgp{T}}(M) \rightarrow \prod_k \Chow^k_{\mgp{T}}(M)$ by the rule $\hat{\mathsf{A}}(\cL) \coloneqq \hat{\mathsf{A}}(c_1(\cL))$ on line bundles.
If $M$ is equipped with a $\mgp{T}$-equivariant perfect obstruction theory, we define its ($\mgp{T}$-equivariant virtual) \smash{$\hat{\mathsf{A}}$}-genus by
\begin{equation}
    \hat{\mathsf{A}}_{\mgp{T}}(M)
    \coloneqq
    \int_{[M]^{\vir}_{\mgp{T}}} \hat{\mathsf{A}}(T_M^{\vir}) \,.
\end{equation}
Since the \smash{$\hat{\mathsf{A}}$}-genus plays a prominent role in these notes, the following functions will appear frequently:
\begin{equation}
    \cS(x) \coloneqq \frac{\sinh (x/2)}{x/2} \,,
    \qquad
    \cs(x) \coloneqq x\cS(x) = 2 \sinh (x/2) \,.
\end{equation}

\subsubsection{Evaluating the \texorpdfstring{$\hat{\mathsf{A}}$}{Â}-genus}
We now apply this formalism to the moduli space of genus-zero degree-one stable maps to a local curve with rational base and evaluate the right-hand side of the conjectural formula \eqref{eq: local curve GW vs M2}.

\begin{lem}
    \label{lem: Ahat local curve formula}
    We have
    \begin{equation}
        \hat{\mathsf{A}}_{\mgp{T}}\Big( \Mbar_{0}\big(\CYf,[\bP^1]\big) \Big)
        =
        \prod_{i=2}^5 \prod_{k=0}^{\deg \cL_i}
            \cs(\epsilon_i - k\epsilon_1)^{-1} \,,
    \end{equation}
    where we interpret $\prod_{k=0}^{d} a(k) \coloneqq \prod_{k=1}^{-d-1} a(-k)^{-1}$ whenever $d<0$.
\end{lem}

\begin{proof}
    There is a unique $\mgp{T}$-fixed genus-zero stable map to $\CYf$ in degree one, namely the embedding of the zero section $0 \colon \bP^1 \hookrightarrow \CYf$. Since the fixed locus consists of a single isolated point, its virtual tangent bundle is trivial and the virtual normal bundle coincides with the full virtual tangent space at that point. The latter is given by
    \begin{equation}
        T^{\vir}_{\Mbar_{0}(\CYf,[\bP^1])} \,\Big|_{0}
        =
        \hhh^0(\bP^1,\oplus_i \cL_i) - \hhh^1(\bP^1,\oplus_i \cL_i)
    \end{equation}
    in equivariant $K$-theory. Write this virtual representation as $\sum_a x_a - \sum_b y_b$, where the $x_a$ and $y_b$ are one-dimensional $\mgp{T}$-representations. By virtual localization, the contribution of a weight $x$ is $\hat{\mathsf{A}}(c_1(x))/c_1(x)$. Here the denominator comes from the Euler class of the virtual normal bundle, which for a one-dimensional representation of weight $x$ is just $c_1(x)$. Since $\hat{\mathsf{A}}(x) = x\cs(x)^{-1}$, it follows that
    \begin{equation}
        \hat{\mathsf{A}}_{\mgp{T}}\Big( \Mbar_{0}\big(\CYf,[\bP^1]\big) \Big)
        =
        \frac{\prod_b \cs\big(c_1(y_b)\big)}{\prod_a \cs\big(c_1(x_a)\big)} \,.
    \end{equation}
    Finally, expressing the $\mgp{T}$-weights of the above virtual tangent representation in terms of the generators $\epsilon_1,\ldots,\epsilon_5$ yields the claimed formula.
\end{proof}

\Cref{lem: Ahat local curve formula} gives an explicit conjectural formula for the all-genus Gromov--Witten series.

\begin{cor}
    \label{cor: local curve GW g0 formula}
    \cref{conj:GW:vertex} for $n=1$ holds if and only if
    \begin{equation}
    \label{eq: local curve GW g0 formula}
        \GWargs{\CYf}{\mgp{T}}{[\bP^1]}
        =
        \prod_{i=2}^5 \prod_{k=0}^{\deg \cL_i} \cs(\epsilon_i - k\epsilon_1)^{-1} \,.
    \end{equation}
\end{cor}

\begin{example}
    \label{ex: RC and A1 local curves}
    Two special cases of this formula were already conjectured in \cite{BS24:refGW}:
    \begin{itemize}
        \item If $\cL_2=\cL_3=\cO_{\bP^1}(-1)$ and $\cL_4$ and $\cL_5$ are trivial, then the Calabi--Yau fivefold $\CYf$ is the product of the resolved conifold with the affine plane. In this case, formula \eqref{eq: local curve GW g0 formula} becomes
        \begin{equation}
            \GWargs{\CYf}{\mgp{T}}{[\bP^1]}
            =
            \frac{1}{\cs(\epsilon_4)\,\cs(\epsilon_5)} \,.
        \end{equation}
        
        \item If $\cL_2=\cO_{\bP^1}(-2)$ and $\cL_3$, $\cL_4$, and $\cL_5$ are trivial, then formula \eqref{eq: local curve GW g0 formula} specializes to
        \begin{equation}
            \GWargs{\CYf}{\mgp{T}}{[\bP^1]}
            =
            \frac{\cs(\epsilon_1+\epsilon_2)}{\cs(\epsilon_3)\,\cs(\epsilon_4)\,\cs(\epsilon_5)} \,.
        \end{equation}
    \end{itemize}
\end{example}

\subsubsection{One-pointed Hodge integrals and the function \texorpdfstring{$\gamma$}{ɣ}}
Having discussed \cref{conj:GW:vertex}, we now turn our attention to the Hodge integrals of \cref{thm: 5 Hodge n point formula}. In the case $n=1$, we ought to show the conjecture is equivalent to following formula for one-pointed Hodge integrals:
\begin{equation}
    \label{eq: 5 Hodge 1 point formula}
    \HodgeH(\epsilon_1^{-1}; \epsvec{[5]};u)
    \overset{?}{=}
    \frac{1}{u^2} \prod_{i=2}^5
        \frac{\gamfct{-\epsilon_1^{-1}\epsilon_i}{u\epsilon_1}}{\epsilon_i} \,.
\end{equation}
The function $\gamma$ appearing on the right-hand side is characterized as follows.

\begin{defprop}
    \label{defprop: gamma fct}
    The power series
    \begin{equation}
        \gamfct{x}{u}
        =
        \exp\bigg(
            \sum_{k>0}
                \frac{B_{2k}}{2k \, (2k+1)!} \,
                B_{2k+1}(x) \,
                u^{2k}
        \bigg) \,.
    \end{equation}
    is the unique element of $\bQ[x]\llbracket u \rrbracket$ satisfying
    \begin{equation}
        \label{eq: gamma fct shift}
        \gamfct{x+1}{u}
        =
        \cS(u x) \, \gamfct{x}{u} \,,
        \qquad
        \gamfct{0}{u}
        =
        1 \,.
    \end{equation}
    Here, $B_k(x)$ denotes the Bernoulli polynomial $\frac{t \re^{x t}}{\re^{t}-1} \eqqcolon \sum_{k\geq 0} \frac{B_k(x)}{k!} t^k $, and $B_k \coloneqq B_k(0)$ denotes the Bernoulli number.
\end{defprop}

\begin{proof}
    The identity $\gamfct{0}{u}=1$ follows immediately from the vanishing $B_{2k+1}(0)=0$. The shift equation is a consequence of the relations
    \begin{equation}
        \cS(x)
        =
        \exp \bigg(\sum_{k>0} \frac{B_{2k}}{2k\,(2k)!}x^{2k}\bigg) \,,
        \qquad
        B_m(x+1) = B_m(x)+ m x^{m-1} \,.
    \end{equation}
    For uniqueness, note that the initial condition at $x=0$ together with the shift equation recursively determines the value of the power series at $x=d$ for every non-negative integer $d$. Hence any two elements of $\bQ[x]\llbracket u \rrbracket$ satisfying \eqref{eq: gamma fct shift} agree at infinitely many values of $x$. Since their coefficients are polynomial functions in $x$, it follows that they agree identically.
\end{proof}

We will need the following property, analogous to the reflection formula satisfied by the Gamma function, which follow immediately from the fact that the power series is uniquely determined by the shift \cref{eq: gamma fct shift}.

\begin{lem}
    \label{lem: gamma fct reflect}
    For all positive integers $d$ we have
    \begin{equation}
        \gamfct{-x}{u}
        =
        \frac{1}{\gamfct{x+1}{u}} \,,\qquad
        \gamfct{x}{u}\,\gamfct{-x-d}{u}
        =
        \prod_{k=0}^{d} \cS\big(u(x+k)\big)^{-1} \,.
    \end{equation}
\end{lem}

\begin{rmk}
    The formal series $\gamfct{x}{u}$ is related to the asymptotics of the $q$-Gamma function
    \begin{equation}
        \Gamma_q(x)
        =
        \frac{\prod_{k= 0}^{\infty}(1-q^{-1-k})}{\prod_{k \ge 0}^{\infty}(1-q^{-x-k})} \,,
        \qquad
        |q|>1 \,.
    \end{equation}
    Indeed, as $q=\re^{u} \rightarrow 1^+$ we have
    \begin{equation}
        \log \Gamma_q(x)
        \sim
        (x-1) \log u + \log \Gamma(x)  -\frac{x(x-1)}{4} u + \log \gamfct{x}{u} \,.
    \end{equation}
\end{rmk}

\subsubsection{Proof of \texorpdfstring{\cref{thm: 5 Hodge n point formula}}{the equivalence} for \texorpdfstring{$n=1$}{n=1}}
\label{sec: membrane moduli n 1}
For ease of notation, set $d_i = \deg \cL_i$. We also introduce the genus-counting variable in the Gromov--Witten generating series, weighting genus-$g$ contributions by $u^{2g-2}$. By \cref{cor: local curve GW g0 formula}, it suffices to prove the equivalence between the formula
\begin{equation}
\label{eq: local curve GW g0 formula 2}
    \GWargs{\CYf}{\mgp{T}}{[\bP^1]}
    =
    \prod_{i=2}^5 \prod_{k=0}^{d_i} \cs(u\epsilon_i - k u\epsilon_1)^{-1} \,.
\end{equation}
for the Gromov--Witten series and the formula for one-pointed Hodge integrals given in \eqref{eq: 5 Hodge 1 point formula}. To this end, we evaluate the Gromov--Witten series by virtual localization \cite{GP97:virtloc}:
\begin{equation}
    \label{eq: local curve GW localized}
    \GWargs{\CYf}{\mgp{T}}{[\bP^1]}
    =
    \HodgeH\big( \epsilon_1^{-1}; \epsvec[0]{[5]};u \big)
    \,
    \frac{u^2}{\prod_{i=2}^5 \prod_{k=1}^{d_i-1} (\epsilon_i-k\epsilon_1)}
    \,
    \HodgeH\big( -\epsilon_1^{-1}; \epsvec[\infty]{[5]};u \big) \,.
\end{equation}
Here,
\begin{equation}
    \epsvec[0]{[5]} = (\epsilon_1,\epsilon_2,\epsilon_3,\epsilon_4,\epsilon_5) \,,
    \quad
    \epsvec[\infty]{[5]} = (-\epsilon_1,\epsilon_2-d_2\epsilon_1,\epsilon_3-d_3\epsilon_1,\epsilon_4-d_4\epsilon_1,\epsilon_5-d_5\epsilon_1) \,,
\end{equation}
denote the collections of tangent $\mgp{T}$-weights at the fixed points $0$ and $\infty$ in $\bP^1\subset \CYf$. The two Hodge series arise from contracted components in the domains of $\mgp{T}$-fixed stable maps, while the middle factor records the deformations and obstructions of the rational bridge connecting them.

One implication is now immediate. Substituting \eqref{eq: 5 Hodge 1 point formula} for the quintuple Hodge generating series into the right-hand side of \eqref{eq: local curve GW localized}, and using the reflection formula from \cref{lem: gamma fct reflect}, we recover \eqref{eq: local curve GW g0 formula 2}. For the converse, suppose that the Gromov--Witten series satisfies \eqref{eq: local curve GW g0 formula 2}. Then, in particular, the formulas of \cref{ex: RC and A1 local curves} hold. Combined with \eqref{eq: local curve GW localized}, these imply that the generating series $\HodgeH(\epsilon_1^{-1}; \epsvec{[5]} ;u)$ satisfies the shift equation
\begin{equation}
    \HodgeH(\epsilon_1^{-1}; \epsilon_1, \epsilon_2 + \epsilon_1, \epsilon_3 - \epsilon_1, \epsilon_4,\epsilon_5 ;u)
    =
    \frac{\epsilon_2}{\cs(u(\epsilon_1+\epsilon_2))} \,
    \frac{\cs(u\epsilon_3)}{\epsilon_3 - \epsilon_1} \,
    \HodgeH(\epsilon_1^{-1}; \epsilon_1, \epsilon_2, \epsilon_3, \epsilon_4,\epsilon_5 ;u) \,.
\end{equation}
By the defining property \eqref{eq: gamma fct shift} of the function $\gamma$, the product
\begin{equation}
    \frac{1}{u^2} \prod_{i=2}^5 \frac{\gamfct{-\epsilon_1^{-1}\epsilon_i}{u\epsilon_1}}{\epsilon_i}
\end{equation}
satisfies the same shift equation. Later, in \cref{sec:special regimes}, we will show that
\begin{equation}
    \HodgeH(\epsilon_1^{-1}; \epsvec{[5]};u)
    \qquad \text{and} \qquad
    \frac{1}{u^2} \prod_{i=2}^5 \frac{\gamfct{-\epsilon_1^{-1}\epsilon_i}{u\epsilon_1}}{\epsilon_i} \,,
\end{equation}
agree after specializing to $\epsilon_4 = -\epsilon_5$. By uniqueness, it follows that they agree identically.\qed

\subsubsection{Local curves of arbitrary genus}
Before turning to the proof of \cref{thm: 5 Hodge n point formula} for $n>1$, we generalize the formula obtained above to local curves with an arbitrary base. More precisely, we show that the equality between the all-genus Gromov--Witten series and the \smash{$\hat{\mathsf{A}}$}-genus, established above for rational local curves, extends to local Calabi--Yau curves over a smooth base of any genus.

Let $C$ be a smooth curve of genus $g_C \geq 0$. We consider the total space of four line bundles:
\begin{equation}
        \CYf = 
        \Tot_{C}\, \cL_2 \oplus \cL_3 \oplus \cL_4 \oplus \cL_5 \,.
\end{equation}
Let $\mgp{T}$ be a torus acting with weights $\epsilon_2,\ldots,\epsilon_5$ on the fibers of $\CYf$ while fixing the base. We prove the following formula for the Gromov--Witten series of this local fivefold.

\begin{customthm}{\ref{thm: intro local curve GW g formula}} 
    \label{thm: local curve GW g formula}
    Suppose \cref{conj:GW:vertex} holds for $n=1$. Then, for every Calabi--Yau local curve equipped with a fiber-wise Calabi--Yau torus action, we have
    \begin{equation}
        \label{eq: local curve GW g formula}
        \GWargs{\CYf}{\mgp{T}}{[C]}
        =
        \prod_{i=2}^5 \cs(\epsilon_i)^{g_C-1-\deg \cL_i}
        =
        \hat{\mathsf{A}}_{\mgp{T}}\Big( \Mbar_{g_C}\big(\CYf,[C]\big) \Big) \,.
    \end{equation}
\end{customthm}

The second equality is a simple generalization of \cref{lem: Ahat local curve formula}. We prove the first by a degeneration argument reducing the claim to the case of a rational base, where the required formula has already been established in \cref{cor: local curve GW g0 formula}. For this we need the following universal-factorization statement.

\begin{lem}
    \label{lem: local curve GW universal series}
    There exist universal series $A,B_2,\ldots,B_5$ in $\epsilon_2,\ldots,\epsilon_5$ such that
    \begin{equation}
        \GWargs{\CYf}{\mgp{T}}{[C]} = A^{g_C -1}\prod_{i=2}^5 B_i^{\deg \cL_i} \,.
    \end{equation}
    At this stage, we do not impose the Calabi--Yau condition.
\end{lem}

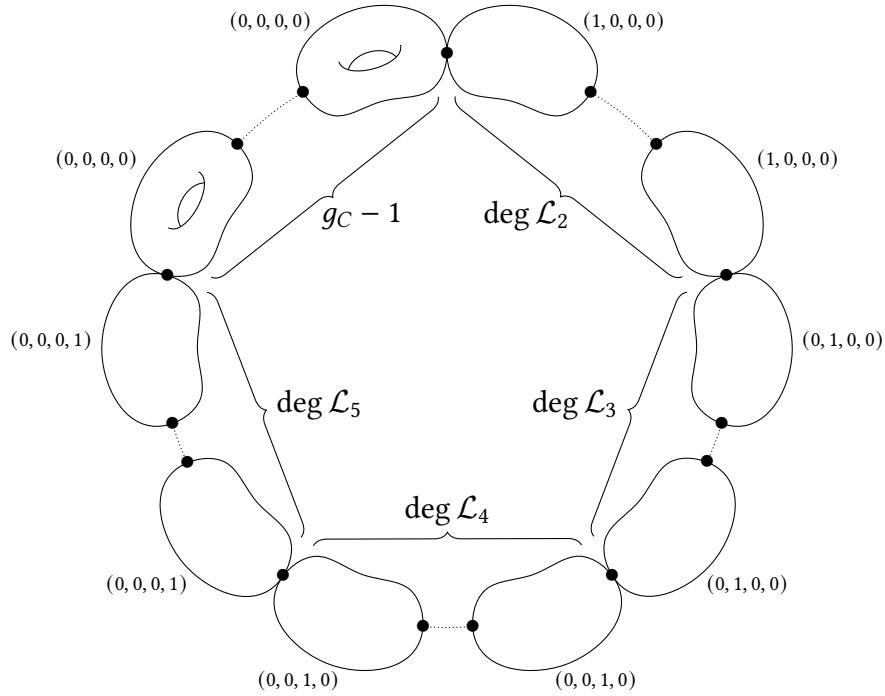
\begin{figure}
    \centering
    \begin{tikzpicture}[x=1pt,y=1pt,scale=.5]
        \draw[shift={(78.177, 660.931)}, rotate=61.6306](0, 0) .. controls (0, -16) and (48, -16) .. (48, 0);
        \draw[shift={(87.542, 661.811)}, rotate=61.6306](0, 0) .. controls (6.776, 11.823) and (30.776, 11.823) .. (37.433, -0.07);
        \draw[shift={(105.864, 630.826)}, rotate=61.6306](0, 0) .. controls (10.6667, -4) and (24, 1.3333) .. (37.3333, 1.3333) .. controls (50.6667, 1.3333) and (64, -4) .. (74.6667, 0) .. controls (85.3333, 4) and (93.3333, 17.3333) .. (94.6667, 29.3333) .. controls (96, 41.3333) and (90.6667, 52) .. (80, 60) .. controls (69.3333, 68) and (53.3333, 73.3333) .. (37.3333, 73.3333) .. controls (21.3333, 73.3333) and (5.3333, 68) .. (-5.3333, 60) .. controls (-16, 52) and (-21.3333, 41.3333) .. (-20, 29.3333) .. controls (-18.6667, 17.3333) and (-10.6667, 4) .. cycle;
        \draw[shift={(206.623, 786.071)}, rotate=15.0598](0, 0) .. controls (0, -16) and (48, -16) .. (48, 0);
        \draw[shift={(213.7, 779.874)}, rotate=15.0598](0, 0) .. controls (6.776, 11.823) and (30.776, 11.823) .. (37.433, -0.07);
        \draw[shift={(203.794, 745.268)}, rotate=15.0598](0, 0) .. controls (10.6667, -4) and (24, 1.3333) .. (37.3333, 1.3333) .. controls (50.6667, 1.3333) and (64, -4) .. (74.6667, 0) .. controls (85.3333, 4) and (93.3333, 17.3333) .. (94.6667, 29.3333) .. controls (96, 41.3333) and (90.6667, 52) .. (80, 60) .. controls (69.3333, 68) and (53.3333, 73.3333) .. (37.3333, 73.3333) .. controls (21.3333, 73.3333) and (5.3333, 68) .. (-5.3333, 60) .. controls (-16, 52) and (-21.3333, 41.3333) .. (-20, 29.3333) .. controls (-18.6667, 17.3333) and (-10.6667, 4) .. cycle;
        \draw[shift={(300.043, 764.672)}, rotate=-15.0413](0, 0) .. controls (10.6667, -4) and (24, 1.3333) .. (37.3333, 1.3333) .. controls (50.6667, 1.3333) and (64, -4) .. (74.6667, 0) .. controls (85.3333, 4) and (93.3333, 17.3333) .. (94.6667, 29.3333) .. controls (96, 41.3333) and (90.6667, 52) .. (80, 60) .. controls (69.3333, 68) and (53.3333, 73.3333) .. (37.3333, 73.3333) .. controls (21.3333, 73.3333) and (5.3333, 68) .. (-5.3333, 60) .. controls (-16, 52) and (-21.3333, 41.3333) .. (-20, 29.3333) .. controls (-18.6667, 17.3333) and (-10.6667, 4) .. cycle;
        \draw[shift={(434.656, 696.528)}, rotate=-61.6295](0, 0) .. controls (10.6667, -4) and (24, 1.3333) .. (37.3333, 1.3333) .. controls (50.6667, 1.3333) and (64, -4) .. (74.6667, 0) .. controls (85.3333, 4) and (93.3333, 17.3333) .. (94.6667, 29.3333) .. controls (96, 41.3333) and (90.6667, 52) .. (80, 60) .. controls (69.3333, 68) and (53.3333, 73.3333) .. (37.3333, 73.3333) .. controls (21.3333, 73.3333) and (5.3333, 68) .. (-5.3333, 60) .. controls (-16, 52) and (-21.3333, 41.3333) .. (-20, 29.3333) .. controls (-18.6667, 17.3333) and (-10.6667, 4) .. cycle;
        \draw[shift={(102.562, 532.677)}, rotate=91.7378](0, 0) .. controls (10.6667, -4) and (24, 1.3333) .. (37.3333, 1.3333) .. controls (50.6667, 1.3333) and (64, -4) .. (74.6667, 0) .. controls (85.3333, 4) and (93.3333, 17.3333) .. (94.6667, 29.3333) .. controls (96, 41.3333) and (90.6667, 52) .. (80, 60) .. controls (69.3333, 68) and (53.3333, 73.3333) .. (37.3333, 73.3333) .. controls (21.3333, 73.3333) and (5.3333, 68) .. (-5.3333, 60) .. controls (-16, 52) and (-21.3333, 41.3333) .. (-20, 29.3333) .. controls (-18.6667, 17.3333) and (-10.6667, 4) .. cycle;
        \draw[shift={(169.641, 427.386)}, rotate=130.346](0, 0) .. controls (10.6667, -4) and (24, 1.3333) .. (37.3333, 1.3333) .. controls (50.6667, 1.3333) and (64, -4) .. (74.6667, 0) .. controls (85.3333, 4) and (93.3333, 17.3333) .. (94.6667, 29.3333) .. controls (96, 41.3333) and (90.6667, 52) .. (80, 60) .. controls (69.3333, 68) and (53.3333, 73.3333) .. (37.3333, 73.3333) .. controls (21.3333, 73.3333) and (5.3333, 68) .. (-5.3333, 60) .. controls (-16, 52) and (-21.3333, 41.3333) .. (-20, 29.3333) .. controls (-18.6667, 17.3333) and (-10.6667, 4) .. cycle;
        \draw[shift={(259.94, 388.676)}, rotate=160.4666](0, 0) .. controls (10.6667, -4) and (24, 1.3333) .. (37.3333, 1.3333) .. controls (50.6667, 1.3333) and (64, -4) .. (74.6667, 0) .. controls (85.3333, 4) and (93.3333, 17.3333) .. (94.6667, 29.3333) .. controls (96, 41.3333) and (90.6667, 52) .. (80, 60) .. controls (69.3333, 68) and (53.3333, 73.3333) .. (37.3333, 73.3333) .. controls (21.3333, 73.3333) and (5.3333, 68) .. (-5.3333, 60) .. controls (-16, 52) and (-21.3333, 41.3333) .. (-20, 29.3333) .. controls (-18.6667, 17.3333) and (-10.6667, 4) .. cycle;
        \draw[shift={(387.545, 413.689)}, rotate=-160.361](0, 0) .. controls (10.6667, -4) and (24, 1.3333) .. (37.3333, 1.3333) .. controls (50.6667, 1.3333) and (64, -4) .. (74.6667, 0) .. controls (85.3333, 4) and (93.3333, 17.3333) .. (94.6667, 29.3333) .. controls (96, 41.3333) and (90.6667, 52) .. (80, 60) .. controls (69.3333, 68) and (53.3333, 73.3333) .. (37.3333, 73.3333) .. controls (21.3333, 73.3333) and (5.3333, 68) .. (-5.3333, 60) .. controls (-16, 52) and (-21.3333, 41.3333) .. (-20, 29.3333) .. controls (-18.6667, 17.3333) and (-10.6667, 4) .. cycle;
        \draw[shift={(475.702, 607.31)}, rotate=-91.7378](0, 0) .. controls (10.6667, -4) and (24, 1.3333) .. (37.3333, 1.3333) .. controls (50.6667, 1.3333) and (64, -4) .. (74.6667, 0) .. controls (85.3333, 4) and (93.3333, 17.3333) .. (94.6667, 29.3333) .. controls (96, 41.3333) and (90.6667, 52) .. (80, 60) .. controls (69.3333, 68) and (53.3333, 73.3333) .. (37.3333, 73.3333) .. controls (21.3333, 73.3333) and (5.3333, 68) .. (-5.3333, 60) .. controls (-16, 52) and (-21.3333, 41.3333) .. (-20, 29.3333) .. controls (-18.6667, 17.3333) and (-10.6667, 4) .. cycle;
        \draw[shift={(455.677, 484.454)}, rotate=-130.2451](0, 0) .. controls (10.6667, -4) and (24, 1.3333) .. (37.3333, 1.3333) .. controls (50.6667, 1.3333) and (64, -4) .. (74.6667, 0) .. controls (85.3333, 4) and (93.3333, 17.3333) .. (94.6667, 29.3333) .. controls (96, 41.3333) and (90.6667, 52) .. (80, 60) .. controls (69.3333, 68) and (53.3333, 73.3333) .. (37.3333, 73.3333) .. controls (21.3333, 73.3333) and (5.3333, 68) .. (-5.3333, 60) .. controls (-16, 52) and (-21.3333, 41.3333) .. (-20, 29.3333) .. controls (-18.6667, 17.3333) and (-10.6667, 4) .. cycle;
        \draw[densely dotted](179.6095, 762.8355) arc[start angle=120.1197, end angle=136.8902, radius=216];
        \draw[densely dotted](483.6937, 484.5671) arc[start angle=-25.0431, end angle=-16.7831, radius=216];
        \draw[densely dotted](445.6872, 723.6169) arc[start angle=43.1108, end angle=59.8988, radius=216];
        \draw[densely dotted](81.2006, 513.63) arc[start angle=-163.2169, end angle=-154.8071, radius=216];
        \draw[densely dotted](269.8923, 360.7603) arc[start angle=-94.8089, end angle=-84.871, radius=216];
        \node at (77.649, 625.0761) {$\bullet$};
        \node at (164.602, 398.718) {$\bullet$};
        \node at (412.256, 399.3182) {$\bullet$};
        \node at (498.3451, 625.1012) {$\bullet$};
        \node at (288.0003, 792) {$\bullet$};
        \node at (130.3102, 723.6141) {$\bullet$};
        \node at (179.6095, 762.8355) {$\bullet$};
        \node at (396.3303, 762.8704) {$\bullet$};
        \node at (445.6872, 723.6169) {$\bullet$};
        \node at (494.7994, 513.63) {$\bullet$};
        \node at (483.6937, 484.5671) {$\bullet$};
        \node at (307.3101, 360.8649) {$\bullet$};
        \node at (81.2006, 513.63) {$\bullet$};
        \node at (92.546, 484.0559) {$\bullet$};
        \node at (269.8923, 360.7603) {$\bullet$};
        \draw[decorate,decoration={brace,amplitude=5pt,raise=-5pt,mirror}](116.603, 615.9883) -- (288.0006, 752) node[midway,xshift=11pt,yshift=-7pt] {$g_C-1$};
        \draw[decorate,decoration={brace,amplitude=5pt,raise=-5pt,mirror}](288.0006, 752) -- (459.3924, 616.0081) node[midway,xshift=-13pt,yshift=-7pt] {$\deg \cL_2$};
        \draw[decorate,decoration={brace,amplitude=5pt,raise=-5pt,mirror}](459.3924, 616.0081) -- (389.2453, 432.0369) node[midway,xshift=-20pt,yshift=2pt] {$\deg \cL_3$};
        \draw[decorate,decoration={brace,amplitude=5pt,raise=-5pt,mirror}](389.2453, 432.0369) -- (187.4534, 431.5479) node[midway,yshift=8pt] {$\deg \cL_4$};
        \draw[decorate,decoration={brace,amplitude=5pt,raise=-5pt,mirror}](187.4534, 431.5479) -- (116.603, 615.9883) node[midway,xshift=20pt,yshift=2pt] {$\deg \cL_5$};
        \node at (421, 816) {\tiny$(1,0,0,0)$};
        \node at (552, 712) {\tiny$(1,0,0,0)$};
        \node at (586, 576) {\tiny$(0,1,0,0)$};
        \node at (514, 392) {\tiny$(0,1,0,0)$};
        \node at (400, 320) {\tiny$(0,0,1,0)$};
        \node at (176, 320) {\tiny$(0,0,1,0)$};
        \node at (62, 392) {\tiny$(0,0,0,1)$};
        \node at (-10, 576) {\tiny$(0,0,0,1)$};
        \node at (24, 712) {\tiny$(0,0,0,0)$};
        \node at (155, 816) {\tiny$(0,0,0,0)$};
    \end{tikzpicture}
    \caption{Degeneration of a local curve. The brackets $(d_2,\ldots,d_5)$ indicate the line bundle degrees $(\deg \cL_2,\ldots,\deg \cL_5)|_{C_k}$ on each irreducible component $C_k$ of the degenerate curve.}
    \label{fig: degeneration}
\end{figure}

\begin{proof}
    Suppose first that $g_C>0$ and $\deg \cL_i \geq 0$ for all $i$. We degenerate the base curve $C$ into a necklace of rational and genus-one curves, as indicated in \cref{fig: degeneration}, and extend the line bundles $\cL_2,\ldots,\cL_5$ so that the levels, i.e. the degrees on each irreducible component, are as shown in the figure. Applying the degeneration formula for Gromov--Witten invariants \cite{Li02:degFormula} then yields the claimed factorization, with the series $A,B_2,\ldots,B_5$ identified with the generating series of the corresponding relative Gromov--Witten invariants attached to the components appearing in \cref{fig: degeneration}.

    The same universal formula also holds when $g_C=0$, since $A^{-1}$ can be identified with the generating series of relative Gromov--Witten invariants of two level $(0,0,0,0)$ caps (and one may then repeat the same argument with an open necklace in place of the closed one). This follows from the fact that for $g_C=1$ one has \smash{$\GWargs{C\times \Aaff{4}}{\mgp{T}}{[C]} = 1$} together with the relation obtained from the degeneration shown in \cref{fig:deg:elliptic}.
    \begin{figure}
        \begin{center}
        \begin{tikzpicture}[x=1pt,y=1pt,scale=.45]
            \fill[opacity=.1](596.69, 607.2301) arc[start angle=-109.0412, end angle=-26.5643, radius=34.6667] .. controls (629.0023, 616.1657) and (614.8967, 610.41) .. (596.69, 607.23) -- cycle;
            \fill[opacity=.1](448, 581.333) .. controls (469.333, 576) and (490.667, 576) .. (512, 581.333) .. controls (533.333, 586.667) and (554.667, 597.333) .. (565.333, 613.333) .. controls (543.111, 593.7777) and (504, 583.111) .. (448, 581.333) -- cycle;
            \fill[opacity=.1](340.69, 607.2301) arc[start angle=-109.0412, end angle=-26.5643, radius=34.6667] .. controls (373.0023, 616.1657) and (358.8967, 610.41) .. (340.69, 607.23) -- cycle;
            \fill[opacity=.1](466.727, 653.184) .. controls (454.2423, 652.3947) and (445.3333, 650.6667) .. (440, 648) .. controls (442.108, 644.838) and (444.772, 642.092) .. (447.845, 639.763) .. controls (451.906, 646.51) and (458.778, 650.901) .. (466.686, 653.173);
            \draw(394.6667, 613.3333) .. controls (405.3333, 597.3333) and (426.6667, 586.6667) .. (448, 581.3333) .. controls (469.3333, 576) and (490.6667, 576) .. (512, 581.3333) .. controls (533.3333, 586.6667) and (554.6667, 597.3333) .. (565.3333, 613.3333) .. controls (576, 629.3333) and (576, 650.6667) .. (565.3333, 666.6667) .. controls (554.6667, 682.6667) and (533.3333, 693.3333) .. (512, 698.6667) .. controls (490.6667, 704) and (469.3333, 704) .. (448, 698.6667) .. controls (426.6667, 693.3333) and (405.3333, 682.6667) .. (394.6667, 666.6667) .. controls (384, 650.6667) and (384, 629.3333) .. cycle;
            \draw(440, 648) .. controls (456, 624) and (504, 624) .. (520, 648);
            \draw(448, 640) .. controls (460, 660) and (500, 660) .. (512, 640);
            \draw(352, 640) circle[radius=34.6667];
            \draw(608, 640) circle[radius=34.6667];
            \node at (386.6667, 640) {$\bullet$};
            \node at (573.3333, 640) {$\bullet$};
            \node at (480, 720) {\small$(0,0,0,0)$};
            \node at (352, 688) {\small$(0,0,0,0)$};
            \node at (608, 688) {\small$(0,0,0,0)$};
            \fill[opacity=.1](0, 581.3333) .. controls (21.3333, 576) and (42.6667, 576) .. (64, 581.3333) .. controls (85.3333, 586.6667) and (106.6667, 597.3333) .. (117.3333, 613.3333) .. controls (95.1111, 593.7778) and (56, 583.1111) .. (0, 581.3333) -- cycle;
            \fill[opacity=.1](18.7267, 653.1845) .. controls (6.2422, 652.3948) and (-2.6667, 650.6667) .. (-8, 648) .. controls (-5.8919, 644.8378) and (-3.2282, 642.0923) .. (-0.1555, 639.7634) .. controls (3.9058, 646.5097) and (10.778, 650.9006) .. (18.6856, 653.1727);
            \draw(-53.3333, 613.3333) .. controls (-42.6667, 597.3333) and (-21.3333, 586.6667) .. (0, 581.3333) .. controls (21.3333, 576) and (42.6667, 576) .. (64, 581.3333) .. controls (85.3333, 586.6667) and (106.6667, 597.3333) .. (117.3333, 613.3333) .. controls (128, 629.3333) and (128, 650.6667) .. (117.3333, 666.6667) .. controls (106.6667, 682.6667) and (85.3333, 693.3333) .. (64, 698.6667) .. controls (42.6667, 704) and (21.3333, 704) .. (0, 698.6667) .. controls (-21.3333, 693.3333) and (-42.6667, 682.6667) .. (-53.3333, 666.6667) .. controls (-64, 650.6667) and (-64, 629.3333) .. cycle;
            \draw(-8, 648) .. controls (8, 624) and (56, 624) .. (72, 648);
            \draw(0, 640) .. controls (12, 660) and (52, 660) .. (64, 640);
            \node at (32, 720) {\small$(0,0,0,0)$};
            \draw[->](192, 640) -- (256, 640);
        \end{tikzpicture}
        \end{center}
        \caption{Degeneration of a local elliptic curve at level $(0,0,0,0)$ into a relative elliptic curve with two caps, all at level $(0,0,0,0)$.}
        \label{fig:deg:elliptic}
    \end{figure}
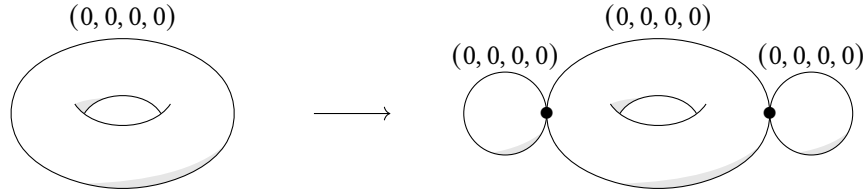
    
    Finally, the universal factorization extends to negative line bundle degrees by a similar degeneration argument involving the tube at level $(0,0,0,0)$.
\end{proof}

\begin{proof}[Proof of \cref{thm: local curve GW g formula}]
    To prove
    \begin{equation}
        \label{eq: local curve GW g formula proof}
        \GWargs{\CYf}{\mgp{T}}{[C]}
        =
        \prod_{i=2}^5 \cs(\epsilon_i)^{g_C-1-\deg \cL_i},
    \end{equation}
    we determine the universal series of \cref{lem: local curve GW universal series} under the Calabi--Yau conditions $\sum_{i} \epsilon_i = 0$ and $\sum_{i} \deg \cL_i = 2g_C-2$. We have already established the required formula for a rational base curve, assuming \cref{conj:GW:vertex}, in \cref{cor: local curve GW g0 formula}. In particular, it holds whenever the degrees are $(-2,0,0,0)$, or any permutation thereof. This gives
    \begin{equation}
        A^{-1} B_i^{-2}
        =
        \frac{\cs(\epsilon_i)^2}{\prod_{j=2}^5 \cs(\epsilon_j)}
    \end{equation}
    for every $i\in\{2,\ldots,5\}$. It follows from \cref{lem: local curve GW universal series} that \eqref{eq: local curve GW g formula proof} holds up to sign. The sign is then fixed by the leading term
    \begin{equation}
        \GWargs{\CYf}{\mgp{T}}{[C]}
        =
        \prod_{i=2}^5 \epsilon_i^{g_C-1-\deg \cL_i}
        + \cdots \,.
    \end{equation}
    This concludes the proof.
\end{proof}

\subsection{Pointed Hodge integrals for \texorpdfstring{$n>1$}{n>1}}
\label{sec: proof C n larger one}
We proceed with the proof of \cref{thm: 5 Hodge n point formula} when $n>1$. We will treat this case following a similar strategy as in the last section. Fix $n\in\{2,\ldots,5\}$, and let $\CYf$ be the partial compactification of the first $n$ coordinate lines of $\Aaff{5}$ to rational curves $C_1,\ldots,C_n$, obtained by adding a divisor $\Aaff{4}$ at infinity. As before, a torus $\mgp{T}$ acts with tangent weights $\epsilon_1,\ldots,\epsilon_5$ satisfying $\sum_i \epsilon_i=0$ at the origin. Our goal is to compare the conjectural identity of \cref{conj:GW:vertex}, namely 
\begin{equation}
\label{eq: GW compact vertex formula}
    \GWargs{\CYf}{\mgp{T}}{[C_1]+\cdots+[C_n]}
    \overset{?}{=}
    \hat{\mathsf{A}}_{\mgp{T}}\Big(
        \Mbar_0\big(\CYf,[C_1]+\cdots+[C_n]\big)
    \Big) \,,
\end{equation}
with the Hodge integral formula of \cref{thm: 5 Hodge n point formula}.

\subsubsection{Evaluating the \texorpdfstring{$\hat{\mathsf{A}}$}{Â}-genus}
The \smash{$\hat{\mathsf{A}}$}-genus appearing on the right-hand side of \labelcref{eq: GW compact vertex formula} is evaluated as follows.

\begin{prop}
    \label{prop: Ahat n greater 1}
    For $n\in \{2,\ldots,5\}$ we have
    \begin{equation}
        \hat{\mathsf{A}}_{\mgp{T}}\Big( \Mbar_{0}\big(\CYf,[C_1]+\cdots+[C_n]\big) \Big)
        =
        \vertex(\epsvec{[n]}) \,
        \prod_{j=1}^n \hat{\mathsf{A}}_{\mgp{T}}\Big(
            \Mbar_{0}\big(\Tot \, N_{C_j/\CYf}, [C_j]\big)
        \Big) \,,
    \end{equation}
    where the function $\vertex$ is defined as in \cref{eq: vertex function} (here, with $u = 1$):
    \begin{equation}
        \vertex(\epsvec{[n]})
        \coloneqq
        \begin{cases}
            \dfrac{\cs(\epsilon_3)\cs(\epsilon_4)\cs(\epsilon_5)}{\cs(\epsilon_1+\epsilon_2)}
            & n=2 \,, \\[0.95em]
            \cs(\epsilon_4)^2 \cs(\epsilon_5)^2
            & n=3 \,, \\[0.75em]
            \Big(\prod_{i=1}^5 \cs(\epsilon_i) \Big) \, \cs(\epsilon_5)^2 \sum_{i=1}^4 \kappa(\epsilon_i)
            & n=4 \,, \\[0.75em]
            \Big(\prod_{i=1}^5 \cs(\epsilon_i) \Big)^2 \Big(\big(\sum_{i=1}^5 \kappa(\epsilon_i)\big)^2 -\frac{1}{4} \Big)
            & n=5 \,,
        \end{cases}
    \end{equation}
    and we recall the notation $\cs(x)\coloneqq 2 \sinh \tfrac{x}{2}$ and $\kappa(x) \coloneqq \tfrac{1}{2} \coth \tfrac{x}{2}$.
\end{prop}

\begin{proof}
    For ease of notation, set $\beta_n \coloneqq [C_1]+\cdots+[C_n]$. We evaluate the $\hat{\mathsf{A}}$-genus by localization:
    \begin{equation}
        \hat{\mathsf{A}}_{\mgp{T}}\Big( \Mbar_{0}\big(\CYf,\beta_n\big) \Big)
        =
        \int_{[\Mbar_{0}(\CYf, \beta_n)^{\mgp{T}}]^{\vir}}
            \frac{\hat{\mathsf{A}}\Big( T^{\vir}\big|_{\Mbar_{0}(\CYf, \beta_n)^{\mgp{T}}} \Big)}{e_{\mgp{T}}\big( N^{\vir} \big)} \,.
    \end{equation}
    To describe the fixed locus, note that the domain of a $\mgp{T}$-fixed genus-zero stable map to $\CYf$ of class $\beta_n$ contains $n$ rational components, each mapped isomorphically to one of the compactified coordinate lines. If $n=2$, the two components are joined directly by a node mapping to the origin. If $n\geq 3$, the $n$ components are attached to a contracted $n$-pointed rational curve. Thus
    \begin{equation}
        \Mbar_{0}\big(\CYf,\beta_2 \big)^{\mgp{T}} \cong \mathrm{pt}
        \qquad\text{and}\qquad
        \Mbar_{0}\big(\CYf,\beta_n \big)^{\mgp{T}} \cong \Mbar_{0,n}
        \quad\text{for }n\ge 3 \,.
    \end{equation}
    Let us first consider the case $n=2$. The deformations and obstructions of the stable map $C_1 \cup C_2 \hookrightarrow \CYf$ decompose into the deformations and obstructions of the two irreducible components modulo deformations of the node mapping to the origin and its smoothing. Combining these contributions gives
    \begin{equation}
        \hat{\mathsf{A}}_{\mgp{T}}\Big( \Mbar_{0}\big(\CYf,\beta_2\big) \Big)
        =
        \frac{\cs(\epsilon_3) \, \cs(\epsilon_4) \, \cs(\epsilon_5)}{\cs(\epsilon_1+\epsilon_2)} \,
        \prod_{j=1}^2
        \hat{\mathsf{A}}_{\mgp{T}}\Big(
            \Mbar_{0}\big(\Tot \, N_{C_j/\CYf}, [C_j]\big)
        \Big) \,.
    \end{equation}
    The two factors involving $\Tot \, N_{C_j/\CYf}$ are the local-curve contributions of the two irreducible components. The remaining factors come from the node: $\prod_{k=3}^5 \cs(\epsilon_k)$ records deformations of its image in the three transverse directions at the origin, while $\cs(\epsilon_1+\epsilon_2)^{-1}$ is the contribution of the smoothing of the node.
    
    For $n\geq 3$ we argue in the same way. The deformation/obstruction theory splits into two parts: the contributions coming from the $n$ non-contracted components mapping isomorphically to the coordinate lines $C_j$, and the contribution coming from the $n$-pointed contracted component modulo contributions from the $n$ nodes attaching it to the $C_j$. Factoring out the first part gives
    \begin{equation}
        \hat{\mathsf{A}}_{\mgp{T}}\Big( \Mbar_{0}\big(\CYf,\beta_n\big) \Big)
        =
        \left(
            \int_{\Mbar_{0,n}}
                \frac{\hat{\mathsf{A}}\big( \tilde{T}^{\vir} \big)}{e_{\mgp{T}}\big( \tilde{N}^{\vir} \big)}
        \right)
        \prod_{j=1}^n
            \hat{\mathsf{A}}_{\mgp{T}}\Big(
                \Mbar_{0}\big(\Tot \, N_{C_j/\CYf}, [C_j]\big)
            \Big)
        \,.
    \end{equation}
    For $n=3$, the moduli space $\Mbar_{0,3}$ is a point, so the remaining factor is simply the contribution of the contracted component, modulo the contributions of the three nodes and their smoothings. The contracted component contributes \smash{$\prod_{i=1}^5 \cs(\epsilon_i)^{-1}$}, corresponding to its five tangent directions in $\CYf$. For each of the three nodes, the deformations of its image in the directions transverse to the corresponding branch $C_j$ contribute \smash{$\prod_{k\neq j} \cs(\epsilon_k)$}, while smoothing the node contributes \smash{$\cs(\epsilon_j)^{-1}$}. Altogether, this gives
    \begin{equation}
        \int_{\Mbar_{0,n}}
            \frac{\hat{\mathsf{A}}\big( \tilde{T}^{\vir} \big)}{e_{\mgp{T}}\big( \tilde{N}^{\vir} \big)}
        =
        \frac{1}{\prod_{i=1}^5 \cs(\epsilon_i)}
        \Bigg(
            \prod_{j=1}^3 \prod_{k\neq j} \cs(\epsilon_k)
        \Bigg)\Bigg(
            \prod_{j=1}^3 \frac{1}{\cs(\epsilon_j)}
         \Bigg)
         =
         \cs(\epsilon_4)^2 \cs(\epsilon_5)^2 \,.
    \end{equation}
    For $n\in\{4,5\}$, the contributions of the non-contracted components and of the image of each node are unchanged. The difference is that the central contracted component now has non-trivial moduli, so the contribution of the node smoothings must be integrated over \smash{$\Mbar_{0,n}$}. More precisely, we obtain
    \begin{equation}
        \int_{\Mbar_{0,n}}
            \frac{\hat{\mathsf{A}}\big( \tilde{T}^{\vir} \big)}{e_{\mgp{T}}\big( \tilde{N}^{\vir} \big)}
        =
        \frac{1}{\prod_{i=1}^5 \cs(\epsilon_i)}
        \Bigg(
            \prod_{j=1}^n \prod_{k\neq j} \cs(\epsilon_k)
        \Bigg)
        \int_{\Mbar_{0,n}}
            \frac{
                \hat{\mathsf{A}}\big(T_{\Mbar_{0,n}}\big)
            }{
                \prod_{j=1}^n \cs(\epsilon_j-\psi_j)
            } \,.
    \end{equation}
    To evaluate the last integral, note that for $n=4$ we have \smash{$\hat{\mathsf{A}}\big(T_{\Mbar_{0,4}}\big)=1$}, since $\Mbar_{0,4}$ is one-dimensional. Together with the expansion
    \begin{equation}
        \frac{1}{\cs(\epsilon_j-\psi_j)}
        =
        \frac{1}{\cs(\epsilon_j)}
        +
        \frac{\kappa(\epsilon_j)}{\cs(\epsilon_j)} \, \psi_j
        \in \hhh^{\bullet}(\Mbar_{0,4}) \,,
    \end{equation}
    this immediately gives
    \begin{equation}
        \int_{\Mbar_{0,4}}
            \frac{
                \hat{\mathsf{A}}\big(T_{\Mbar_{0,n}}\big)
            }{
                \prod_{j=1}^n \cs(\epsilon_j-\psi_j)
            }
        =
        \frac{\sum_{j=1}^4 \kappa(\epsilon_j)}{\prod_{j=1}^4 \cs(\epsilon_j)}  \,,
    \end{equation}
    which proves the claimed formula for $n=4$. The case $n=5$ is similar. One now has to expand $\cs(\epsilon_j-\psi_j)^{-1}$ to second order, and there is in addition a non-trivial contribution from\footnote{
        The $\hat{\mathsf{A}}$-class of $\Mbar_{0,5}$ can be computed as follows: since $M = \Mbar_{0,5}$ is a smooth surface, one has $\hat{\mathsf{A}}(T_{M})=1-\frac{1}{24}( c_1(T_{M})^2-2c_2(T_{M}) )$. Moreover, $\Mbar_{0,5}$ is isomorphic to the blow-up of $\bP^2$ at four points in general position, hence $c_1(T_{M})^2=5[\mathrm{pt}]$ and $c_2(T_{M})=7[\mathrm{pt}]$. Therefore $\hat{\mathsf{A}}(T_{M})=1+\frac{3}{8}[\mathrm{pt}]$.
    } \smash{$\hat{\mathsf{A}}(T_{\Mbar_{0,5}})=1+\frac{3}{8}[\mathrm{pt}]$}. We leave the remaining details to the reader.
\end{proof}

\subsubsection{Proof of \texorpdfstring{\cref{thm: 5 Hodge n point formula}}{the equivalence}}
For $n=1$, the equivalence has already been established in \cref{sec: membrane moduli n 1}. We therefore assume $n>1$, and set
\begin{equation}
    \beta_n \coloneqq [C_1]+\cdots+[C_n] \,,
    \qquad
    N_{C_j/\CYf} \cong \bigoplus_{i \ne j} \cL_{i,j} \,,
    \qquad
    d_{i,j} \coloneqq \deg \cL_{i,j} \,.
\end{equation}
Applying virtual localization to the all-genus Gromov--Witten series of $\CYf$ in curve class $\beta_n$, we obtain
\begin{equation}
    \GWargs{\CYf}{\mgp{T}}{\beta_n}
    =
    \HodgeH\big(\epsvec[-1]{[n]};\epsvec[0]{[5]};u\big)
    \prod_{j=1}^n
    \left(
        \frac{u^2}{
            \prod_{i \ne j} \prod_{k=1}^{d_{i,j}-1} (\epsilon_i - k \epsilon_j)
        }
        \,
        \HodgeH(-\epsilon_j^{-1};\epsvec[\infty_j]{[5]};u)
    \right) \,.
\end{equation}
Here \smash{$\epsvec[0]{[5]}=(\epsilon_1,\ldots,\epsilon_5)$} and \smash{$\epsvec[\infty_j]{[5]}=(\epsilon_1^{\infty_j},\ldots,\epsilon_5^{\infty_j})$} denote the tangent weights at the origin $0\in\Aaff{5}\subset \CYf$ and at the fixed point $\infty_j\in C_j$, respectively. The latter is given by
\begin{equation}
    \epsilon_i^{\infty_j} =
    \begin{cases}
        -\epsilon_j &\text{if } i=j \,,\\
        \epsilon_i - d_{i,j}\epsilon_j &\text{if } i\ne j \,.
    \end{cases}
\end{equation}
As before, the first Hodge series is the contribution of the contracted component mapped to the origin, while for each $j$ the second Hodge series is the contribution of the contracted component mapped to the fixed point $\infty_j\in C_j$. The unstable cases are understood via the conventions introduced in \cref{eq: Hodge unstable}. The remaining product of torus weights records the deformations and obstructions of the non-contracted rational components mapping isomorphically to the compactified coordinate lines.

Since the case $n=1$ is already known, we may use \cref{thm: 5 Hodge n point formula} for $n=1$ to rewrite the Hodge series contributions from the points $\infty_j$. We thus obtain
\begin{equation}
\begin{split}
    \GWargs{\CYf}{\mgp{T}}{\beta_n}
    &=
    \HodgeH\big(\epsvec[-1]{[n]};\epsvec{[5]};u\big)
    \prod_{j=1}^n \prod_{i \ne j}
        \frac{
            \gamfct{\epsilon_j^{-1}\epsilon_i - d_{i,j}}{u\epsilon_j}
        }{
            \prod_{k=1}^{d_{i,j}} (\epsilon_i - k \epsilon_j)
        } \\
    &=
    \HodgeH\big(\epsvec[-1]{[n]};\epsvec{[5]};u\big)
    \prod_{j=1}^n
        u^2
        \prod_{i \ne j} \frac{\epsilon_i}{\gamfct{-\epsilon_j^{-1}\epsilon_i}{u\epsilon_j}}
        \prod_{k=0}^{d_{i,j}} \cs\bigl( u(\epsilon_i - k \epsilon_j) \bigr)^{-1}
     \,,
\end{split}
\end{equation}
where in the second equality we used the reflection formula from \cref{lem: gamma fct reflect} together with the Calabi--Yau condition $\sum_{i \neq j} d_{i,j} = -2$. By \cref{lem: Ahat local curve formula}, the product of $\cs$-functions is precisely the \smash{$\hat A$}-genus of the moduli space of genus-zero stable maps to the local curve $\Tot \, N_{C_j/\CYf}$ in class $[C_j]$:
\begin{equation}
    \hat{\mathsf{A}}_{\mgp{T}}\Big(
        \Mbar_{0}\big(\Tot \, N_{C_j/\CYf}, [C_j]\big)
    \Big)
    =
    \prod_{i \ne j}
    \prod_{k=0}^{d_{i,j}} \cs\bigl( u(\epsilon_i - k \epsilon_j) \bigr)^{-1} \,.
\end{equation}
Combining this with \cref{prop: Ahat n greater 1}, we conclude that $\GWargs{\CYf}{\mgp{T}}{\beta_n} = \hat{\mathsf{A}}_{\mgp{T}}\big( \Mbar_{0}(\CYf,\beta_n) \big)$ holds if and only if
\begin{equation}
    \HodgeH\big(\epsvec[-1]{[n]};\epsvec{[5]};u\big)
    =
    \vertex\big(\epsvec{[5]};u\big) \,
    \prod_{j=1}^n
        \frac{1}{u^2} \prod_{i \ne j} \frac{\gamfct{-\epsilon_j^{-1}\epsilon_i}{u\epsilon_j}}{\epsilon_i} \,,
\end{equation}
which is precisely the statement of \cref{thm: 5 Hodge n point formula} for $n > 1$. This completes the proof. \qed

\section{Constant maps}
\label{sec: const maps main part}

Recall that \cref{conj:index} predicts that the constant-map contribution to the equivariant Gromov--Witten theory of a Calabi--Yau fivefold is governed by the corresponding supergravity index. The main result of this section is the proof of \cref{thm: five hodge no psi}, namely that \cref{conj:index} is equivalent to the closed formula \eqref{eq: 5 lambda 0 psi formula} for unpointed quintuple Hodge integrals. The proof is based on the asymptotic expansion of the supergravity index after the change of variables $q_i=\re^{u\epsilon_i}$, which is then compared with the conjectural Hodge series.

We first derive the asymptotic expansion of the supergravity index of affine five-space by Mellin transform methods. We then use this expansion to prove the equivalence with the unpointed Hodge formula and finally explain how the latter leads, via Hirzebruch--Riemann--Roch, to a formula for constant-map contributions on a general Calabi--Yau fivefold.

\subsection{Asymptotics of the supergravity index}
We begin with a discussion of the local model $\CYf=\Aaff{5}$ equipped with an action of a torus $\mgp{T}$ with tangent weights $\epsilon_1,\ldots,\epsilon_5$, satisfying the Calabi--Yau condition $\sum_{i=1}^5 \epsilon_i=0$. The supergravity index 
appearing in \cref{conj:index} reads
\begin{equation}
    \label{eq: sugra index local model}
    \Exp\Big( \chi_{\mgp{T}}\big(
        \Aaff{5},
        T_{\Aaff{5}}^\vee -T_{\Aaff{5}}
    \big) \Big) \,.
\end{equation}
Here $\Exp$ denotes the plethystic exponential and the $\mgp{T}$-equivariant Euler characteristic in its argument is defined as a K-theoretic residue. Thus, writing $q_i=\re^{u\epsilon_i}$ for the irreducible characters of $T_{\Aaff{5}}|_0$, the index equates to
\begin{equation}
    \Exp\left(
        - \frac{\sum_{i=1}^5 (q_i-q_i^{-1})}{\prod_{i=1}^5 (q_i^{1/2}-q_i^{-1/2})}
    \right)\,,
\end{equation}
where we used the Calabi--Yau condition $\prod_{i=1}^5 q_i=1$.

Our goal is to determine the asymptotic expansion of this quantity as $u\to 0$. This is motivated by \cref{conj:index} which claims that the doubled generating series of disconnected constant maps $\exp( 2\GWargs{\CYf}{\mgp{T}}{0})$ coincides with the positive-power part of this expansion. We will in fact compute the non-positive terms as well, since they are of independent physical interest.

\begin{prop}
    \label{prop: sugra index asymptotics}
    Suppose that $\epsilon_1,\ldots,\epsilon_5$ are generic equivariant parameters satisfying the Calabi--Yau condition $\sum_i \epsilon_i = 0$.
    Then, as $u\to 0$, one has the asymptotic expansion
    \begin{multline}
        \label{eq: sugra asymptotic expansion}
        \log \Exp\left(
            - \frac{\sum_{i=1}^5 (q_i-q_i^{-1})}{\prod_{i=1}^5 (q_i^{1/2}-q_i^{-1/2})}
        \right) \\
        \sim
        - \zeta(3)\,\frac{p_3}{3 e_5}\frac{1}{u^2}
        + \frac{1}{12}\log(u)
        - c
        +
        \sum_{g > 1} \frac{B_{2g-2}}{2g-2}
        \left[
            \frac{\sum_{i=1}^5 \cs(2u\epsilon_i)}{\prod_{i=1}^5 \cs(u\epsilon_i)}
        \right]_{2g-2}
        u^{2g-2} \,,
    \end{multline}
    where $q_i=\re^{u\epsilon_i}$, $p_3=\sum_{i=1}^5\epsilon_i^3$, $e_5=\prod_{i=1}^5\epsilon_i$, and
    \begin{equation}
        c
        =
        \log \prod_{i=1}^5
        \frac{
        	\Gamma_4(-\epsilon_i;\epsvec{\hat{\imath}})
        }{
        	\rho_4(\epsvec{\hat{\imath}})
        } \,.
    \end{equation}
    Here $\epsvec{\hat{\imath}}=(\epsilon_1,\ldots,\widehat{\epsilon_i},\ldots,\epsilon_5)$, and $\Gamma_4$, respectively $\rho_4$, denote the $4$-fold Barnes Gamma function and the corresponding modular constant. (See \eqref{eq: Barnes zeta} and \eqref{eq: Barnes Gamma and modular constant}.)
\end{prop}

\Cref{conj:index} asserts that the positive-power part of this asymptotic expansion equates to twice the generating series of constant maps to affine five-space. Thus, \cref{prop: sugra index asymptotics} provides the following equivalent characterization of the \namecref{conj:index}.

\begin{cor}
    \label{cor: const maps loc model}
    For $\CYf=\Aaff{5}$, \cref{conj:index} is equivalent to the identity
    \begin{equation}
        \label{eq: 5 lambda 0 psi local model}
        \GWargs{\Aaff{5}}{\mgp{T}}{0}
        =
        \frac{1}{2}
        \sum_{g>1}
            \frac{B_{2g-2}}{2g-2}
            \left[
                \frac{\sum_{i=1}^5 \cs(2u\epsilon_i)}{\prod_{i=1}^5 \cs(u\epsilon_i)}
            \right]_{2g-2}
            u^{2g-2} \,,
    \end{equation}
    where $\sum_{i=1}^5\epsilon_i=0$. \qed
\end{cor}

Before proving \cref{prop: sugra index asymptotics}, which is based on an adaptation of the Mellin-transform method \cite[\S4.1.1]{PK01:Asymp}, let us first explain its structure. The asymptotic expansion is first established in a real convergence chamber and then extended meromorphically to generic equivariant parameters. The first three terms in \eqref{eq: sugra asymptotic expansion} arise from the poles of the Mellin transform $\Phi(s)$ at $s=2$ and $s=0$, corresponding to the unstable topologies $g=0$ and $g=1$, respectively. The remaining series, which is the positive-power part at $u=0$, is governed by the poles at the negative even integers and therefore corresponds to the stable topologies $g>1$. Its coefficients are given by the residues of the Mellin transform at these poles, hence are meromorphic in the equivariant parameters and, in particular, independent of the choice of convergence chamber after analytic continuation. This positive part is exactly of the form predicted by \cref{eq: 5 lambda 0 psi formula}, and it will match the generating series of unpointed quintuple Hodge integrals.

\begin{proof}
	We first establish the asymptotic expansion in the real chamber
	\begin{equation}
	    \label{eq:chamber}
	    \epsilon_1,\ldots,\epsilon_5\in\bR^{*} \,,
	    \qquad
	    \sum_{i=1}^5 \epsilon_i=0 \,,
	    \qquad
	    \max_i |\epsilon_i|
	    <
	    \frac12 \sum_{i=1}^5 |\epsilon_i| \,,
	    \qquad
	    u \to 0^+ \,,
	\end{equation}
	and then extend its coefficients meromorphically to generic equivariant parameters. Set
	\begin{equation}
	    \chi(u)
	    \coloneqq
	    \log \Exp\left(
	        - \frac{\sum_{i=1}^5 (q_i-q_i^{-1})}{\prod_{i=1}^5 (q_i^{1/2}-q_i^{-1/2})}
	    \right)
	    =
	    - \sum_{m \ge 1} \frac{1}{m} F(mu) \,,
	    \qquad
	    F(t) \coloneqq \frac{\sum_{i=1}^5 \cs(2t\epsilon_i)}{\prod_{i=1}^5 \cs(t\epsilon_i)} \,.
	\end{equation}
	Since $\sum_{i} \epsilon_i=0$, the function $F$ is even and meromorphic in $t$, with at worst a double pole at $t=0$. Its Laurent expansion is of the form
	\begin{equation}
	    F(t)=\frac{a_0}{t^2}+a_1+\sum_{g>1} a_g\,t^{2g-2} \,,
	\end{equation}
	where
	\begin{equation}
	    a_0=\frac{p_3}{3e_5} \,,
	    \qquad
	    a_1=\frac{1}{12} \,,
	    \qquad
	    a_g=
	    \left[
	        \frac{\sum_{i=1}^5 \cs(2t\epsilon_i)}{\prod_{i=1}^5 \cs(t\epsilon_i)}
	    \right]_{2g-2}
	    \quad\text{for }g>1 \,.
	\end{equation}
	The chamber assumption \eqref{eq:chamber} implies exponential decay of $F$ along the positive real axis. Indeed, as $t \to + \infty$ one has $|\cs(t\epsilon_i)| \asymp \re^{t|\epsilon_i|/2}$ and $|\cs(2t\epsilon_i)| \asymp \re^{t|\epsilon_i|}$, so that
	\begin{equation}
	    |F(t)|
	    \ll
	    \exp\left(
	        t\max_i |\epsilon_i| - \frac{t}{2}\sum_{i=1}^5 |\epsilon_i|
	    \right) \,.
	\end{equation}
	By assumption, $\delta \coloneqq \frac{1}{2} \sum_{i} |\epsilon_i|-\max_i |\epsilon_i| >0$, and therefore $F(t)=\mathrm{O}(\re^{-\delta t})$ as $t\to+\infty$. In particular, the series defining $\chi(u)$ converges absolutely for every sufficiently small $u>0$, and the Mellin transform
	\begin{equation}
	    \Phi(s)
	    \coloneqq
	    \int_0^\infty t^{s-1}F(t)\,dt \,,
	\end{equation}
	is convergent for $\Re(s)>2$.
    By Lemma~3.3 and Equation~(3.2.4) of \cite{PK01:Asymp}, the Laurent expansion of $F(t)$ at $t=0$ yields a meromorphic continuation of $\Phi(s)$ to $\bC$, with simple poles at $s=2-2g$ for $g\ge 0$ and residue $a_g$. 
	Since $\Phi(s)$ is initially defined for $\Re(s)>2$, Mellin inversion gives, for any $c>2$,
	\begin{equation}
	    F(mu)
	    =
	    \frac{1}{2\pi \ri}
	    \int_{c-\ri\infty}^{c+\ri\infty}
	    (mu)^{-s}\Phi(s)\,ds \,.
	\end{equation}
	Substituting this into the definition of $\chi(u)$ and interchanging summation and integration, which is justified by absolute convergence,
    we obtain
	\begin{equation}
	    \chi(u)
	    =
	    - \frac{1}{2\pi \ri}
	    \int_{c-\ri\infty}^{c+\ri\infty}
	    u^{-s}\zeta(s+1)\Phi(s)\,ds \,.
	\end{equation}
    We now shift the contour to the left. After subtracting finitely many terms of the asymptotic expansion of $F(t)$ at $t=0$, repeated integration by parts in the corresponding Mellin transform gives rapid decay of $\Phi(s)$ in vertical strips, away from its poles. Together with the standard polynomial growth bounds for $\zeta(s+1)$, this implies that the horizontal contributions vanish. The asymptotic expansion of $\chi(u)$ is therefore obtained from the residues of $u^{-s}\zeta(s+1)\Phi(s)$ at the poles crossed during the shift. The pole at $s=2$ contributes
	\begin{equation}
	    -\zeta(3)a_0\,u^{-2}
	    =
	    -\zeta(3)\frac{p_3}{3e_5}\frac{1}{u^2} \,.
	\end{equation}
	At $s=0$, both $\zeta(s+1)$ and $\Phi(s)$ have simple poles. Writing
	\begin{equation}
	    u^{-s}=1-s\log(u)+\mathrm{O}(s^2) \,,
	    \quad
	    \zeta(s+1)=\frac{1}{s}+\gamma+\mathrm{O}(s) \,,
	    \quad
	    \Phi(s)=\frac{a_1}{s}+\Phi_0+\mathrm{O}(s) \,,
	\end{equation}
	one finds that the residue at $s=0$ is
	\begin{equation}
	    a_1(\log(u)-\gamma)-\Phi_0
	    =
	    \frac{1}{12}\log(u)-\left(
	    	\Phi_0+\frac{\gamma}{12}
	    \right) \,.
	\end{equation}
	Finally, for $g>1$, the pole at $s=2-2g$ contributes
	\begin{equation}
	    -\zeta(3-2g)a_g\,u^{2g-2}
	    =
	    \frac{B_{2g-2}}{2g-2}
	    \left[
	        \frac{\sum_{i=1}^5 \cs(2\epsilon_i)}{\prod_{i=1}^5 \cs(\epsilon_i)}
	    \right]_{2g-2}
	    u^{2g-2} \,.
	\end{equation}
	Hence
	\begin{equation}
	    \chi(u)
	    \sim
	    - \frac{p_3 \zeta(3)}{3 e_5}\frac{1}{u^2}
	    + \frac{1}{12}\log(u)
	    - \left(\Phi_0+\frac{\gamma}{12}\right)
	    +
	    \sum_{g > 1} \frac{B_{2g-2}}{2g-2}
	    \left[
	        \frac{\sum_{i=1}^5 \cs(2\epsilon_i)}{\prod_{i=1}^5 \cs(\epsilon_i)}
	    \right]_{2g-2}
	    u^{2g-2} \,.
	\end{equation}
    This proves \eqref{eq: sugra asymptotic expansion} up to the explicit identification of the constant $c=\Phi_0+\frac{\gamma}{12}$, where $\Phi_0$ is the finite part of the Mellin transform $\Phi$ at $s=0$.
    
    To compute $\Phi_0$, we decompose $F$ into five terms whose Mellin transforms can be expressed in terms of Barnes zeta functions. Since the individual terms do not decay at both $0$ and $+\infty$, we first regularize their Mellin transforms by subtracting the relevant asymptotic terms. After summing over the five contributions, the subtraction terms recombine precisely into the meromorphic continuation of $\Phi$. Expanding the resulting Barnes-zeta expression at $s=0$ will then identify $\Phi_0$, and hence $c$. More precisely, the Calabi--Yau condition gives
	\begin{equation}
	    F(t)
	    =
	    \frac{\sum_{i=1}^5(\re^{t\epsilon_i}-\re^{-t\epsilon_i})}{\prod_{i=1}^5(1-\re^{-t\epsilon_i})}
	    =
	    \sum_{i=1}^5 \frac{1+\re^{t\epsilon_i}}{\prod_{j\neq i}(1-\re^{-t\epsilon_j})} \,.
	\end{equation}
	For each $i$, set
	\begin{equation}
	    \Pi_i(t)\coloneqq \prod_{j\neq i}(1-\re^{-t\epsilon_j}),
	    \qquad
	    f_i(t)\coloneqq \frac{1+\re^{t\epsilon_i}}{\Pi_i(t)},
	    \qquad
	    \epsvec{\hat{\imath}}\coloneqq (\epsilon_1,\ldots,\widehat{\epsilon_i},\ldots,\epsilon_5) \,.
	\end{equation}
    Then $F(t)=\sum_{i=1}^5 f_i(t)$. Moreover, as $t\to0$, we have $f_i(t)=\sum_{m=-4}^0 b_{i,m}t^m+\mathrm{O}(t)$ for suitable coefficients $b_{i,m}$. To identify the Mellin transform of each $f_i$ with Barnes zeta functions, we temporarily work in a convergence chamber for the corresponding Barnes integrals, where $f_i(t)\to 1$ as $t\to+\infty$, and then extend the resulting identities meromorphically in the equivariant parameters. We therefore define the endpoint-subtracted Mellin transform of $f_i$ by
	\begin{equation}
	    \varphi_i(s)
	    \coloneqq
	    \int_0^1
	    t^{s-1}
	    \left(
	        f_i(t)-\sum_{m=-4}^0 b_{i,m}t^m
	    \right)\,dt
	    +
	    \sum_{m=-4}^0 \frac{b_{i,m}}{s+m}
	    +
	    \int_1^\infty
	    t^{s-1}\bigl(f_i(t)-1\bigr)\,dt
	    -\frac{1}{s} \,.
	\end{equation}
	Next, define the ordinary and homogeneous Barnes zeta functions by
	\begin{equation}
        \label{eq: Barnes zeta}
	    \zeta_4(s,-\epsilon_i;\epsvec{\hat{\imath}})
	    \coloneqq
	    \frac{1}{\Gamma(s)}
	    \int_0^\infty
	    t^{s-1}
	    \frac{\re^{t\epsilon_i}}{\Pi_i(t)}\,dt,
	    \quad
	    \zeta_{4,h}(s;\epsvec{\hat{\imath}})
	    \coloneqq
	    \frac{1}{\Gamma(s)}
	    \int_0^\infty
	    t^{s-1}
	    \left(
	        \frac{1}{\Pi_i(t)}-1
	    \right)\,dt \,,
	\end{equation}
	whenever the integrals converge, with meromorphic continuation elsewhere. Since $f_i(t)-1 = ( \frac{1}{\Pi_i(t)}-1 ) + \frac{\re^{t\epsilon_i}}{\Pi_i(t)}$, we obtain, initially in a convergence chamber for the corresponding Barnes integrals and hence, by meromorphic continuation in the equivariant parameters,
	\begin{multline}
	    \varphi_i(s)
	    =
	    \Gamma(s)\Bigl(
	        \zeta_{4,h}(s;\epsvec{\hat{\imath}})
	        +
	        \zeta_4(s,-\epsilon_i;\epsvec{\hat{\imath}})
	    \Bigr) \\
	    +
	    \int_0^1
	    t^{s-1}
	    \left(
	        1-\sum_{m=-4}^0 b_{i,m}t^m
	    \right)\,dt
	    +
	    \sum_{m=-4}^0 \frac{b_{i,m}}{s+m}
	    -\frac{1}{s} \,.
	\end{multline}
	But
	\begin{equation}
	    \int_0^1
	    t^{s-1}
	    \left(
	        1-\sum_{m=-4}^0 b_{i,m}t^m
	    \right)\,dt
	    =
	    \frac{1}{s}
	    -
	    \sum_{m=-4}^0 \frac{b_{i,m}}{s+m} \,,
	\end{equation}
	so the correction terms cancel identically, and $\varphi_i(s) = \Gamma(s)( \zeta_{4,h}(s;\epsvec{\hat{\imath}}) + \zeta_4(s,-\epsilon_i;\epsvec{\hat{\imath}}) )$ as meromorphic functions of $s$.
    
    By meromorphic continuation in the equivariant parameters, we may now return to the chamber \eqref{eq:chamber}. In this chamber, each $f_i(t)$ decays exponentially as $t\to+\infty$, so the endpoint-subtracted expression defining $\varphi_i(s)$ is convergent for $-1<\Re(s)<0$. Since $F(t)=\sum_{i=1}^5 f_i(t)$, one finds
	\begin{multline}
	    \sum_{i=1}^5 \varphi_i(s)
	    =
	    \int_0^1
	    t^{s-1}
	    \left(
	        F(t)-\sum_{i=1}^5\sum_{m=-4}^0 b_{i,m}t^m
	    \right)\,dt \\
	    +
	    \sum_{i=1}^5\sum_{m=-4}^0 \frac{b_{i,m}}{s+m}
	    +
	    \int_1^\infty t^{s-1}\bigl(F(t)-5\bigr)\,dt
	    -\frac{5}{s} \,.
	\end{multline}
	By comparison with the Laurent expansion of $F$ at $t=0$, the coefficients satisfy
	\begin{equation}
	    \sum_{i=1}^5 b_{i,-4}
	    =
	    \sum_{i=1}^5 b_{i,-3}
	    =
	    \sum_{i=1}^5 b_{i,-1}
	    =0,
	    \qquad
	    \sum_{i=1}^5 b_{i,-2}=a_0,
	    \qquad
	    \sum_{i=1}^5 b_{i,0}=a_1 \,.
	\end{equation}
	Hence
	\begin{equation}
	    \sum_{i=1}^5 \varphi_i(s)
	    =
	    \int_0^1
	    t^{s-1}
	    \bigl(F(t)-a_0t^{-2}-a_1\bigr)\,dt
	    +
	    \frac{a_0}{s-2}
	    +
	    \frac{a_1}{s}
	    +
	    \int_1^\infty t^{s-1}\bigl(F(t)-5\bigr)\,dt
	    -\frac{5}{s} \,.
	\end{equation}
	For $-1<\Re(s)<0$, the last two terms combine to
	\begin{equation}
	    \int_1^\infty t^{s-1}\bigl(F(t)-5\bigr)\,dt
	    -\frac{5}{s}
	    =
	    \int_1^\infty t^{s-1}F(t)\,dt \,,
	\end{equation}
	since $\int_1^\infty t^{s-1}\,dt=-\frac{1}{s}$ in this strip. Therefore, on the non-empty strip $-1<\Re(s)<0$, one has
	\begin{equation}
	    \sum_{i=1}^5 \varphi_i(s)
	    =
	    \frac{a_0}{s-2}
	    +
	    \frac{a_1}{s}
	    +
	    \int_0^1
	    t^{s-1}\bigl(F(t)-a_0t^{-2}-a_1\bigr)\,dt
	    +
	    \int_1^\infty t^{s-1}F(t)\,dt \,.
	\end{equation}
	By the standard endpoint-subtracted Mellin continuation \cite[Eq.~(3.2.4)]{PK01:Asymp}, the right-hand side is precisely the meromorphic continuation of $\Phi(s)$. Hence $\sum_{i=1}^5 \varphi_i(s)=\Phi(s)$. Combining this with the previous identity for $\varphi_i(s)$ yields
	\begin{equation}
	    \Phi(s)=\Gamma(s)\,\Psi(s),
	    \qquad
	    \Psi(s)
	    \coloneqq
	    \sum_{i=1}^5
	    \Bigl(
	        \zeta_{4,h}(s;\epsvec{\hat{\imath}})
	        +
	        \zeta_4(s,-\epsilon_i;\epsvec{\hat{\imath}})
	    \Bigr) \,.
	\end{equation}
	Expanding at $s=0$, we have $\Gamma(s)=\frac{1}{s}-\gamma+\mathrm{O}(s)$, $\Psi(s)=\Psi(0)+\Psi'(0)s+\mathrm{O}(s^2)$,
	hence
	\begin{equation}
	    \Phi(s)
	    =
	    \frac{\Psi(0)}{s}
	    +
	    \bigl(\Psi'(0)-\gamma\Psi(0)\bigr)
	    +
	    \mathrm{O}(s) \,.
	\end{equation}
	Comparing with $\Phi(s)=\frac{1}{12}\frac{1}{s}+\Phi_0+\mathrm{O}(s)$, we obtain $\Psi(0)=\frac{1}{12}$ and $\Phi_0=\Psi'(0)-\frac{\gamma}{12}$. Therefore
	\begin{equation}
	    \Phi_0+\frac{\gamma}{12}
	    =
        \Psi'(0)
        =
	    \sum_{i=1}^5
	    \Bigl(
	        \zeta_{4,h}'(0;\epsvec{\hat{\imath}})
	        +
	        \zeta_4'(0,-\epsilon_i;\epsvec{\hat{\imath}})
	    \Bigr) \,.
	\end{equation}
	If one further identifies the ordinary and homogeneous Barnes derivatives with the Gamma function and the modular constant respectively, namely,
	\begin{equation}
        \label{eq: Barnes Gamma and modular constant}
		\zeta_4'(0,-\epsilon_i;\epsvec{\hat{\imath}})
		=
		\log \Gamma_4(-\epsilon_i;\epsvec{\hat{\imath}}),
		\qquad
		\zeta_{4,h}'(0;\epsvec{\hat{\imath}})
		=
		-\log \rho_4(\epsvec{\hat{\imath}}) \,,
	\end{equation}
	we obtain the claim. In any case, the right-hand side is meromorphic in the equivariant parameters away from the resonance hyperplanes, so the coefficients of the asymptotic expansion extend meromorphically to generic Calabi--Yau weights.
\end{proof}

\begin{rmk}
    One might be tempted to specialize \cref{eq: sugra asymptotic expansion} to a non-generic locus such as $\epsilon_5=-\epsilon_4=\epsilon$, for which the logarithm of the index reduces to the logarithm of the MacMahon function. Indeed, in this case
    \begin{equation}
        \frac{\sum_{i=1}^5 \cs(2t\epsilon_i)}{\prod_{i=1}^5 \cs(t\epsilon_i)}
        =
        -\frac{1}{\cs(t\epsilon)^2} \,,
    \end{equation}
    so that
    \begin{equation}
        \chi(u)
        =
        \sum_{m\ge 1}\frac{1}{m\,\cs(mu\epsilon)^2}
        =
        \log M(\re^{-u\epsilon}) \,,
    \end{equation}
    where $M(q)=\prod_{n\ge 1}(1-q^n)^{-n}$ is the MacMahon function. However, this specialization lies on a resonance hyperplane, so it is not covered by the generic asymptotic formula above. In particular, the constant term is different: using the classical asymptotic expansion of the MacMahon function, one finds
    \begin{equation}
        \log M(\re^{-\epsilon})
        \sim
        \frac{\zeta(3)}{\epsilon^2}
        +\frac{1}{12}\log \epsilon
        +\zeta'(-1)
        +\sum_{g\ge 2}
        \frac{B_{2g}}{2g} \frac{B_{2g-2}}{2g-2} \,
        \frac{\epsilon^{2g-2}}{(2g-2)!} \,.
    \end{equation}
    Thus, in this resonant specialization, the constant term is $c=-\zeta'(-1)$. The same asymptotic can be derived using the Mellin transform method, see \cite{Pioline}.
\end{rmk}

\subsection{Unpointed Hodge integrals}
In the case of affine five-space we have already reformulated \cref{conj:index} as a concrete conjectural formula for the generating series of constant maps. We will now relate the latter to Hodge integrals.

The moduli space of constant maps to a Calabi--Yau fivefold $\CYf$ is $\Mbar_g(\CYf,0)= \Mbar_g\times\CYf$. The link to Hodge integrals is established by the fact that this moduli space carries the non-trivial obstruction bundle $T_{\CYf}\boxtimes \bE^\vee$. This means its virtual fundamental class is
\begin{equation}
    [\Mbar_g(\CYf,0)]_{\mgp{T}}^{\vir}
    =
    e(T_{\CYf}\boxtimes \bE^\vee )\cap  [\Mbar_g \times \CYf]_{\mgp{T}} \,.
\end{equation}
Indeed, obstructions to deforming a constant map $f\colon C\to \mr{pt}\subset \CYf$ are encoded in
\begin{equation}
    \hhh^1(C,f^*T_\CYf)
    =
    T_\CYf|_{\mr{pt}} \otimes \hhh^1(C,\cO_C)
    =
    T_\CYf|_{\mr{pt}} \otimes \hhh^0(C,\omega_C)^\vee  \,,
\end{equation}
where the last expression is precisely the fiber of $T_\CYf \boxtimes \bE^\vee$ at the constant map $f$.

Concretely, for affine five-space this yields the identification
\begin{equation}
    \GWargs{\Aaff{5}}{\mgp{T}}{0}
    =
    \sum_{g>1} u^{2g-2}
        \int_{\Mbar_g} \prod_{i=1}^5 \HodgeLambda[g]{\epsilon_i}
    =
    \HodgeH(~;\epsilon_1,\ldots,\epsilon_5;u) \,.
\end{equation}
Since in \cref{cor: const maps loc model} we already reformulated \cref{conj:index} as an explicit formula for the left-hand side of the above identity, we deduce the following equivalent characterization of the \namecref{conj:index}.

\begin{prop}
    \label{prop: five hodge no psi local model}
    For $\CYf=\Aaff{5}$, \cref{conj:index} is equivalent to the identity
    \begin{equation}
        \HodgeH(~;\epsilon_1,\ldots, \epsilon_5;u)
        =
        \frac{1}{2}
        \sum_{g>1}
            \frac{B_{2g-2}}{2g-2}
            \left[
                \frac{\sum_{i=1}^5 2\sinh(u\epsilon_i)}
                {\prod_{i=1}^5 2\sinh(\tfrac{u\epsilon_i}{2}) }
            \right]_{2g-2} u^{2g-2}  \,,
    \end{equation}
    where $\sum_{i=1}^5\epsilon_i=0$. \qed
\end{prop}

\subsection{Constant-map contributions for general fivefolds}
We now return to an arbitrary Calabi--Yau fivefold $\CYf$ equipped with a $\mgp{T}$-action. The constant-map contribution to its equivariant Gromov--Witten theory is obtained by integrating the universal degree-zero Hodge expression against the equivariant fundamental class of the target. More precisely, if $\alpha_1,\ldots,\alpha_5$ denote the equivariant Chern roots of $T_{\CYf}$, then
\begin{equation}
    \label{eq: constant maps general fivefold}
    \GWargs{\CYf}{\mgp{T}}{0}
    =
    \int_{[\CYf]_{\mgp{T}}} \left({\textstyle\prod_i} \alpha_i\right) \cdot 
    \HodgeH(~;\alpha_1,\ldots,\alpha_5;u) \,.
\end{equation}
Assuming formula \eqref{eq: 5 lambda 0 psi formula}, this expression can be rewritten in terms of the equivariant characteristic classes of $\CYf$. In this way, the local formula of \cref{prop: five hodge no psi local model} gives rise to a general expression for constant-map contributions.

\begin{prop}
    \label{prop: constant maps general fivefold}
    \Cref{conj:index} is equivalent to the claim that for all Calabi--Yau fivefolds $\CYf$ with the action of a torus $\mgp{T}$ fixing the Calabi--Yau form we have
    \begin{equation}
        \label{eq: constant maps general fivefold formula}
        \GWargs{\CYf}{\mgp{T}}{0}
        =
        \frac{1}{2}
        \sum_{g>1}
        \frac{B_{2g-2}}{2g-2}
        \left[
            \ch[\mgp{T}]\,\chi_{\mgp{T}}\big(\CYf,T_{\CYf}-T_{\CYf}^\vee\big)
        \right]_{2g-2}
        u^{2g-2} \,,
    \end{equation}
    where $\left[\cdot\right]_{2g-2}$ denotes the cohomological degree $(2g-2)$-part.
\end{prop}

\begin{proof}
    First, suppose \Cref{conj:index} is true. By \cref{prop: five hodge no psi local model} and \cref{eq: constant maps general fivefold}, we obtain
    \begin{equation}
        \GWargs{\CYf}{\mgp{T}}{0}
        =
        \frac{1}{2}
        \sum_{g>1}
        \frac{B_{2g-2}}{2g-2}
        \left[
            \int_{[\CYf]_{\mgp{T}}}
            \frac{\sum_{i=1}^5 2\sinh(\alpha_i)}
                 {\prod_{i=1}^5 \cS(\alpha_i)}
        \right]_{2g-2} u^{2g-2}\,.
    \end{equation}
    By the splitting principle, 
    \begin{equation}
        \ch[\mgp{T}]\big(T_{\CYf}-T_{\CYf}^\vee\big)
        =
        \sum_{i=1}^5 2\sinh(\alpha_i) \,,
        \qquad
        \td[\mgp{T}](T_{\CYf})
        =
        \prod_{i=1}^5 \frac{\re^{\alpha_i/2}}{\cS(\alpha_i)} \,.
    \end{equation}
    Using the Calabi--Yau condition \(\sum_i \alpha_i=0\), it follows that
    \begin{equation}
        \frac{\sum_{i=1}^5 2\sinh(\alpha_i)}
             {\prod_{i=1}^5 \cS(\alpha_i)}
        =
        \ch[\mgp{T}]\big(T_{\CYf}-T_{\CYf}^\vee\big)\,\td[\mgp{T}](T_{\CYf}) \,.
    \end{equation}
    Therefore, by equivariant Hirzebruch--Riemann--Roch,
    \begin{equation}
        \left[
            \int_{[\CYf]_{\mgp{T}}}
            \frac{\sum_{i=1}^5 2\sinh(\alpha_i)}
                 {\prod_{i=1}^5 \cS(\alpha_i)}
        \right]_{2g-2}
        =
        \left[
            \ch[\mgp{T}]\,\chi_{\mgp{T}}\big(\CYf,T_{\CYf}-T_{\CYf}^\vee\big)
        \right]_{2g-2} \,.
    \end{equation}
    This proves that \cref{conj:index} implies \cref{eq: constant maps general fivefold formula}. For the converse, one applies the same Mellin-transform argument as in the local case to the plethystic exponential of $\ch[\mgp{T}]\,\chi_{\mgp{T}}(\CYf,T_{\CYf}^\vee-T_{\CYf})$, which recovers the regular tail in \cref{eq: constant maps general fivefold formula} and hence the supergravity index.
\end{proof}

Putting together \cref{prop: sugra index asymptotics}, \cref{prop: five hodge no psi local model}, and \cref{prop: constant maps general fivefold}, we obtain the proof of \cref{thm: five hodge no psi}.

\section{Special regimes for 5-Hodge integrals}
\label{sec:special regimes}

The main result of this section is the proof of \cref{thm: limits}, i.e. a proof of the quintuple Hodge formulas of \cref{thm: 5 Hodge n point formula,thm: five hodge no psi} in two special regimes. In both cases, the argument proceeds by reducing quintuple Hodge integrals to triple Hodge integrals.

\subsection{The anti-diagonal limit}
\label{sec: proof anti-diag limit}
We begin with the anti-diagonal limit:

\begin{customthm}{\ref{thm: limits}.\labelcref{item: antidiagonal}} 
    \label{thm:antidiag}
    Formulas \eqref{eq: 5 Hodge n point formula} and \eqref{eq: 5 lambda 0 psi formula} hold if $\epsilon_i=-\epsilon_j$ for some $i\neq j$.
\end{customthm}

The proof uses Mumford's relation \cite[Eq.~(5.4)]{Mum83:Towards}
\begin{equation}
    \HodgeLambda[g]{\epsilon}\HodgeLambda[g]{-\epsilon}
    =
    (-\epsilon^2)^{g-1} \,.
\end{equation}
By this identity, quintuple Hodge integrals reduce to triple Hodge integrals in the anti-diagonal specialization:
\begin{equation}
    \label{eq: dim reduction 5 to 3}
    \HodgeH(\boldsymbol{z};\epsvec{[5]};u)\big|_{\epsilon_j=-\epsilon_i}
    =
    \HodgeH(\boldsymbol{z};\epsvec{[5]\setminus\{i,j\}};\ri\epsilon_i u) \,.
\end{equation}
Thus, the proof in this regime reduces to explicit formulas for triple Hodge integrals. The unmarked case follows from \cite{FP00:HodgeIntGW}, the one-point case from \cite{FP00:HodgeIntGW} and \cite{OP04:HodgeIntUnknot}, while the case with several markings is treated in the appendix via a Fock space generalization of the latter result.

\subsubsection{No markings}
In \cite[Thm.~4]{FP00:HodgeIntGW} it was shown that, in the unmarked case, the triple Hodge generating series is
\begin{equation}
    \label{eq: 3 lambda 0 psi formula}
    \HodgeH(~;\epsvec{[3]};\ri u)
    =
    -\frac{1}{2} \sum_{g>1}
    \frac{B_{2g-2}}{2g-2}
    \left[\cs(u)^{-2}\right]_{2g-2} u^{2g-2} \,,
\end{equation}
under the Calabi--Yau condition $\sum_{i=1}^3\epsilon_i=0$. This is compatible with our conjectural quintuple formula \labelcref{{eq: 5 lambda 0 psi formula}}, that is
\begin{equation}
    \HodgeH(~;\epsvec{[5]};u)
    =
    \frac{1}{2} \sum_{g>1}
    \frac{B_{2g-2}}{2g-2}
    \left[
        \frac{\sum_{i=1}^5 \cs(2u\epsilon_i)}{\prod_{i=1}^5 \cs(u\epsilon_i)}
    \right]_{2g-2} u^{2g-2} \,.
\end{equation}
Indeed, using
\begin{equation}
    \left.
    \frac{\sum_{i=1}^5 \cs(2u\epsilon_i)}{\prod_{i=1}^5 \cs(u\epsilon_i)}
    \right|_{\epsilon_j=-\epsilon_i}
    =
    -\frac{1}{\cs(u\epsilon_i)^2} \,,
\end{equation}
one sees that the quintuple expression restricts precisely to \eqref{eq: 3 lambda 0 psi formula} with $u$ replaced by $\epsilon_i u$.

\subsubsection{A single marking}
We next prove \cref{thm:antidiag} for the one-point series $\HodgeH(\epsilon_1^{-1};\epsvec{[5]};u)$. Here one must distinguish the cases $1\notin\{i,j\}$ and $1\in\{i,j\}$.

We begin with the case $1\notin\{i,j\}$. Without loss of generality, consider the specialization $\epsilon_4=-\epsilon_5$. In this limit using Mumford's relation, formula \eqref{eq: 5 Hodge n point formula} is equivalent to
\begin{equation}
\begin{split}
    \HodgeH(\epsilon_1^{-1};\epsvec{[3]};\ri\epsilon_4u)
    &=
    \frac{1}{u^2}
    \frac{\gamfct{-\epsilon_1^{-1}\epsilon_2}{u\epsilon_1}}{\epsilon_2}
    \frac{\gamfct{-\epsilon_1^{-1}\epsilon_3}{u\epsilon_1}}{\epsilon_3}
    \frac{\gamfct{-\epsilon_1^{-1}\epsilon_4}{u\epsilon_1}}{\epsilon_4}
    \frac{\gamfct{\epsilon_1^{-1}\epsilon_4}{u\epsilon_1}}{-\epsilon_4} \\
    &=
    -\frac{1}{\epsilon_2 \epsilon_3} \frac{1}{( \epsilon_4 u)^2 \cS(\epsilon_4 u)} \,,
\end{split}
\end{equation}
where the last equality follows from \cref{lem: gamma fct reflect}. Therefore, the specialization $\epsilon_4=-\epsilon_5$ of our one-point quintuple formula follows from the triple Hodge identity proved in \cite[Prop.~5.1]{Nes22:GWHurwitz} (see also \cref{prop: triple hodge formulas 1}):
\begin{equation}
    \label{eq: 3H n1 nonCY}
    \HodgeH(\epsilon_1^{-1};\epsvec{[3]};\ri u)
    =
    - \frac{1}{\epsilon_2 \epsilon_3} \frac{1}{u^2 \cS(u)} \,.
\end{equation}

It remains to consider the case $1 \in \{i,j\}$. Without loss of generality, consider the specialization $\epsilon_1=-\epsilon_5$. In this regime, $\HodgeH(\epsilon_1^{-1};\epsvec{[5]};u) = \HodgeH(\epsilon_1^{-1};\epsvec{\{2,3,4\}};\ri\epsilon_1u)$. Thus, unlike in the previous case, the cotangent-class variable is independent of the Hodge variables, so \eqref{eq: 3H n1 nonCY} no longer applies directly. Fortunately, the required formula is available in \cite[Eq.~(2.15)]{OP04:HodgeIntUnknot}:
\begin{equation}
    \HodgeH(z;\epsvec{\{2,3,4\}};\ri u)
    =
    - \frac{1}{z u^2}
    \prod_{i=2}^4 \frac{\gamfct{-z\epsilon_i}{u}}{\epsilon_i} \,.
\end{equation}
This is again compatible with \eqref{eq: 5 Hodge n point formula}.

\subsubsection{Multiple markings}
The proof of \cref{thm: limits}.\labelcref{item: antidiagonal} with $n>1$ markings proceeds in the same spirit: one shows that the conjectural formulas \eqref{eq: 5 Hodge n point formula} are compatible with the corresponding formulas for triple Hodge integrals. The main point is that the required triple Hodge formulas do not seem to be available in the literature. We therefore derive them in \cref{sec: triple Hodge formulas 1,sec: triple Hodge formulas 2} by two different methods. The first method extracts the relevant identities from localization relations together with known Gromov--Witten invariants. The second uses a result of Okounkov and the second author \cite{OP04:HodgeIntUnknot}, which expresses triple Hodge integrals with integer $\psi$-class variables $\zvec{[n]}\in \bZ_{>0}^n$ as vacuum expectation values of operators on Fock space. The main difficulty in this approach is to analytically continue the resulting expressions to non-integer arguments.

The triple Hodge formulas needed for the proof of \cref{thm:antidiag} are collected in \cref{prop: triple hodge formulas 1,prop: triple hodge formulas 2}. We omit the straightforward verification that these formulas imply the required anti-diagonal specialization of \eqref{eq: 5 Hodge n point formula}.

\subsection{The \texorpdfstring{$\lambda_g\lambda_{g-1}$-}{top Hodge classes }limit}
\label{sec: proof lamg lamg-1 limit}

In this section, we establish our conjectural formulas in the following specialization.

\begin{customthm}{\ref{thm: limits}.\labelcref{item: lamg lamg-1}} 
    \label{thm: lamg lamg-1 limit}
    After multiplication of both sides of the equation by $\epsilon_1\cdots\epsilon_5$, formulas \eqref{eq: 5 Hodge n point formula}, for $1\le n\le 3$, and \eqref{eq: 5 lambda 0 psi formula}
    hold modulo $\left(\sum_{i=1}^5 \epsilon_i,\epsilon_4^2,\epsilon_4\epsilon_5,\epsilon_5^2\right)$.
\end{customthm}

The content of \cref{thm: lamg lamg-1 limit} is that the formulas of \cref{thm: 5 Hodge n point formula,thm: five hodge no psi} hold to next-to-leading order in the Laurent expansion around $\epsilon_4 = \epsilon_5=0$. The leading coefficient $\epsilon_4^{-1}\epsilon_5^{-1}$ follows from \cref{thm: limits}.\labelcref{item: antidiagonal}, by setting $\epsilon_4 = -\epsilon_5$ and then taking $\epsilon_5 \to 0$. It therefore remains to compare the coefficient of $\epsilon_4^0 \epsilon_5^{-1}$.

Throughout this section, we impose
\begin{equation}
    \epsilon_1+\epsilon_2+\epsilon_3=0
\end{equation}
and, for any expression $F$ in the five weights, use the notation
\begin{equation}
    \big[\epsilon_4^0\epsilon_5^{-1}\big]_{\mathrm{CY}} \, F
    \coloneqq
    \big[\epsilon_4^0\epsilon_5^{-1}\big]\,
    F\big(
        \epsilon_1,\epsilon_2,
        \epsilon_3-\epsilon_4-\epsilon_5,
        \epsilon_4,\epsilon_5
    \big) \,.
\end{equation}
Thus, the coefficient is extracted after imposing the fivefold Calabi--Yau condition, while the variables $\epsilon_1,\epsilon_2,\epsilon_3$ in the resulting expression satisfy the threefold Calabi--Yau condition. In particular, the substitution applies to every occurrence of the weights, including those in the cotangent-line variables.

To formulate the required identities, it is convenient to introduce the linear operator $\Psi \colon \bQ\llbracket u^2 \rrbracket \to \bQ\llbracket u^2 \rrbracket$ defined on monomials by
\begin{equation}
    \Psi(u^{2g-2}) = \frac{B_{2g}}{2g} \, u^{2g-2} \,,
    \qquad
    g \geq 1 \,,
\end{equation}
and extended linearly and coefficient-wise after extension of scalars.

With this notation, the coefficient identities required for the proof of
\cref{thm: lamg lamg-1 limit} take the following form.

\begin{prop}
    \label{prop: first derivative formula}
    For the unpointed series, we have
    \begin{equation}
        \label{eq: first derivative formula n0}
        \big[\epsilon_4^0\epsilon_5^{-1}\big]_{\mathrm{CY}} \,
        \HodgeHst\big(~;\epsvec{[5]};u\big)
        =
        u^2 \,
        \Psi\left(
            \frac{1}{2u^2}
            \sum_{i=1}^3
            \left(
                \frac{\kappa(u\epsilon_i)}{u}
                -
                \frac{1}{u^2\epsilon_i}
            \right)
        \right) \,.
    \end{equation}
    For $n\in\{1,2,3\}$, we have
    \begin{equation}
        \big[\epsilon_4^0\epsilon_5^{-1}\big]_{\mathrm{CY}} \,
        \HodgeHst\big(\epsvec[-1]{[n]};\epsvec{[5]};u\big)
        =
        \Psi \big(\vertexlimit\big(\epsvec{[n]};u\big)\big)
        \frac{(-1)^n}{2}
        \prod_{j=1}^n \prod_{i\neq j}
            \frac{1}{\epsilon_i}\,
        +
        \delta_{n,3}\,
        \frac{2}{u^2\epsilon_1^2\epsilon_2^2\epsilon_3^3} \,,
    \end{equation}
    where
    \begin{equation}
        \label{eq: U vertex}
        \vertexlimit(\epsvec{[n]};u)
        \coloneqq
        \begin{cases}
            \dfrac{1}{u}
            \dfrac{
                \cs(u\epsilon_2)\cs(u\epsilon_3)
            }{
                \cs(u\epsilon_1)
            }
            & n=1 \,,\\[1em]
            \dfrac{1}{u}
            \dfrac{
                \cs(u\epsilon_3)^3
            }{
                \cs(u\epsilon_1)\cs(u\epsilon_2)
            }
            & n=2 \,,\\[1em]
            \dfrac{1}{u}
            \left(\prod_{i=1}^3\cs(u\epsilon_i)\right)
            \left(
                \left(\sum_{i=1}^3\kappa(u\epsilon_i)\right)^2
                -\frac14
            \right)
            & n=3 \,.
        \end{cases}
    \end{equation}
    In particular, \cref{thm: lamg lamg-1 limit} follows.
\end{prop}

A direct expansion shows that the right-hand sides of \cref{prop: first derivative formula} are precisely the coefficients of $\epsilon_4^0 \epsilon_5^{-1}$ in the conjectural formulas \eqref{eq: 5 Hodge n point formula} and \eqref{eq: 5 lambda 0 psi formula}. 
The following identities are useful for this comparison:
\begin{equation}
    \sum_{k\geq0}
        \frac{B_{2k}(x)}{(2k)!}\,t^{2k}
    =
    \frac{t}{2}\,
    \frac{\cosh\frac{t-2tx}{2}}{\sinh\frac{t}{2}} \,,
    \quad
    \gamfct{x}{u}
    =
    \exp\left(
        \frac{u^2}{2}\,
        \Psi\left(
            \frac{1}{u^2}
            \left(
                \frac{\cs\big(u(2x-1)\big)}{\cs(u)}
                -2x+1
            \right)
        \right)
    \right) \,.
\end{equation}
The remainder of this section is devoted to the proof of \cref{prop: first derivative formula}. The strategy is to reduce the fivefold coefficient to a combination of triple Hodge integrals.

We first separate the genus-zero contribution. This occurs only for $n=3$, and a direct expansion gives
\begin{equation}
    \big[\epsilon_4^0\epsilon_5^{-1}\big]_{\mathrm{CY}}
    \HodgeHst\big(\epsvec[-1]{[3]};\epsvec{[5]};u\big)
    =
    \frac{2}{u^2\epsilon_1^2\epsilon_2^2\epsilon_3^3} + O(u^0) \,.
\end{equation}
For positive genus, the vanishing $\lambda_g^2=0$ gives
\begin{equation}
    \label{eq: H circ lam g lam g1 limit}
    \big[\epsilon_4^0\epsilon_5^{-1}\big]_{\mathrm{CY}}\,
    \HodgeHst\big(\epsvec[-1]{[n]};\epsvec{[5]};u\big)
    =
    \delta_{n,3}
    \frac{2}{u^2\epsilon_1^2\epsilon_2^2\epsilon_3^3}
    -
    \sum_{g>0}
    u^{2g-2}
    \int_{\Mbar_{g,n}}
    \frac{
        \HodgeLambda[g]{\epsilon_1}
        \HodgeLambda[g]{\epsilon_2}
        \HodgeLambda[g]{\epsilon_3}
    }{
        \prod_{j=1}^n(\epsilon_j-\psi_j)
    }
    \lambda_g\lambda_{g-1} \,.
\end{equation}
The following \namecref{lem: lam g lam g minus one formula} reduces the
positive-genus contribution to triple Hodge integrals.

\begin{lem}
\label{lem: lam g lam g minus one formula}
    Let $\iota \colon \Mbar_{g-1,n+2} \to \Mbar_{g,n}$ be the normalization map of the irreducible boundary divisor. For
    $g\geq1$ and $2g-2+n>0$, we have
    \begin{equation}
        \label{eq: lam g lam g minus one formula}
        \lambda_g\lambda_{g-1}
        =
        \frac{B_{2g}}{4g}\,
        \iota_*
        \sum_{k=0}^{2g-2}
            (-1)^{g-1-k} \,
            \psi_{n+1}^k \psi_{n+2}^{2g-2-k}
    \end{equation}
    in $\Chow^{2g-1}(\Mbar_{g,n})$.
\end{lem}

\begin{proof}
    For $g=1$, the statement is the standard relation
    $\lambda_1=\frac{1}{24}\iota_*1$, so we assume $g\geq2$.
    By \cite[\S4.3]{FP00:HodgeIntGW},
    \begin{equation}
        \lambda_g\lambda_{g-1}
        =
        (-1)^{g-1} (2g-1)! \,
        \ch[2g-1](\bE) \,.
    \end{equation}
    Mumford's formula for the Chern character
    \cite[Eq.~(5.2)]{Mum83:Towards} therefore gives
    \begin{multline}
        \lambda_g\lambda_{g-1}
        =
        \frac{B_{2g}}{2g}
        \Bigg(
            (-1)^{g-1} \kappa_{2g-1}
            -
            (-1)^{g-1} \sum_{i=1}^n\psi_i^{2g-1}
            \\
            +
            \frac{1}{2}
            \sum_{k=0}^{2g-2}
            (-1)^{g-1-k}
            \Bigl(
                \iota'_* \psi_{n+1}^k\psi_{n+2}^{2g-2-k}
                +
                \iota_* \psi_{n+1}^k\psi_{n+2}^{2g-2-k}
            \Bigr)
        \Bigg) \,,
    \end{multline}
    where $\iota'$ denotes the normalization map of the separating boundary. On the other hand, pulling back the relation
    \cite[Eq.~(7)]{LP11:NewTRR} from $\Mbar_g$ to $\Mbar_{g,n}$ gives
    \begin{equation}
        \kappa_{2g-1}
        -
        \sum_{i=1}^n \psi_i^{2g-1}
        +
        \frac{1}{2}
        \iota'_*
        \sum_{k=0}^{2g-2}
            (-1)^k
            \psi_{n+1}^k\psi_{n+2}^{2g-2-k}
        =
        0 \,.
    \end{equation}
    Substituting this identity into Mumford's formula leaves only the
    irreducible-boundary contribution, as claimed.
\end{proof}

\begin{cor}
    \label{cor: first derivative to 3H}
    For $n \in \{0,1,2,3\}$, we have
    \begin{multline}
    \label{eq: first derivative to 3H}
        \big[\epsilon_4^0\epsilon_5^{-1}\big]_{\mathrm{CY}} \,
        \HodgeHst\big(\epsvec[-1]{[n]};\epsvec{[5]};u\big) \\
        =
        \delta_{n,3}\,
        \frac{2}{u^2\epsilon_1^2\epsilon_2^2\epsilon_3^3}
        +
        [z^2]\,
        \frac{\epsilon_1 \epsilon_2 \epsilon_3}{2}\,
        \Psi\left(
            (\ri z^{-1}u)^2 \,
            \HodgeHst\big(
                z,-z,\epsvec[-1]{[n]};
                \epsvec{[3]};
                \ri z^{-1}u
            \big)
        \right) .
    \end{multline}
\end{cor}

\begin{proof}
    By \cref{lem: lam g lam g minus one formula},
    \begin{equation}
        \lambda_g\lambda_{g-1}
        =
        [z^{2g}]\,
        (-1)^g\frac{B_{2g}}{4g}\,
        \iota_*
        \frac{1}{
            (z^{-1}-\psi_{n+1})
            (-z^{-1}-\psi_{n+2})
        }.
    \end{equation}
    Moreover, $\iota^*\HodgeLambda[g]{\epsilon_i} = \epsilon_i \HodgeLambda[g-1]{\epsilon_i}$. Plugging this into \eqref{eq: H circ lam g lam g1 limit} and recalling that $\Psi(u^{2g-2}) = \frac{B_{2g}}{2g} u^{2g-2}$ immediately yields \eqref{eq: first derivative to 3H}.
\end{proof}

By the last \namecref{cor: first derivative to 3H}, the evaluation of the quintuple Hodge integrals in \cref{prop: first derivative formula} is reduced to the extraction of certain Laurent coefficients of triple Hodge integrals, for which explicit formulas are given in \cref{prop: triple hodge formulas 2}.

\begin{proof}[Proof of \cref{prop: first derivative formula}]
    We only prove the statement for the unpointed series, leaving the cases $n>0$ to the reader. For $n=0$, \cref{cor: first derivative to 3H} gives
    \begin{equation}
        \big[\epsilon_4^0\epsilon_5^{-1}\big]_{\mathrm{CY}} \,
        \HodgeHst(~;\epsvec{[5]};u)
        =
        \frac{\prod_{i=1}^3\epsilon_i}{2} \,
        [z^2] \,
        \Psi\Bigl(
            (\ri z^{-1}u)^2 \,
            \HodgeHst\big(z,-z;\epsvec{[3]};\ri z^{-1}u\big)
        \Bigr) \,.
    \end{equation}
    Substituting \eqref{eq: 3H n2 zz} and extracting the coefficient of $z^2$ yields
    \begin{equation}
        \big[\epsilon_4^0\epsilon_5^{-1}\big]_{\mathrm{CY}} \,
        \HodgeHst(~;\epsvec{[5]};u)
        =
        \frac{u^2}{2} \,
        \Psi\left(
            \frac{1}{u^2}
            \Psi\left(
                \sum_{i=1}^3
                \epsilon_i\cS(2u\epsilon_i)
            \right)
        \right) \,.
    \end{equation}
    The claimed formula \eqref{eq: first derivative formula n0} follows from
    \begin{equation}
        \label{eq: Psi sinh to coth}
        \Psi\bigl(\cS(2u\epsilon)\bigr)
        =
        \frac{\kappa(u\epsilon)}{u\epsilon}
        -
        \frac{1}{u^2\epsilon^2} \,.
        \qedhere
    \end{equation}
\end{proof}

\appendix

\section{Triple Hodge integrals from the vertex}
\label{sec: triple Hodge formulas 1}

In this appendix we prove formulas for triple Hodge $\psi$-class integrals in the setting where the formal variable recording the cotangent line coincides with one of the variables recording the insertion of a $\lambda$-class. We stress that we do not impose any Calabi--Yau condition on the variables of the generating series.

\begin{prop}
    \label{prop: triple hodge formulas 1}
    For $n\in\{1,2,3\}$, we have
    \begin{equation}
        \label{eq: 3 lambda 1 3 psi}
        \HodgeH(\epsvec[-1]{[n]} ; \epsvec{[3]} ; u)
        =
        \frac{
            \bigl( \sum_{i=1}^n \epsilon_i^{-1} \bigr)^{n-3}
        }{
            u^2 \bigl( \prod_{j=1}^3\epsilon_j \bigr) \bigl( \prod_{i=1}^n\epsilon_i \bigr)
        }
        \left( \frac{\sin(u/2)}{u/2} \right)^{p_n(\epsvec{[3]})} \,,
    \end{equation}
    where
    \begin{equation}
        p_n(\epsvec{[3]})
        =
        n - 2 +
        \left(\sum_{j=1}^3\epsilon_j\right) \left(\sum_{i=1}^n\epsilon_i^{-1}\right) \,.
    \end{equation}
\end{prop}

\begin{proof}
    The formula for $n=1$ was proved by Schimpf in \cite[Prop.~5.1]{Nes22:GWHurwitz} using the methods of Faber and the second author \cite{FP00:HodgeIntGW}. We therefore turn to the remaining cases, beginning with the three-point formula, which follows directly from the calculation of the closed vertex partition function in \cite{BK05:closedVertex}. The two-point case is slightly more subtle: the analogous curve class in the closed vertex does not give rise to a proper moduli space of stable maps and must therefore be treated using a different local Calabi--Yau geometry.
    
    The closed vertex is a local Calabi--Yau threefold $X$ whose compact toric locus consists of three rational curves $C_1,C_2,C_3$, each with normal bundle $\cO_{\bP^1}(-1)\oplus \cO_{\bP^1}(-1)$, meeting at a point. We display the configuration of these torus orbit closures together with the tangent $\Gm[3]$-weights at each torus fixed point in \cref{fig: closed top vertex}. On the one hand, we compute the all-genus Gromov--Witten series of $X$ in class $\beta=[C_1]+[C_2]+[C_3]$ by localization:
    \begin{equation}
    \label{eq: closed vertex loc}
    \begin{split}
        \sum_{g \geq 0} u^{2g-2}
            \int_{[\Mbar_g(X,\beta)]^{\vir}} 1
        &=
        u^6
        (\epsilon_1\epsilon_2\epsilon_3)^2
        (\epsilon_1+\epsilon_2)^2
        (\epsilon_1+\epsilon_3)^2
        (\epsilon_2+\epsilon_3)^2 \\
        &\quad\times
        \HodgeH(\epsvec[-1]{[3]} ; \epsvec{[3]} ; u) \,
        \HodgeH(-\epsilon_1^{-1} ;
            -\epsilon_1,\epsilon_1+\epsilon_2,\epsilon_1+\epsilon_3 ; u)
        \\
        &\quad\times
        \HodgeH(-\epsilon_2^{-1} ;
            -\epsilon_2,\epsilon_1+\epsilon_2,\epsilon_2+\epsilon_3 ; u) \,
        \HodgeH(-\epsilon_3^{-1} ;
            -\epsilon_3,\epsilon_1+\epsilon_3,\epsilon_2+\epsilon_3 ; u)
        \\
        &=
        (\epsilon_1\epsilon_2\epsilon_3)^2 \,
        \HodgeH(\epsvec[-1]{[3]} ; \epsvec{[3]} ; u)
        \left( \frac{\sin(u/2)}{u/2} \right)^{
            -\frac{2\epsilon_1+\epsilon_2+\epsilon_3}{\epsilon_1}
            -\frac{\epsilon_1+2\epsilon_2+\epsilon_3}{\epsilon_2}
            -\frac{\epsilon_1+\epsilon_2+2\epsilon_3}{\epsilon_3}} \,,
    \end{split}
    \end{equation}
    where, in the second equality, we used the $n=1$ case of \cref{prop: triple hodge formulas 1}. On the other hand, since $\Mbar_g(X,\beta)$ is proper, \cite{BK05:closedVertex} computes the series as
    \begin{equation}
        \sum_{g \geq 0} u^{2g-2}
            \int_{[\Mbar_g(X,\beta)]^{\vir}} 1
        =
        \frac{1}{u^2}
        \left( \frac{\sin (u/2)}{u/2} \right)^{-2} \,.
    \end{equation}
    Combining the last two equations and solving for the three-point series yields \eqref{eq: 3 lambda 1 3 psi} for $n=3$.

    \begin{figure}
    \centering
        \begin{subfigure}[t]{0.47\textwidth}
            \centering
            \begin{tikzpicture}
                \draw (-1.06,-1.06) -- (0,0) -- (0,1.5);
                \draw (0,0) -- (1.5,0);
                \draw (-2.56,-1.06) -- (-1.06,-1.06) -- (-1.06,-2.56);
                \draw (-1.5,1.5) -- (0,1.5) -- (1.06,2.56);
                \draw (1.5,-1.5) -- (1.5,0) -- (2.56,1.06);
                \node[left] at (-0.7,-0.6) {$C_1$};
                \node[right] at (0,1) {$C_2$};
                \node[below] at (1,0) {$C_3$};
                \node at (0,-0.4) {$\epsilon_1$};
                \node (B) at (-0.25,0.3) {$\epsilon_2$};
                \node (C) at (0.3,0.2) {$\epsilon_3$};
                \node[left] at (-0.1,1.7) {$\epsilon_1+\epsilon_2$};
                \node[above=0.9 of B.east, anchor=east] {$-\epsilon_2$};
                \node[right] at (0.15,1.6) {$\epsilon_2+\epsilon_3$};
                \node[right] at (1.5,-0.4) {$\epsilon_1+\epsilon_3$};
                \node[right] at (1.9,0.3) {$\epsilon_2+\epsilon_3$};
                \node at (1.1,0.2) {$-\epsilon_3$};
                \node at (-0.5,-0.95) {$-\epsilon_1$};
                \node[left] at (-1.2,-1.3) {$\epsilon_1+\epsilon_2$};
                \node[right] at (-1.06,-1.5) {$\epsilon_1+\epsilon_3$};
            \end{tikzpicture}
            \caption{The closed topological vertex.}
            \label{fig: closed top vertex}
        \end{subfigure}
        \hfill
        \begin{subfigure}[t]{0.47\textwidth}
            \centering
            \begin{tikzpicture}
                \draw (-2.56,-1.06) -- (-1.06,-1.06) -- (0,0) -- (0,1.5) -- (-2.56,1.5) -- (-2.56,-1.06);
                \draw (0,0) -- (1.5,0);
                \draw (0,1.5) -- (1.06,2.56);
                \draw (-1.06,-1.06) -- (-1.06,-2.56);
                \draw (-2.56,-1.06) -- (-3.62,-2.12);
                \draw (-2.56,1.5) -- (-3.62,2.56);
                \node at (0,-0.4) {$\epsilon_1$};
                \node (B) at (-0.25,0.3) {$\epsilon_2$};
                \node (C) at (0.3,0.2) {$\epsilon_3$};
                \node[left] at (-0.1,1.7) {$\epsilon_1+\epsilon_2$};
                \node[above=0.9 of B.east, anchor=east] {$-\epsilon_2$};
                \node[right] at (0.15,1.6) {$\epsilon_2+\epsilon_3$};
                \node at (-0.5,-0.95) {$-\epsilon_1$};
                \node[right] at (-1.06,-1.5) {$\epsilon_1+\epsilon_3$};
                \node[left] at (-1.06,-1.5) {$\epsilon_1+\epsilon_2$};
                \node[left] at (-2.56,0.22) {$\vdots$};
                \node[right] at (-2.56,1.7) {$\dots$};
                \node[right] at (-2.56,-0.85) {$\dots$};
            \end{tikzpicture}
            \caption{The local surface $X' = \Tot \, K_{\operatorname{Bl}_{(0,0)}(\bP^1\times\bP^1)}$.}
            \label{fig: local Bl P1xP1 weights}
        \end{subfigure}
        \caption{Toric skeleta used in the proof, together with the tangent weights at the indicated torus fixed points.}
        \label{fig: toric skeleta}
    \end{figure}
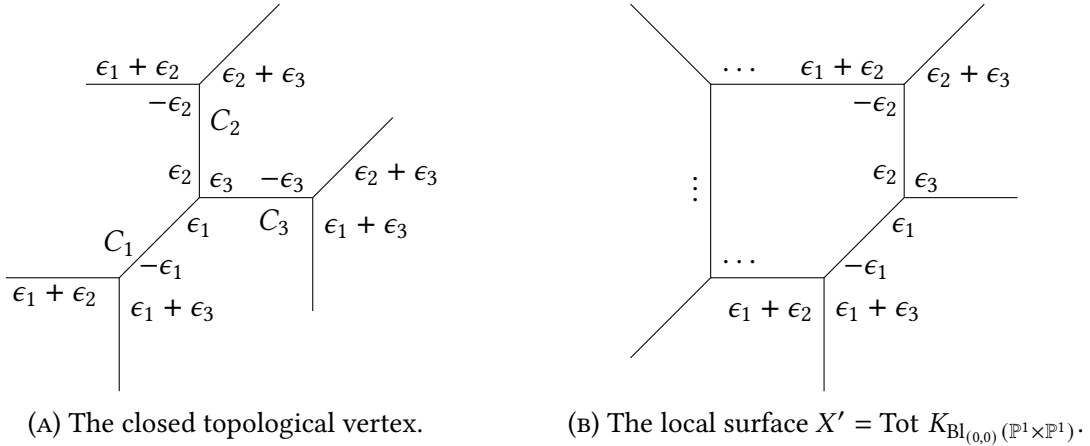
    
    A similar argument applies to the case $n=2$. In this case, however, we cannot simply consider the closed vertex $X$ in curve class $[C_1]+[C_2]$, since the corresponding moduli space of stable maps is not proper. Instead, we consider stable maps to $X' = \Tot \, K_{\operatorname{Bl}_{(0,0)}(\bP^1\times\bP^1)}$ in the curve class obtained by pulling back a fiber class under the blow-up, which we denote by $\beta$. The resulting moduli space is proper, and its Gromov--Witten series is
    \begin{equation}
        \label{eq: local Bl P1xP1 GW rhs}
        \sum_{g \geq 0} u^{2g-2}
            \int_{[\Mbar_{g}(X',\beta)]^{\vir}} 1
        =
        -\frac{2}{u^2} \left( \frac{\sin (u/2)}{u/2} \right)^{-2} \,.
    \end{equation}
    The same series can also be computed by localization with respect to the dense torus $\Gm[3]$. Choosing the tangent weights as shown in \cref{fig: local Bl P1xP1 weights}, we obtain
    \begin{multline}
    \label{eq: local Bl P1xP1 GW lhs}
        \sum_{g \geq 0} u^{2g-2}
            \int_{[\Mbar_g(X',\beta)]^{\vir}} 1
        =
        -\frac{2\epsilon_1+2\epsilon_2+\epsilon_3}{u^2(\epsilon_1+\epsilon_2)}
        \left( \frac{\sin(u/2)}{u/2} \right)^{
            \frac{\epsilon_1+\epsilon_3}{\epsilon_2}
            -
            \frac{\epsilon_1+2\epsilon_2+\epsilon_3}{\epsilon_2}
        }
        \\
        +
        \epsilon_1\epsilon_2\epsilon_3^2 \,
        \HodgeH(\epsvec[-1]{[2]} ; \epsvec{[3]} ; u)
        \left(\frac{\sin(u/2)}{u/2}\right)^{
            -\frac{2\epsilon_1+\epsilon_2+\epsilon_3}{\epsilon_1}
            -\frac{\epsilon_1+2\epsilon_2+\epsilon_3}{\epsilon_2}
        } \,.
    \end{multline}
    Here we have again used the $n=1$ case of \cref{prop: triple hodge formulas 1} to evaluate the triple Hodge integrals with a single $\psi$-class. The first term on the right-hand side arises from maps to the fiber through $(\infty,\infty)$, while the second arises from maps whose composition with the blow-up morphism is the fiber through $(0,0)$. Combining the last two equations and solving for the two-point series yields \eqref{eq: 3 lambda 1 3 psi} for $n=2$.
\end{proof}

\section{Triple Hodge integrals from Fock space}
\label{sec: triple Hodge formulas 2}

In this appendix, we prove the remaining formulas for Calabi--Yau triple Hodge integrals required in the main text. Our starting point is the description of triple Hodge integrals as vacuum expectation values of the Fock-space operators introduced by Okounkov and the second author in \cite{OP04:HodgeIntUnknot}. These formulas are initially established for positive integral values of the cotangent-line variables. We obtain the identities below by analytically continuing the corresponding vacuum expectation values and then specializing the variables.

\begin{prop}
    \label{prop: triple hodge formulas 2}
    Suppose that $\epsilon_1+\epsilon_2+\epsilon_3=0$. Then the following identities hold.
    \begin{enumerate}[label=\textup{(\roman*)}]
        \item \emph{One marking:}
        \begin{equation}
            \label{eq: 3H n1 z}
            \HodgeH\bigl(z;\epsvec{[3]};\ri u\bigr)
            =
            -
            \frac{1}{u^2 z}
            \prod_{i=1}^3 \frac{\gamfct{-z\epsilon_i}{u}}{\epsilon_i} \,.
        \end{equation}

        \item \emph{Two markings:}
        \begin{align}
            \HodgeH\bigl(z,\epsilon_1^{-1};\epsvec{[3]};\ri u\bigr)
            &=
            -\frac{z}{u^2(\epsilon_1 z + 1)}
            \frac{
                \cS(u\epsilon_2z)\cS(u\epsilon_3z)
            }{
                \cS\bigl(u(\epsilon_1z+1)\bigr)
            }
            \prod_{i=1}^3 \frac{\gamfct{-z\epsilon_i}{u}}{\epsilon_i} \,,
            \label{eq: 3H n2 z}
            \\
            \HodgeHst\bigl(z,-z;\epsvec{[3]};\ri u\bigr)
            &=
            -\frac{1}{\epsilon_1\epsilon_2\epsilon_3}
            \Psi\left(
                \frac{z^2}{u^2\cS(u)^2}
                \sum_{i=1}^3
                \epsilon_i\cS(2u\epsilon_i z)
            \right) \,.
            \label{eq: 3H n2 zz}
        \end{align}

        \item \emph{Three markings:}
        \begin{align}
            \HodgeH\bigl(z,\epsvec[-1]{[2]};\epsvec{[3]};\ri u\bigr)
            &=
            -\frac{z}{u^2\epsilon_1\epsilon_2}
            \cS(u\epsilon_3z)^2
            \prod_{i=1}^3 \frac{\gamfct{-z\epsilon_i}{u}}{\epsilon_i} \,,
            \label{eq: 3H n3 z}
            \\
            \HodgeH\bigl(z,-z,\epsilon_1^{-1};\epsvec{[3]};\ri u\bigr)
            &=
            \frac{z^2}{u^2\epsilon_1^2\epsilon_2\epsilon_3}
            \frac{\cS(u\epsilon_2z)\cS(u\epsilon_3z)}
            {\cS(u\epsilon_1z)\cS(u)} \,.
            \label{eq: 3H n3 zz}
        \end{align}

        \item \emph{Four markings:}
        \begin{align}
            \HodgeH\bigl(z,\epsvec[-1]{[3]};\epsvec{[3]};\ri u\bigr)
            &=
            -\frac{z^2}{u\epsilon_1 \epsilon_2 \epsilon_3}
            \cS(u)
            \left(
                \kappa(u)
                +
                \sum_{i=1}^3 \kappa(u\epsilon_i z)
            \right)
            \prod_{i=1}^3
                \cS(u\epsilon_i z)
                \frac{\gamfct{-z\epsilon_i}{u}}{\epsilon_i} \,,
            \label{eq: 3H n4 z}
            \\
            \HodgeH\bigl(z,-z,\epsvec[-1]{[2]};\epsvec{[3]};\ri u\bigr)
            &=
            -\frac{z^2}{u^2\epsilon_1^3\epsilon_2^3}
            \frac{\cS(u\epsilon_3z)^3}{\cS(u\epsilon_1z)\cS(u\epsilon_2z)} \,.
            \label{eq: 3H n4 zz}
        \end{align}

        \item \emph{Five markings:}
        \begin{equation}
            \label{eq: 3H n5 zz}
                \HodgeH\bigl(z,-z,\epsvec[-1]{[3]};\epsvec{[3]};\ri u\bigr)
                =
                \frac{z^4}{\epsilon_1^2\epsilon_2^2\epsilon_3^2}
                \cS(u)
                \prod_{i=1}^3\cS(u\epsilon_i z)
                \left[
                    \left(
                        \sum_{i=1}^3
                        \kappa(u\epsilon_i z)
                    \right)^2
                    - \frac{1}{4}
                \right] \,.
        \end{equation}
    \end{enumerate}
\end{prop}

\Cref{eq: 3H n1 z} is equivalent, after translating conventions, to the formula proved by Okounkov and the second author in \cite[Eq.~(2.15)]{OP04:HodgeIntUnknot}.

\subsection{Preliminaries}
To prove the formulas, we use an identification of triple Hodge integrals with vacuum expectation values of the Fock-space operators
\begin{equation}
    \label{eq: defn A op}
    \mathsf{A}(z,a,u)
    =
    \frac{\gamfct{-z,-az,(a+1)z}{u}}{u^2 a(a+1) z} \,
    \cs(uaz)
    \sum_{k\in \bZ}
        \frac{(q^{(a+1)z};q)_{k}}{(q^{z+1};q)_{k}} \,
        \mathcal{E}_k(uaz)
\end{equation}
acting on Fock space. Here we introduce the shorthand
\begin{equation}
    \gamfct{x_1,\ldots,x_n}{u}
    \coloneqq
    \prod_{i=1}^n \gamfct{x_i}{u}
\end{equation}
and let $(q^{a};q)_k$ denote the symmetrized $q$-Pochhammer symbol
\begin{equation}
    (q^{a};q)_k
    =
    \begin{cases}
        \cs(ua)\cs(u(a+1)) \cdots \cs(u(a+k-1))
        & k\geq 0 \,, \\[0.3em]
        \dfrac{1}{\cs(u(a-1))\cs(u(a-2)) \cdots \cs(u(a+k))}
        & k < 0 \,,
    \end{cases}
    \qquad
    q \coloneqq \re^u \,.
\end{equation}
Finally, $\mathcal{E}_k(z)$ denotes the Fock-space operator whose vacuum expectation values of products can be computed recursively.

\begin{lem}
    \label{lem:VEV}
    \cite[\S3.3]{OP06:GW-H}
    The connected vacuum expectation value $\langle \mathcal{E}_{k_1}(z_1) \cdots \mathcal{E}_{k_n}(z_n) \rangle^{\circ}$ vanishes for $n>1$ whenever $k_1\leq 0$, $k_n\geq 0$, or $\sum_j k_j\neq 0$. If $k_1 > 0$, it satisfies the recursion
    \begin{equation}
        \left\langle \mathcal{E}_{k_1}(z_1) \cdots \mathcal{E}_{k_n}(z_n) \right\rangle^{\circ}
        =
        \sum_{i=2}^n
        \cs\Bigl(
            \det\begin{psmallmatrix} k_1 & k_i \\ z_1 & z_i \end{psmallmatrix}
        \Bigr)
        \left\langle
            \mathcal{E}_{k_2}(z_2) \cdots
            \mathcal{E}_{k_i+k_1}(z_i+z_1) \cdots
            \mathcal{E}_{k_n}(z_n)
        \right\rangle^{\circ} \,,
    \end{equation}
    with initial condition $\langle \mathcal{E}_{k}(z) \rangle^{\circ} = \frac{\delta_{k,0}}{\cs(z)}$.
\end{lem}

The following result of Okounkov and the second author relates vacuum expectation values of $\msf{A}$-operators to triple Hodge integrals.

\begin{thm}
    \label{thm: OP unknot}
    \cite[Thm.~2]{OP04:HodgeIntUnknot}
    If $z_1,\ldots,z_n$ are positive integers, then
    \begin{equation}
        \HodgeH\left(\zvec{[n]};1,a,-(a+1);\ri u\right)
        =
        \left\langle
            \prod_{j=1}^n \mathsf{A}\bigl(z_j,a,u\bigr)
        \right\rangle^{\circ} \,.
    \end{equation}
\end{thm}

By homogeneity, the Hodge integrals in \cref{prop: triple hodge formulas 2} are related to the above ones by the change of variables
\begin{equation} \label{eq: homog}
    \HodgeH\bigl(
        \zvec{[n]};\epsilon_1,\epsilon_2,\epsilon_3;\ri u
    \bigr)
    =
    \epsilon_1^{-2n}\,
    \HodgeH\bigl(
        \epsilon_1\zvec{[n]};1,a,-(a+1);\ri u
    \bigr) \,,
    \qquad
    a = \frac{\epsilon_2}{\epsilon_1} \,,
\end{equation}
Hence, to deduce \cref{prop: triple hodge formulas 2}, it remains only to treat the cases with at least two markings, since the one-point formula is already contained in \cite{OP04:HodgeIntUnknot}. For the remaining cases, we analytically continue the identity in \cref{thm: OP unknot} to non-integral values of the variables $\boldsymbol{z}$ and evaluate the resulting continuation at the specializations appearing in the \namecref{prop: triple hodge formulas 2}.

\subsection{Two markings}
To prove \eqref{eq: 3H n2 z} and \eqref{eq: 3H n2 zz}, we evaluate the connected vacuum expectation value representing $\HodgeH(z_1,z_2;1,a,-(a+1);\ri u) = \langle \sfA(z_1,a,u)\, \sfA(z_2,a,u) \rangle^{\circ}$, initially for $z_1,z_2\in\bZ_{>0}$. Substituting the definition \eqref{eq: defn A op} of the $\sfA$-operators gives
\begin{multline}
    \HodgeH\bigl(z_1,z_2;1,a,-(a+1);\ri u\bigr)
    =
    \frac{
        \prod_{j=1}^2
        \gamfct{-z_j,-az_j,(a+1)z_j}{u}
    }
    {u^4 a^2(a+1)^2 z_1z_2}
    \,
    \cs(uaz_1)\,
    \cs(uaz_2) \\
    \times
    \sum_{k_1,k_2\in\bZ}
    \frac{
        (q^{(a+1)z_1};q)_{k_1}
        (q^{(a+1)z_2};q)_{k_2}
    }{
        (q^{z_1+1};q)_{k_1}
        (q^{z_2+1};q)_{k_2}
    }
    \left\langle
        \cE_{k_1}(uaz_1)\,
        \cE_{k_2}(uaz_2)
    \right\rangle^{\circ} \,.
\end{multline}
By \cref{lem:VEV},
\begin{equation}
    \left\langle
        \cE_{k_1}(uaz_1)\,
        \cE_{k_2}(uaz_2)
    \right\rangle^{\circ}
    =
    \begin{cases}
        \dfrac{\cs\bigl(ua k_1(z_1+z_2)\bigr)}{\cs\bigl(ua(z_1+z_2)\bigr)}
        &
        k_1+k_2=0,\quad k_1>0,
        \\
        0
        & \text{otherwise}.
    \end{cases}
\end{equation}
Therefore, using $(q^a;q)_{-k} = \frac{(-1)^k}{(q^{1-a};q)_k}$, we obtain
\begin{multline}
\label{eq: Hodge vs q hypergeom n2}
    \HodgeH\bigl(z_1,z_2;1,a,-(a+1);\ri u\bigr)
    =
    \frac{
        \prod_{j=1}^2
        \gamfct{-z_j,-az_j,(a+1)z_j}{u}
    }{u^4 a^2(a+1)^2 z_1z_2}
    \,
    \frac{\cs(uaz_1)\,\cs(uaz_2)}
         {\cs\bigl(ua(z_1+z_2)\bigr)}
    \\
    \times
    \sum_{k>0}
    \frac{
        (q^{(a+1)z_1};q)_k
        (q^{-z_2};q)_k
    }{
        (q^{z_1+1};q)_k
        (q^{1-(a+1)z_2};q)_k
    }
    \cs\bigl(uak(z_1+z_2)\bigr) \,.
\end{multline}

\subsubsection{The specialization \texorpdfstring{$z_2=1$}{z2=1}.}
To prove formula \eqref{eq: 3H n2 z}, we specialize \cref{eq: Hodge vs q hypergeom n2} to $z_2=1$. The factor $(q^{-1};q)_k$ vanishes for $k\geq2$, so the sum reduces to the single term $k=1$, and we obtain
\begin{equation}
    \HodgeH\left(z,1;1,a,-(a+1);\ri u\right)
    =
    \frac{
        \gamfct{-z,-az,(a+1)z}{u}
    }{
        u^3a^2(a+1)^2z
    }
    \frac{
        \cs(uaz)\cs\bigl(u(a+1)z\bigr)
    }{
        \cs\bigl(u(z+1)\bigr)
    } \,,
\end{equation}
where we used $\gamfct{-1,-a,a+1}{u}=u/\cs(u)$ to simplify the second $\gamma$-factor. By the homogeneity relation \eqref{eq: homog}, the resulting identity is precisely \eqref{eq: 3H n2 z}.

\subsubsection{The \texorpdfstring{specialization $z_1+z_2=0$}{anti-diagonal specialization}.}
Unlike the previous specialization, this case requires additional care because the two-point Hodge series contains the unstable genus-zero contribution $1/(z_1+z_2)$. Consequently, the specialization $z_1+z_2=0$ can only be taken after removing this pole.

To explicitly evaluate the two-point Hodge series, fix $z_2 \in \bZ_{>0}$. Since the factor $(q^{-z_2};q)_k$ vanishes for $k > z_2$, the sum on the right-hand side of \cref{eq: Hodge vs q hypergeom n2} is finite. Hence it is a formal Laurent series in $u$ whose coefficients are rational functions of $z_1$. By definition, the same is true for the left-hand side. Since positive integers are Zariski dense, identity \eqref{eq: Hodge vs q hypergeom n2} therefore extends to arbitrary values of $z_1$.

We now specialize to $z_1+z_2=0$. Both sides have a simple pole in this limit. On the left-hand side, however, the pole is entirely due to the unstable genus-zero contribution:
\begin{equation}
    \HodgeH\bigl(z_1,z_2;1,a,-(a+1);\ri u\bigr)
    =
    \frac{1}{u^2 a (a+1)}
    \frac{z_1z_2}{z_1+z_2}
    +
    \HodgeHst\bigl(z_1,z_2;1,a,-(a+1);\ri u\bigr) \,.
\end{equation}
It therefore remains to extract the finite part of the right-hand side of \cref{eq: Hodge vs q hypergeom n2} at $z_1+z_2=0$, which is precisely $\HodgeHst(z,-z;1,a,-(a+1);\ri u)$.

The right-hand side of \cref{eq: Hodge vs q hypergeom n2} terminates at $k=z_2$ and develops a simple pole along $z_1+z_2=0$. This pole arises solely from the summand with $k=z_2$, whose denominator contains the factor
\begin{equation}
    \frac{1}{(q^{z_1+1};q)_{z_2}}
    =
    \frac{1}{\cs(u(z_1+1))\cdots\cs(u(z_1+z_2))} \,.
\end{equation}
All remaining summands are regular along $z_1+z_2=0$. It is therefore convenient to separate the regular part $\mathsf{h}_{\mr{r}}$ and the polar part $\mathsf{h}_{\mr{p}}$ by writing
\begin{equation}
\begin{split}
    \frac{\mathsf{h}_{\mr{r}}(z_1,z_2)}{
        \frac{\prod_{j=1}^2 \gamfct{-z_j,-az_j,(a+1)z_j}{u}}{u^4 a^2(a+1)^2 z_1z_2}
        \frac{\cs(uaz_1)\cs(uaz_2)}{\cs(ua(z_1+z_2))}
    }
    &\coloneqq
    \sum_{k=0}^{z_2-1}
    \frac{
        (q^{(a+1)z_1};q)_k
        (q^{-z_2};q)_k
    }{
        (q^{z_1+1};q)_k
        (q^{1-(a+1)z_2};q)_k
    }
    \cs\bigl(uak(z_1+z_2)\bigr) \,,
    \\
    \frac{\mathsf{h}_{\mr{p}}(z_1,z_2)}{
        \frac{\prod_{j=1}^2 \gamfct{-z_j,-az_j,(a+1)z_j}{u}}{u^4 a^2(a+1)^2 z_1z_2}
        \frac{\cs(uaz_1)\cs(uaz_2)}{\cs(ua(z_1+z_2))}
    }
    &\coloneqq
    \frac{
        (q^{(a+1)z_1};q)_{z_2}
        (q^{-z_2};q)_{z_2}
    }{
        (q^{z_1+1};q)_{z_2}
        (q^{1-(a+1)z_2};q)_{z_2}
    }
    \cs\bigl(uaz_2(z_1+z_2)\bigr) .
\end{split}
\end{equation}
We suppress the dependence on $a$ and $u$ for ease of notation. Thus, it remains to evaluate
\begin{equation}
    \label{eq: 3H n2 zz to reg sing}
    \HodgeHst(z,-z;1,a,-(a+1);\ri u)
    =
    \mathsf{h}_{\mr{r}}(z,-z)
    +
    \lim_{z_1\to z}
    \left(
        \mathsf{h}_{\mr{p}}(z_1,-z)
        +
        \frac{1}{u^2 a (a+1)}
        \frac{z_1z}{z_1-z}
    \right) \,.
\end{equation}
The regular contribution simplifies to
\begin{equation}
    \label{eq: h regular}
    \mathsf{h}_{\mr{r}}(z,-z)
    =
    \frac{z}{u a(a+1)} \cs(uaz)
    \sum_{k=0}^{-z-1}
        \frac{k}{\cs(uz+uk) \cs(u(a+1)z+uk)} \,,
\end{equation}
where we used the identities $\frac{\cs(x\delta)}{\cs(y\delta)} \to \frac{x}{y}$ as $\delta \to 0$ and
\begin{equation}
    \label{eq: gamma op App B}
    \gamfct{-z,-az,(a+1)z}{u} \,
    \gamfct{z,az,-(a+1)z}{u}
    =
    \frac{1}{\cS(uz)\cS(uaz)\cS(u(a+1)z)} \,,
\end{equation}
which is an immediate consequence of \cref{lem: gamma fct reflect}.

The polar contribution is slightly more involved. We first notice that
\begin{equation}
\label{eq: h Hop}
    \lim_{z_1\to z}
    \left(
        \mathsf h_{\mr p}(z_1,-z)
        +
        \frac{1}{u^2a(a+1)} \frac{z_1z}{z_1-z}
    \right)
    =
    \frac{z}{u^2a(a+1)}
    \frac{\mr d}{\mr d z_1} \Bigg|_{z_1=z}
    F(z_1,-z) \,,
\end{equation}
where
\begin{multline}
    F(z_1,z_2)
    \coloneqq
    \left( \prod_{j=1}^2 \gamfct{-z_j,-az_j,(a+1)z_j}{u} \right)
        \frac{
            \cS(uaz_1)
        }{
            \cs\bigl(ua(z_1+z_2)\bigr)
        } \\
        \times
        \frac{
            \cS(uz_1) \cS\bigl(u(a+1)z_2\bigr)
        }{
            \cS\bigl(u(z_1+z_2)\bigr)
        }
        \frac{
            (q^{(a+1)z_1};q)_{z_2}
            (q^{-z_2};q)_{z_2}
        }{
            (q^{z_1};q)_{z_2}
            (q^{-(a+1)z_2};q)_{z_2}
        }
        \cs\bigl(uaz_2(z_1+z_2)\bigr) \,.
\end{multline}
Indeed, a direct calculation shows that $\mathsf h_{\mr p}(z_1,-z) = \frac{1}{u^2a(a+1)} \frac{z_1}{z_1-z} F(z_1,-z)$, while $F(z_1,-z)$ tends to $-z$ as $z_1 \to z$. Thus, \cref{eq: h Hop} follows from L'H{\^o}pital's rule.

The derivative in \eqref{eq: h Hop} can be evaluated using \cref{eq: Psi sinh to coth} and
\begin{equation}
    \frac{\mr d}{\mr d z}\gamfct{z}{u}
    =
    \gamfct{z}{u} \,
    u^2 \,
    \Psi\left(
        \frac{
            2\mathqoppa\bigl(u(1-2z)\bigr)
        }{
            u^2 \cS(u)
        }
        -
        \frac{1}{u^2}
    \right) \,,
\end{equation}
where $\mathqoppa(x) \coloneqq \frac{1}{2}\cosh\frac{x}{2}$. After simplification, this gives
\begin{multline}
    \label{eq: h pole antidiagonal}
    \lim_{z_1\to z}
    \left(
        \mathsf h_{\mr p}(z_1,-z)
        +
        \frac{1}{u^2a(a+1)}
        \frac{z_1z}{z_1-z}
    \right)
    \\
    =
    \frac{2z^2}{a(a+1)}
    \Psi\left(
        \frac{
            \mathqoppa\bigl(u(1-2z)\bigr)
            +a\mathqoppa\bigl(u(1-2az)\bigr)
            -(a+1)\mathqoppa\bigl(u(1-2(a+1)z)\bigr)
        }{
            u^2 \cS(u)
        }
    \right)
    \\
    +
    \frac{z^2}{ua(a+1)}
    \sum_{k=0}^{-z-1}
    \Bigl(
        -(a+1)\kappa\bigl(u((a+1)z+k)\bigr)
        +
        \kappa\bigl(u(z+k)\bigr)
    \Bigr) \,.
\end{multline}
Combining this last identity with \eqref{eq: 3H n2 zz to reg sing} and \eqref{eq: h regular}, we see that in order to prove \eqref{eq: 3H n2 zz} it suffices to establish the following identity:
\begin{multline}
    \label{eq: 3H n2 zz reg final relation}
        \sum_{k=0}^{-z-1}
        \left(
            uz\,
            \kappa\bigl(u(z+k)\bigr)
            -
            uz(a+1)\,
            \kappa\bigl(u((a+1)z+k)\bigr)
            +
            \frac{
                uk\,\cs(uaz)
            }{
                \cs(u(z+k))
                \cs(u((a+1)z+k))
            }
        \right)
        \\
        =
        u^2 \Psi\Bigg(
            \frac{1}{u^2 \cS(u)}
            \Bigg(
                -2z\mathqoppa\bigl(u(1-2z)\bigr)
                -2az\mathqoppa\bigl(u(1-2az)\bigr)
                +2(a+1)z\mathqoppa\bigl(u(1-2(a+1)z)\bigr)
                \\
                +
                \frac{
                    \cs(2uz)+\cs(2auz)-\cs(2(a+1)uz)
                }{
                    2\cs(u)
                }
            \Bigg)
        \Bigg) \,.
\end{multline}
Let us denote the left-hand side of \eqref{eq: 3H n2 zz reg final relation} by $f(z;a)$. We claim that it suffices to establish the following functional equation.

\begin{lem}
    \label{lem: func eq 3H n2 zz reg}
    We have
    \begin{equation}
    \label{eq: func eq 3H n2 zz reg}
        f(z;a+z^{-1})-f(z;a)
        =
        u(a+1)z\,\kappa\bigl(u(a+1)z\bigr)
        -
        uaz\,\kappa(auz) \,.
    \end{equation}
\end{lem}

Before proving the lemma, let us explain why \eqref{eq: func eq 3H n2 zz reg} implies \eqref{eq: 3H n2 zz reg final relation}. A direct calculation shows that the right-hand side of \eqref{eq: 3H n2 zz reg final relation} satisfies the same functional equation. Hence, the difference between the two sides is invariant under the translation $a \mapsto a+z^{-1}$. Both sides are power series in $u$ whose coefficients are rational functions of $a$. Since a rational function with a nonzero period is constant, their difference is independent of $a$. Finally, both sides vanish at $a=0$, and therefore \eqref{eq: 3H n2 zz reg final relation} follows. It remains to prove \cref{lem: func eq 3H n2 zz reg}.

\begin{proof}[Proof of \cref{lem: func eq 3H n2 zz reg}]
    We first observe the elementary identity $\frac{\cs(x-y)}{\cs(x)\cs(y)} = \kappa(y)-\kappa(x)$, which follows immediately from $\kappa(x)=\frac{1}{2}\coth\frac{x}{2}$. Taking $x=u((a+1)z+k)$ and $y=u(z+k)$, we obtain
    \begin{equation}
        \frac{\cs(uaz)}{\cs(u(z+k))\cs(u((a+1)z+k))}
        =
        \kappa\bigl(u(z+k)\bigr)
        -
        \kappa\bigl(u((a+1)z+k)\bigr) \,.
    \end{equation}
    Thus, the summand defining $f(z;a)$ simplifies to
    \begin{multline}
        -(a+1)uz\, \kappa\bigl(u((a+1)z+k)\bigr)
        +
        uz\, \kappa\bigl(u(z+k)\bigr)
        +
        \frac{
            uk\,\cs(uaz)
        }{
            \cs(u(z+k)) \cs(u((a+1)z+k))
        }
        \\
        =
        u(z+k)\, \kappa\bigl(u(z+k)\bigr)
        -
        u((a+1)z+k)\, \kappa\bigl(u((a+1)z+k)\bigr) \,.
    \end{multline}
    Consequently,
    \begin{equation}
        f(z;a)
        =
        \sum_{k=0}^{-z-1}
        \Bigl(
            u(z+k)\,
            \kappa\bigl(u(z+k)\bigr)
            -
            u((a+1)z+k)\,
            \kappa\bigl(u((a+1)z+k)\bigr)
        \Bigr) \,.
    \end{equation}
    The first term is independent of $a$, and therefore
    \begin{multline}
        f(z;a+z^{-1})-f(z;a)
        =
        \sum_{k=0}^{-z-1}
        \Big(
            u((a+1)z+k)\,
            \kappa\bigl(u((a+1)z+k)\bigr)
        \\
            -
            u((a+1)z+k+1)\,
            \kappa\bigl(u((a+1)z+k+1)\bigr)
        \Big) \,.
    \end{multline}
    The sum telescopes. Since the last value of $k$ is $-z-1$, this proves the claim.
\end{proof}

\subsection{Three markings}
\label{sec: 3H n3 proof}
To prove \eqref{eq: 3H n3 z} and \eqref{eq: 3H n3 zz}, we follow a similar strategy and evaluate $\HodgeH(z_1,z_2,1;1,a,-(a+1);\ri u ) = \langle \sfA(z_1,a,u)\, \sfA(z_2,a,u)\, \sfA(1,a,u) \rangle^{\circ}$, initially for $z_1,z_2\in\bZ_{>0}$. Expanding the $\sfA$-operators gives
\begin{equation}
\begin{split}
    &\HodgeH(z_1,z_2,1;1,a,-(a+1);\ri u)
    =
    \frac{
        \prod_{j=1}^2\gamfct{-z_j,-az_j,(a+1)z_j}{u}
    }{
        u^5a^3(a+1)^3z_1z_2
    }
    \frac{
        \cs(uaz_1)\cs(uaz_2)\cs(ua)
    }{
        \cs(u)
    } \\
    &\qquad \times
    \sum_{k_1,k_2,k_3\in\bZ}
    \frac{
        (q^{(a+1)z_1};q)_{k_1}
        (q^{(a+1)z_2};q)_{k_2}
        (q^{a+1};q)_{k_3}
    }{
        (q^{z_1+1};q)_{k_1}
        (q^{z_2+1};q)_{k_2}
        (q^2;q)_{k_3}
    }
    \left\langle
        \cE_{k_1}(uaz_1)\,
        \cE_{k_2}(uaz_2)\,
        \cE_{k_3}(ua)
    \right\rangle^{\circ} \,,
\end{split}
\end{equation}
where we used $\gamfct{-1,-a,a+1}{u}=1/\cS(u)$. Now $1/(q^2;q)_{k_3}=0$ for $k_3 < -1$, while \cref{lem:VEV} implies that the expectation value vanishes for $k_3 \geq 0$. Thus, only $k_3 = -1$ contributes. Using again \cref{lem:VEV}, together with $\cs(ua) (q^{a+1};q)_{-1} = \cs(u) (q^2;q)_{-1}$, and setting $k=k_1$, $k_2=1-k$, we obtain
\begin{equation}
\label{eq: Hodge vs q hypergeom n3}
\begin{split}
    &\HodgeH(z_1,z_2,1;1,a,-(a+1);\ri u)
    =
    \frac{
        \prod_{j=1}^2\gamfct{-z_j,-az_j,(a+1)z_j}{u}
    }{
        u^5a^3(a+1)^3z_1z_2
    }
    \frac{
        \cs(uaz_1)\cs(uaz_2)
        \cs(u(a+1)z_2)
    }{
        \cs(u(z_2+1))
    }
    \\
    &\quad \times
    \sum_{k>0}
    \frac{
        (q^{(a+1)z_1};q)_k
        (q^{-(z_2+1)};q)_k
    }{
        (q^{z_1+1};q)_k
        (q^{-(a+1)z_2};q)_k
    }
    \cs\bigl(uak(z_1+z_2)-uaz_1\bigr) \,.
\end{split}
\end{equation}
For $z_2\in\bZ_{>0}$, the sum on the right-hand side truncates, and the resulting expression extends to arbitrary $z_1$.

\subsubsection{The specialization \texorpdfstring{$z_1=-1/(a+1)$}{z1=-1/(a+1)}.}
For $z_1=-1/(a+1)$, the sum in \eqref{eq: Hodge vs q hypergeom n3} truncates at $k=1$, and using $\gamfct{1/(a+1),a/(a+1),-1}{u}  = \frac{1}{\cS(u)}$ we obtain
\begin{equation}
    \HodgeH\left(-\tfrac{1}{a+1},z,1;1,a,-(a+1);\ri u\right)
    =
    -\frac{
        \gamfct{-z,-az,(a+1)z}{u}
    }{
        u^4 a^3(a+1)^2 z
    }
    \cs(uaz)^2 \,.
\end{equation}
By homogeneity, \cref{eq: homog}, this is equivalent to \eqref{eq: 3H n3 z}.

\subsubsection{The \texorpdfstring{specialization $z_1+z_2=0$}{anti-diagonal specialization}.}
Set $z_2=-z$, initially with $-z\in\bZ_{>0}$. Then the sum in \eqref{eq: Hodge vs q hypergeom n3} truncates at $k=1-z$. Let $f_k(z_1,z_2)$ denote its $k$-th summand. As $z_1\to z$, all terms are regular except those with $k=-z,1-z$, which have simple poles whose
residues cancel. More precisely,
\begin{equation}
\begin{gathered}
    \sum_{k=1}^{-z-1} f_k(z,-z)
    =
    \frac{
        \cs(u(1-z))\cs(uaz)\cs(u(z+1))
    }{
        \cs(u)^2
    } \,, \\
    \lim_{z_1\to z}
    \bigl(
        f_{-z}(z_1,-z)+f_{1-z}(z_1,-z)
    \bigr)
    =
    \frac{
        \cs(u(1-z))\cs(uz)\cs(u(1-az))
    }{
        \cs(u)^2
    } \,.
\end{gathered}
\end{equation}
Using $\cs(uaz)\cs(u(z+1)) + \cs(uz)\cs(u(1-az)) = \cs(u(a+1)z)\cs(u),$ together with \eqref{eq: gamma op App B}, we obtain
\begin{equation}
    \HodgeH\left(z,-z,1;1,a,-(a+1);\ri u\right)
    =
    -\frac{z}{u^2a^2(a+1)^2}
    \frac{
        \cs(uaz)\cs(u(a+1)z)
    }{
        \cs(uz)\cs(u)
    } \,.
\end{equation}
By homogeneity \eqref{eq: homog}, this is equivalent to
\eqref{eq: 3H n3 zz}.

\subsection{Four markings}
We deduce \eqref{eq: 3H n4 z} and \eqref{eq: 3H n4 zz} from the vacuum expectation value
\begin{equation}
    \HodgeH\left(-\tfrac{1}{a+1},z_1,z_2,1;1,a,-(a+1);\ri u\right)
    =
    \left\langle
        \sfA\big(-\tfrac{1}{a+1},a,u\big)\,
        \sfA\big(z_1,a,u\big)\,
        \sfA\big(z_2,a,u\big)\,
        \sfA\big(1,a,u\big)
    \right\rangle^{\circ} \,.
\end{equation}
Expanding the $\sfA$-operators into a sum over $(k_1,k_2,k_3,k_4)$, the same vanishing arguments for the vacuum expectation values of the $\cE$-operators and Pochhammer symbols used in the previous sections show that only terms of the form $(1,k,-k,-1)$, with $k\in\bZ$, contribute. After the same elementary simplifications as in the previous sections, we obtain
\begin{multline}
\label{eq: Hodge vs q hypergeom n4}
    \HodgeH\left(-\tfrac{1}{a+1},z_1,z_2,1;1,a,-(a+1);\ri u\right)
    =
    -
    \frac{
        \prod_{j=1}^2
        \gamfct{-z_j,-az_j,(a+1)z_j}{u}
    }{
        u^6a^4(a+1)^3z_1z_2
    }
    \cs(uaz_1)\cs(uaz_2) \\
    \times
    \sum_{k\in\bZ}
    \frac{
        (q^{(a+1)z_1};q)_k
        (q^{-z_2};q)_k
    }{
        (q^{z_1+1};q)_k
        (q^{1-(a+1)z_2};q)_k
    }
    \left\langle
        \cE_1\big(-u\tfrac{a}{a+1}\big)\,
        \cE_k(uaz_1)\,
        \cE_{-k}(uaz_2)\,
        \cE_{-1}(ua)
    \right\rangle^{\circ} \,.
\end{multline}
For fixed $z_2 \in \bZ_{>0}$, this is an identity of power series in $u$
with coefficients rational in $z_1$.

\subsubsection{The specialization \texorpdfstring{$z_1=1/a$}{z1=1/a}.}
For $z_1=1/a$, \cref{lem:VEV} gives
\begin{equation}
    \left\langle
        \cE_1\bigl(-u\tfrac{a}{a+1}\bigr)\,
        \cE_k(u)\,
        \cE_{-k}(uaz)\,
        \cE_{-1}(ua)
    \right\rangle^{\circ}
    =
    \begin{cases}
        \cs(u)\cs(uaz) & k=0 \,,\\
        \cs\bigl(u(1+az)\bigr)\cs\bigl(uk(1+az)\bigr) & k>0 \,,\\
        0 & k<0 \,.
    \end{cases}
\end{equation}
Moreover, the first Pochhammer ratio in \eqref{eq: Hodge vs q hypergeom n4} cancels, while $\gamfct{-1/a,-1,(a+1)/a}{u} = \frac{1}{\cS(u)}$. We therefore obtain
\begin{multline}
    \label{eq: Hodge vs q hypergeom n4 2}
        \HodgeH\left(
            -\tfrac{1}{a+1},\tfrac{1}{a},z,1;
            1,a,-(a+1);\ri u
        \right)
        =
        -\frac{
            \gamfct{-z,-az,(a+1)z}{u}
        }{
            u^5 a^3(a+1)^3 z
        }
        \cs(uaz) \\
        \times \Bigg(
            \cs(u)\cs(uaz)
            +
            \cs\bigl(u(1+az)\bigr)
            \sum_{k\geq0}
            \frac{
                (q^{-z};q)_k
            }{
                (q^{1-(a+1)z};q)_k
            }
            \cs\bigl(uk(1+az)\bigr)
        \Bigg) \,.
\end{multline}
To evaluate the remaining sum, consider the $q$-hypergeometric series
\begin{equation}
    {}_2\phi_1
    \left(
        \begin{matrix}
            A,B\\
            C
        \end{matrix};
        q;w
    \right)
    \coloneqq
    \sum_{k\geq0}
    \frac{
        (A;q)_k (B;q)_k
    }{
        (C;q)_k (q;q)_k
    }
    w^k \,.
\end{equation}
Since $\cs(x)=\re^{x/2}-\re^{-x/2}$, the sum in \eqref{eq: Hodge vs q hypergeom n4 2} is
\begin{equation}
    {}_2\phi_1
    \left(
        \begin{matrix}
            q,q^{-z}\\
            q^{1-(a+1)z}
        \end{matrix};
        q;q^{(1+az)/2}
    \right)
    -
    {}_2\phi_1
    \left(
        \begin{matrix}
            q,q^{-z}\\
            q^{1-(a+1)z}
        \end{matrix};
        q;q^{-(1+az)/2}
    \right) \,.
\end{equation}
We evaluate these two terms using the $q$-Chu--Vandermonde identity \cite[Eq.~(1.5.2) \& (1.5.3)]{GR04:BasicHypergeomSeries}, which in our symmetrized convention reads, for $N \in \bZ_{>0}$,
\begin{equation}
\label{eq: Chu Vandermonde}
    {}_2\phi_1
    \left(
        \begin{matrix}
            A,q^{-N}\\
            C
        \end{matrix};
        q;
        \left(\frac{q^{N-1}C}{A}\right)^{\pm1/2}
    \right)
    =
    A^{\mp N/2}
    \frac{(C/A;q)_N}{(C;q)_N} \,.
\end{equation}
Applying \eqref{eq: Chu Vandermonde} with $A=q$, $C=q^{1-(a+1)z}$, and $N=z$, we obtain
\begin{multline}
    \HodgeH\left(
        -\tfrac{1}{a+1},\tfrac{1}{a},z,1;
        1,a,-(a+1);\ri u
    \right) \\
    =
    -
    \frac{
        \gamfct{-z,-az,(a+1)z}{u}
    }{
        u^5 a^3(a+1)^3 z
    }
    \Bigl(
        \cs(u)\cs(uaz)^2
        +
        \cs(uz)
        \cs\bigl(u(1+az)\bigr)
        \cs\bigl(u(a+1)z\bigr)
    \Bigr) \,.
\end{multline}
By the homogeneity relation \eqref{eq: homog}, this is equivalent to
\eqref{eq: 3H n4 z}.

\subsubsection{The \texorpdfstring{specialization $z_1+z_2=0$}{anti-diagonal specialization}.}
For general $z_1$ and $z_2$, the vacuum expectation value of the operators in \eqref{eq: Hodge vs q hypergeom n4} is rather complicated. To take the limit $z_1 \to -z_2$, however, it suffices to observe that
\begin{equation}
    \left\langle
        \cE_1\bigl(-u\tfrac{a}{a+1}\bigr)\,
        \cE_k(uaz_1)\,
        \cE_{-k}(uaz_2)\,
        \cE_{-1}(ua)
    \right\rangle^{\circ}
    =
    \begin{cases}
        \cs(uaz_1)\cs(uaz_2) & k=0 \,,\\
        O\bigl((z_1+z_2)^2\bigr) & k>0 \,,\\
        0 & k<0 \,.
    \end{cases}
\end{equation}
Since the Pochhammer ratio in \eqref{eq: Hodge vs q hypergeom n4} has at most a simple pole as $z_1\to-z_2$, only the term $k=0$ contributes to the limit. Setting $z_2=-z$ and using \eqref{eq: gamma op App B}, we obtain
\begin{equation}
    \HodgeH\left(
        -\tfrac{1}{a+1},z,-z,1;
        1,a,-(a+1);\ri u
    \right)
    =
    \frac{z}{u^3a^3(a+1)^2}
    \frac{
        \cs(uaz)^3
    }{
        \cs(uz)\cs\bigl(u(a+1)z\bigr)
    } \,.
\end{equation}
By the homogeneity relation \eqref{eq: homog}, this is equivalent to
\eqref{eq: 3H n4 zz}.

\subsection{Five markings}
To prove \eqref{eq: 3H n5 zz}, we consider the vacuum expectation value
\begin{multline}
    \HodgeH\left(
        -\tfrac{1}{a+1},z_1,z_2,z_3,1;
        1,a,-(a+1);\ri u
    \right) \\
    =
    \left\langle
        \sfA\bigl(-\tfrac{1}{a+1},a,u\bigr)\,
        \sfA(z_1,a,u)\,
        \sfA(z_2,a,u)\,
        \sfA(z_3,a,u)\,
        \sfA(1,a,u)
    \right\rangle^{\circ} \,.
\end{multline}
The same vanishing arguments and elementary simplifications used in the previous sections show that only terms of the form $(1,k_1,k_2-k_1,-k_2,-1)$ contribute. We therefore obtain
\begin{equation}
    \label{eq: Hodge vs q hypergeom n5}
    \begin{split}
        &\HodgeH\left(
            -\tfrac{1}{a+1},z_1,z_2,z_3,1;
            1,a,-(a+1);\ri u
        \right)
        =
        -\frac{
            \prod_{j=1}^3
            \gamfct{-z_j,-az_j,(a+1)z_j}{u}
        }{
            u^8a^5(a+1)^4z_1z_2z_3
        } \\
        &\qquad\qquad\qquad\qquad\quad \times
        \prod_{j=1}^3\cs(uaz_j)
        \sum_{k_1,k_2\in\bZ}
        \frac{
            (q^{(a+1)z_1};q)_{k_1}
            (q^{(a+1)z_2};q)_{k_2-k_1}
            (q^{(a+1)z_3};q)_{-k_2}
        }{
            (q^{z_1+1};q)_{k_1}
            (q^{z_2+1};q)_{k_2-k_1}
            (q^{z_3+1};q)_{-k_2}
        } \\
        &\qquad\qquad\qquad\qquad\qquad\quad \times
        \left\langle
            \cE_1\bigl(-u\tfrac{a}{a+1}\bigr)\,
            \cE_{k_1}(uaz_1)\,
            \cE_{k_2-k_1}(uaz_2)\,
            \cE_{-k_2}(uaz_3)\,
            \cE_{-1}(ua)
        \right\rangle^{\circ} \,.
    \end{split}
\end{equation}
For fixed $z_2,z_3\in\bZ_{>0}$, \cref{lem:VEV} and the vanishing of the Pochhammer symbols imply that the sum is supported on $0\leq k_1\leq z_2+z_3$ and $0\leq k_2\leq z_3$. Hence, \eqref{eq: Hodge vs q hypergeom n5} is an identity of power series in $u$ with coefficients rational in $z_1$, and we may specialize to $z_1=1/a$.

At this value, the first Pochhammer ratio cancels and $\gamfct{-1/a,-1,(a+1)/a}{u}=1/\cS(u)$. We thus obtain
\begin{equation}
    \label{eq: Hodge vs q hypergeom n5 2}
    \begin{split}
        &\HodgeH\left(
            -\tfrac{1}{a+1},\tfrac{1}{a},z_2,z_3,1;
            1,a,-(a+1);\ri u
        \right)
        =
        -\frac{
            \prod_{j=2}^3
            \gamfct{-z_j,-az_j,(a+1)z_j}{u}
        }{
            u^7a^4(a+1)^4z_2z_3
        }
        \\
        &\qquad\qquad\qquad\qquad\qquad \times
        \cs(uaz_2)\cs(uaz_3)
        \sum_{k_1,k_2\geq0}
        \frac{
            (q^{(a+1)z_2};q)_{k_2-k_1}
            (q^{-z_3};q)_{k_2}
        }{
            (q^{z_2+1};q)_{k_2-k_1}
            (q^{1-(a+1)z_3};q)_{k_2}
        } \\
        &\qquad\qquad\qquad\qquad\qquad\quad \times
        \left\langle
            \cE_1\bigl(-u\tfrac{a}{a+1}\bigr)\,
            \cE_{k_1}(u)\,
            \cE_{k_2-k_1}(uaz_2)\,
            \cE_{-k_2}(uaz_3)\,
            \cE_{-1}(ua)
        \right\rangle^{\circ} \,.
    \end{split}
\end{equation}
Evaluating the last vacuum expectation value using \cref{lem:VEV}, we can write it as
\begin{multline}
        \left\langle
            \cE_1\bigl(-u\tfrac{a}{a+1}\bigr)\,
            \cE_{k_1}(u)\,
            \cE_{k_2-k_1}(auz_2)\,
            \cE_{-k_2}(auz_3)\,
            \cE_{-1}(au)
        \right\rangle^{\circ} \\
        =
        A\,\delta_{k_1>0}\delta_{k_2\geq0}
        +
        B\,\delta_{k_1>0}\delta_{k_2>k_1}
        +
        C\,\delta_{k_1,0}\delta_{k_2>0}
        +
        D\,\delta_{k_1>0}\delta_{k_2,0}
        +
        E\,\delta_{k_1>0}\delta_{k_2,k_1}
        +
        F\,\delta_{k_1,0}\delta_{k_2,0} \,,
\end{multline}
where
\begin{equation}
\begin{aligned}
    A
    &=
    \cs\bigl(u(1+a(z_2+z_3))\bigr)\,
    \cs\bigl(uk_2(1+a(z_2+z_3))\bigr)\,
    \cs\bigl(u(k_1-k_2+k_1(1+az_2))\bigr) \,, \\
    B
    &=
    \cs\bigl(u(1+a(z_2+z_3))\bigr)\,
    \cs\bigl(u(k_2-k_1)(1+a(z_2+z_3))\bigr)\,
    \cs\bigl(u(k_1az_3 + k_2)\bigr) \,, \\
    C
    &=
    \cs(u)\,
    \cs\bigl(ua(z_2+z_3)\bigr)\,
    \cs\bigl(uak_2(z_2+z_3)\bigr) \,, \\
    D
    &=
    \cs(uaz_3)\,
    \cs\bigl(u(1+az_2)\bigr)\,
    \cs\bigl(uk_1(1+az_2)\bigr) \,, \\
    E
    &=
    \cs(uaz_2)\,
    \cs\bigl(u(1+az_3)\bigr)\,
    \cs\bigl(uk_1(1+az_3)\bigr) \,, \\
    F
    &=
    \cs(u)\,
    \cs(uaz_2)\,
    \cs(uaz_3) \,.
\end{aligned}
\end{equation}
We denote the contribution of each of these terms to the sum in
\eqref{eq: Hodge vs q hypergeom n5 2} by the corresponding calligraphic letter. For instance,
\begin{equation}
    \cA
    \coloneqq
    \sum_{k_1>0, \ k_2\geq0}
    \frac{
        (q^{(a+1)z_2};q)_{k_2-k_1}
        (q^{-z_3};q)_{k_2}
    }{
        (q^{z_2+1};q)_{k_2-k_1}
        (q^{1-(a+1)z_3};q)_{k_2}
    } \,
    A \,.
\end{equation}
We now evaluate the six contributions $\cA,\ldots,\cF$ and determine their limits as $z_2\to -z_3$. The last one is immediate: $\cF = F = \cs(u)\cs(uaz_2)\cs(uaz_3)$. The remaining terms require a little more work.

\subsubsection{The calculation of \texorpdfstring{$\cA+\cB$}{A+B}.}
We begin with $\cA$. For fixed $k_2$, we use the identity
\begin{equation}
    \frac{(q^{(a+1)z_2};q)_{k_2-k_1}}{(q^{z_2+1};q)_{k_2-k_1}}
    =
    \frac{(q^{(a+1)z_2};q)_{k_2}}{(q^{z_2+1};q)_{k_2}}
    \frac{(q^{-z_2-k_2};q)_{k_1}}{(q^{1-(a+1)z_2-k_2};q)_{k_1}} \,,
\end{equation}
which separates the dependence on $k_1$. The resulting sum over $k_1>0$ can then be written as the difference of two (terminating) ${}_2\phi_1$ series. Applying the $q$-Chu--Vandermonde identity \eqref{eq: Chu Vandermonde} and simplifying the resulting $q$-Pochhammer ratios gives
\begin{multline}
\label{eq: A sum final}
    \cA
    =
    \frac{
        \cs(u(1+a(z_2+z_3)))\,\cs(u(a+1)z_2)
    }{
        \cs(auz_2)
    } \\
    \times
    \sum_{k_2=0}^{z_3}
    \frac{
        (q^{(a+1)z_2};q)_{k_2}\,(q^{-z_3};q)_{k_2}
    }{
        (q^{z_2+1};q)_{k_2}\,(q^{1-(a+1)z_3};q)_{k_2}
    }
    \cs(uk_2(1+a(z_2+z_3)))\,
    \cs(u(z_2+k_2)) \,.
\end{multline}
A similar calculation for $\cB$ gives
\begin{multline}
\label{eq: B sum final}
    \cB
    =
    \frac{
        \cs(u(1+a(z_2+z_3)))\,\cs(u(a+1)z_3)
    }{
        \cs(auz_3)
    } \\
    \times
    \sum_{k_2=1}^{z_3-1}
    \frac{
        (q^{(a+1)z_2};q)_{k_2}\,(q^{-z_3};q)_{k_2}
    }{
        (q^{z_2+1};q)_{k_2}\,(q^{1-(a+1)z_3};q)_{k_2}
    }
    \cs(uk_2(1+a(z_2+z_3)))\,
    \cs(u(z_3-k_2)) \,.
\end{multline}
Here we use crucially that, for $z_3\in\bZ_{>0}$, the sum over $k_2$ terminates at $k_2=z_3-1$ because of the factor $\cs(u(z_3-k_2))$. We now combine the two contributions:
\begin{equation}
\begin{split}
    \cA+\cB
    &=
    \cs\bigl(u(1+a(z_2+z_3))\bigr)
    \sum_{k=1}^{z_3-1}
        \frac{(q^{(a+1)z_2};q)_{k} \, (q^{-z_3};q)_{k}}{(q^{z_2+1};q)_{k} \, (q^{1-(a+1)z_3};q)_{k}}
        \cs\bigl(uk(1+a(z_2+z_3))\bigr)
    \\
    &\qquad\qquad\times
    \left(
        \frac{\cs(u(a+1)z_2)\,\cs(u(z_2+k))}{\cs(uaz_2)}
        +
        \frac{\cs(u(a+1)z_3)\,\cs(u(z_3-k))}{\cs(uaz_3)}
    \right)
    \\
    &\quad
    +
    \frac{
        \cs(u(1+a(z_2+z_3)))\,\cs(u(a+1)z_2)
    }{
        \cs(uaz_2)
    }
    \frac{
        (q^{(a+1)z_2};q)_{z_3}\,(q^{-z_3};q)_{z_3}
    }{
        (q^{z_2+1};q)_{z_3}\,(q^{1-(a+1)z_3};q)_{z_3}
    }
    \\
    &\qquad\qquad\times
    \cs(uz_3(1+a(z_2+z_3)))\,
    \cs(u(z_2+z_3)) \,.
\end{split}
\end{equation}
Ultimately, we only need the limit $z_2\to z$, with $z_3=-z$. For every
$1\leq k\leq -z-1$, the expression in parentheses vanishes in this limit,
while all the remaining factors are regular. Hence all terms in the sum
vanish, and only the final term contributes. Taking the limit yields
\begin{equation}
\label{eq: AB term final}
\begin{split}
    \lim_{z_2\to z} \bigl( \cA+\cB \bigr)
    &=
    -\frac{\cs(u)\,\cs(u(a+1)z)}{\cs(uaz)}
    \frac{
        (q^{(a+1)z};q)_{-z}\,(q^{z};q)_{-z}
    }{
        (q^{z+1};q)_{-z-1}\,(q^{1+(a+1)z};q)_{-z}
    }
    \cs(uz) \\
    &=
    -\frac{
        \cs(u)\,\cs(uz)^2\,\cs(u(a+1)z)^2
    }{
        \cs(uaz)^2
    } \,.
\end{split}
\end{equation}

\subsubsection{The calculation of \texorpdfstring{$\cC$}{C}.}
We next consider
\begin{equation}
    \cC
    =
    \cs(u)\,\cs\bigl(ua(z_2+z_3)\bigr)
    \sum_{k>0}
    \frac{
        (q^{(a+1)z_2};q)_{k}\,(q^{-z_3};q)_{k}
    }{
        (q^{z_2+1};q)_{k}\,(q^{1-(a+1)z_3};q)_{k}
    }
    \cs\bigl(uak(z_2+z_3)\bigr) \,.
\end{equation}
For fixed $z_3\in\bZ_{>0}$, the factor $(q^{-z_3};q)_{k}$ truncates the sum at $k=z_3$. We may therefore study the limit $z_2\to-z_3$ term by term.

For $1\leq k<z_3$, all $q$-Pochhammer ratios are regular at $z_2=-z_3$, while
\begin{equation}
    \cs(ua(z_2+z_3))\,
    \cs(uak(z_2+z_3))
    =
    u^2 a^2 k (z_2+z_3)^2
    +O\bigl((z_2+z_3)^4\bigr) \,.
\end{equation}
Hence these terms vanish in the limit. The only possible singularity occurs for $k=z_3$, where $(q^{z_2+1};q)_{z_3}$ has a simple zero at $z_2=-z_3$. This produces at most a simple pole, which is again killed by the quadratic zero above. We conclude that, as $z_2\to z$, with $z_3=-z$,
\begin{equation}
    \label{eq: C term final}
    \lim_{z_2 \to z} \cC=0 \,.
\end{equation}

\subsubsection{The calculation of \texorpdfstring{$\cD$}{D}.}
The calculation of $\cD$ is similar to that of $\cA$. We first use
\begin{equation}
    \frac{
        (q^{(a+1)z_2};q)_{-k_1}
    }{
        (q^{z_2+1};q)_{-k_1}
    }
    =
    \frac{
        (q^{-z_2};q)_{k_1}
    }{
        (q^{1-(a+1)z_2};q)_{k_1}
    } \,,
\end{equation}
and the resulting sum can again be expressed as the difference of two terminating ${}_2\phi_1$ series. Applying the $q$-Chu--Vandermonde identity \eqref{eq: Chu Vandermonde} gives
\begin{equation}
\label{eq: D term final}
    \cD
    =
    \frac{
        \cs(uaz_3)\,\cs(u(1+az_2))\,\cs(uz_2)\,\cs(u(a+1)z_2)
    }{
        \cs(uaz_2)
    } \,.
\end{equation}

\subsubsection{The calculation of \texorpdfstring{$\cE$}{E}.}
The calculation of $\cE$ is completely analogous to that of $\cD$, and gives
\begin{equation}
\label{eq: E term final}
    \cE
    =
    \frac{
        \cs(auz_2)\,\cs(u(1+az_3))\,\cs(uz_3)\,\cs(u(a+1)z_3)
    }{
        \cs(auz_3)
    } \,.
\end{equation}

\subsubsection{Conclusion.}
We finally combine the six contributions. Setting $z_3=-z$ and taking the limit $z_2\to z$, the preceding calculations give
\begin{multline}
    \lim_{z_2\to z}
    \bigl(
        \cA+\cB+\cC+\cD+\cE+\cF
    \bigr)
    =
    -\frac{\cs(u)\,\cs(uz)^2\,\cs(u(a+1)z)^2}{\cs(uaz)^2}
    -\cs(u(1+az))\,\cs(uz)\,\cs(u(a+1)z)
    \\
    -\cs(u(1-az))\,\cs(uz)\,\cs(u(a+1)z)
    -\cs(u)\,\cs(uaz)^2 \,.
\end{multline}
With the prefactor from \cref{eq: Hodge vs q hypergeom n5 2}, we therefore obtain
\begin{equation}
\begin{split}
    &\HodgeH\left(
        -\tfrac{1}{a+1},
        \tfrac{1}{a},
        z,-z,1;
        1,a,-(a+1);
        \ri u
    \right)
    \\
    & =
    -\frac{z}{a^3(a+1)^3u^4}
    \frac{\cs(uaz)}
    {\cs(uz)\,\cs(u(a+1)z)}
    \lim_{z_2\to z}
    \bigl(
        \cA+\cB+\cC+\cD+\cE+\cF
    \bigr)
    \\
    & =
    \frac{z^4}{a^2(a+1)^2}
    \cS(u)\,
    \cS(uz)\,
    \cS(uaz)\,
    \cS(u(a+1)z)
    \left[
        \left(
            \kappa(uz)
            +
            \kappa(uaz)
            -
            \kappa(u(a+1)z)
        \right)^2
        -
        \frac{1}{4}
    \right] \,.
\end{split}
\end{equation}
By homogeneity \eqref{eq: homog}, this is precisely \eqref{eq: 3H n5 zz}.

\begingroup
\setlength{\emergencystretch}{.5em}
\renewcommand*{\bibfont}{\footnotesize}
\printbibliography\medskip
\endgroup

\end{document}